\documentclass[dvips,aos,preprint]{imsart}

\RequirePackage[OT1]{fontenc}
\usepackage{amssymb}
\usepackage{amsmath,amsthm}
\usepackage{xr}
\RequirePackage[numbers]{natbib}
\RequirePackage[colorlinks,citecolor=blue,urlcolor=blue]{hyperref}
\RequirePackage{hypernat}
\usepackage{enumitem}
\usepackage{color}
\usepackage{colordvi}
\usepackage{mathrsfs}
\long\def\beginskip#1\endskip{}
\def\endskip{}
\newtheorem{theorem}{Theorem}
\newtheorem{proposition}{Propositon}
\newtheorem{lemma}{Lemma}
\newtheorem{remark}{Remark}
\newtheorem{definition}{Definition}
\newtheorem{corollary}{Corollary}

\newcommand{\diam}{\text{diam}}

\newcommand{\Var}{\mathop{\rm Var}\nolimits}

\newcommand{\eps}{\varepsilon}

\newcommand{\RR}{\mathbb{R}}

\renewcommand{\a}{\alpha}
\renewcommand{\epsilon}{\varepsilon}

\newcommand{\1}{\mathbf 1}

\begin{document}

\begin{frontmatter}
\title{Asymptotic frequentist coverage properties of Bayesian credible sets for sieve priors}
\runtitle{frequentist coverage of credible sets}

\begin{aug}
\author{\fnms{Judith} \snm{Rousseau}\thanksref{t2,m1,m2}\ead[label=e1]{rousseau@ceremade.dauphine.fr}},
\author{\fnms{Botond} \snm{Szabo}\thanksref{t1,m3,m4}\ead[label=e2]{b.t.szabo@math.leidenuniv.nl}}

\thankstext{t2}{The project was partially supported by the ANR IPANEMA, the labex ECODEC.}
\thankstext{t1}{The research leading to these results has received funding from the European Research Council under ERC
Grant Agreement 320637. Research partially supported by Netherlands Organization for Scientific Research NWO.}

\runauthor{Rousseau and Szabo}

\affiliation{ Oxford University \thanksmark{m1}, and Universit\'e Paris Dauphine \thanksmark{m2}, and Budapest University of Technology\thanksmark{m3}, and Leiden University\thanksmark{m4}}

\address{Statistics department, Oxford University\\
24-29 St Giles', Oxford OX1 3LB, UK\\
\printead{e1}
}
\address{Leiden University,\\
Mathematical Institute,\\
Niels Bohrweg 1,
Leiden, 2333 CA,\\
The Netherlands\\
\printead{e2}
}
\end{aug}

\begin{abstract}
We investigate the frequentist coverage properties of Bayesian credible sets in a general, adaptive, nonparametric framework. It is well known that the construction of adaptive and honest confidence sets is not possible in general.  To overcome this problem we introduce an extra assumption on the functional parameters, the so called ``general polished tail'' condition. We then show that under standard assumptions both the hierarchical and empirical Bayes methods results in honest confidence sets for sieve type of priors in general settings and we characterize their size. We apply the derived abstract results to various examples, including the nonparametric regression model, density estimation using exponential families of priors, density estimation using histogram priors and nonparametric classification model, for which we show that their size is near minimax adaptive with respect to the considered specific semi-metrics. 
\end{abstract}



\end{frontmatter}

\section{Introduction}
Uncertainty quantification is of key importance in statistical sciences. Estimators without proper uncertainty quantification have only limited practical applicability, since they contain only limited amount of information about their accuracy. In statistics uncertainty about an estimator is described with the help of confidence sets. Confidence statements are then widely used in statistical practice for instance in hypothesis testing. The construction of confidence sets can be, however, very challenging, especially in complex, nonparametric problems.

A very popular aspect of the Bayesian approach is the built-in, straightforward way of quantifying uncertainty, in particular  with the help of credible sets, i.e. sets with prescribed (typically $95\%$) posterior probability. By accumulating large fraction of the posterior mass these sets describe the remaining uncertainty of the Bayesian procedure. Due to the existing computational machinery of Bayesian techniques (eg. MCMC, ABC,... etc) these sets are widely used in practice for uncertainty quantification. However, only  little is known about their theoretical properties. In parametric models following the celebrated Bernstein-von Mises theorem, credible sets are asymptotically confidence sets as well, laying the base of the practical applicability of the Bayesian approach in simple models, see for instance \cite{vaart_1998}.

However, in nonparametric and high-dimensional models the question is still unanswered about how much we can trust Bayesian credible sets as a measure of confidence in the statistical procedure from a frequentist perspective. The first results in the nonparametric paradigm were discouraging, showing that the Bernstein-von Mises theorem does not hold in general, i.e. even in the standard Gaussian white noise model using conjugate Gaussian priors the resulting credible sets have frequentist coverage tending to zero, see \cite{cox:1993,freedman:99}.

Since then the investigation of frequentist coverage properties of Bayesian credible sets has attracted a lot of attention in nonparametric problems.  Various approaches were proposed to solve this problem. In \cite{knapik,yoo:2016} the authors verified that by slightly undersmoothing the prior one can still achieve credible sets with good frequentist coverage and minimax size in the same setup as \cite{cox:1993}. Another possibility is to consider weaker, negative Sobolev-norms and derive the Bernstein-von-Mises theorem in the corresponding Sobolev space, see \cite{leahu:11,castillo:nickl:2013,castillo:nickl:2014}.

The preceding results are all based on the knowledge of the regularity of the true underlying function, which is in practice generally not available. A more challenging problem is the construction of Bayesian based confidence sets in the adaptive setting where no information is available on the smoothness of the truth. This, however, turns out to be too much to ask for. 
In \cite{Low:97, cai:low:04,robins:2006,bull:2013} it was shown that it is impossible to construct adaptive confidence sets in general. 

More precisely assume that the true (functional) parameter $\theta_0$ belongs to some regularity or sparsity class $\Theta^{\beta}$, indexed by some (unknown) hyper-parameter $\beta$ belonging to some index set $B$.  When $\beta$ is unknown, the confidence set  $\hat{C}$ cannot depend on it and it is said to be optimal adaptive if  first it has uniform coverage: 
\begin{align}\label{coverage}
\liminf_n \inf_{\theta_0\in\cup_{\beta\in B}\Theta^{\beta}}P_{\theta_0}^{(n)}(\theta_0\in \hat{C})\geq 1-\alpha
\end{align}
and second its size is optimal within each parameter class  $\Theta^{\beta}$, i.e. for some universal constant $K>0$
\begin{align} \label{size}
\liminf_n \inf_{\beta\in B}\inf_{\theta_0\in\Theta^{\beta}}P_{\theta_0}^{(n)}(\sup_{\theta_1,\theta_2\in\hat{C}}d(\theta_1,\theta_2)\leq K r_{n,\beta})\geq 1-\alpha,
\end{align}
where $r_{n,\beta}$ is the minimax estimation rate within the class $\Theta^{\beta}$ and with respect to the semi-metric $d(.,.)$. 

As mentioned earlier it is impossible to satisfy both the coverage and the minimax size requirements on the confidence sets in general. To solve this problem additional assumptions were introduced on the parameter value $\theta_0$ making the construction of adaptive confidence sets possible by discarding certain inconvenient parameters $\theta_0$. A frequently applied assumption is self-similarity where it is assumed that the true parameter has similar ``local'' and ``global'' behaviour, see for instance \cite{picard:tribouley:2000,gine:nickl:2010,bull:2012,chernozhukov:2014,nickl:szabo:2014,sohl:trabs:2014}. Another approach is to discard the parameters which make it impossible to test between the classes $\Theta^{\beta}$. This approach was considered in various models in context of regularity classes in \cite{Hoffmann:Nickl:11,bull:2013,carpentier:2013} and in sparse high dimensional models \cite{nickl:geer:2013,carpentier:nickl:2015}.

It is a known fact that Bayesian credible balls associated to posterior distributions which concentrate at the minimax rate verify \eqref{size}, see \cite{hoffman:rousseau:hieber}. The question is then to understand their frequentist coverage and in particular to characterize subsets of $\cup_\beta \Theta^\beta$ over which \eqref{coverage} is verified as well. 

 In \cite{szabo:vdv:vzanten:13} the authors have generalized the self-similarity assumption  introducing the so called polished tail assumption, discussed in this article also in more detail. The polished tail (and the stronger self-similarity) assumption was then applied in nonparametric regression with rescaled Brownian motion prior \cite{sniekers:2015} and spline priors \cite{serra:2014, yoo:vaart:2017} and in the context of Gaussian white noise model with Gaussian priors constructing $L_{\infty}$-credible sets \cite{szabo:linfunc:2015}. Furthermore, an adaptive version of the nonparametric Bernstein-von Mises theorem was given in context of the Gaussian white noise model using conjugate Gaussian priors and spike-and-slab prior \cite{ray2017} under the self-similarity assumption. The polished tail assumption was then slightly extended by the implicit excessive bias assumption introduced in context of the Gaussian white noise model \cite{belitser:2014} and applied in sparse high dimensional models with empirical Bayes spike and slab type of priors \cite{belitser:nurushev:2015, castillo:szabo:2018} and with hierarchical and empirical horseshoe prior \cite{pas:etal::2016}. 

All the above mentioned papers consider specific choices of the model and the prior distribution and use explicit, conjugate computations which obviously have their limitations. Although these papers already shed lights on certain aspects of Bayesian uncertainty quantification, they do not provide a clear understanding of the underlying general phenomena. A general approach for understanding the coverage of credible sets is still missing. Besides for many nonparametric models and priors no conjugate computation is possible and therefore they can not be handled directly. In this work we aim to (partially) fill this gap and contribute to the fundamental understanding of this rapidly growing field. We derive abstract results for general choices of models and sieve type of priors, in the spirit of \cite{ghosal:ghosh:vdv:00,ghosal:vdv:07, rousseau:szabo:2015:main}. 

\subsection{Setup and Notations} \label{sec:notation}

We consider observations $\mathbf Y \in \mathcal Y$ distributed from $P_\theta^{(n)}$, $\theta \in \bar\Theta$, which are absolutely continuous with respect to a given measure $\mu$ with density $p_{\theta}^{(n)}$ and $n$ denotes the sample size or signal-to-noise ratio. We denote by $\ell_n(\theta) = \log p_{\theta}^{(n)}$ the log-likelihood and throughout the paper $\theta_0$ designates the true value of the parameter. We denote by $E_{\theta}^{(n)}$ and $V_{\theta}^{(n)}$ the expectation and the variance with respect to $P_\theta^{(n)}$, respectively. For two positive sequences $a_n$ and $b_n$ we write $a_n\lesssim b_n$ if there exists a constant $C>0$ such that $a_n\leq C b_n$ for every $n\in\mathbb{N}$, furthermore, we denote by $a_n\asymp b_n$ that $a_n\lesssim b_n$ and $b_n\lesssim a_n$ hold simultaneously. For $A,B\in\mathbb{R}^{n\times n}$ the inequality $A\leq B$ denotes that $A-B$ is positive semi-definite.

Let us consider a collection of finite dimensional models $\Theta(k)$
\begin{align}
\Theta=\cup_{k\in \mathbb{N}} \Theta(k),  \quad \Theta(k) \subset \mathbb R^{d_k} , \quad d_k \uparrow \infty,\quad k\in \mathbb{N},\label{def: model}
 \end{align}
 with $d_k \asymp k$. We asume that $\theta_0\in\bar{\Theta}$ with $\Theta \subset \bar{\Theta}$,  note that we do not necessarily assume that $\theta_0$ belongs to any of the models $\Theta(k)$, $k\in\mathbb{N}$, hence we allow for misspecification. These models are very popular and frequently used in practice, see for instance \cite{tsybakov,gine:nickl:2016} for a review. 

The parameter $k$ drives the sparsity or the regularity of the model. 
Finding the model $\Theta(k)$, which is the most appropriate for recovering $\theta_0$, requires additional information about the true parameter (e.g. regularity, sparsity,... etc) which is usually not available. Therefore a natural approach is to let the data decide about the optimal model $\Theta(k)$. In the Bayesian framework one can accomplish this by the hierarchical or the empirical Bayes approach.
In the hierarchical (or also referred to as full) Bayes approach one endows the hyper-parameter $k$ with a  prior distribution $\pi_{k}$ and conditionally on $k$, considers a prior distribution  $\pi_{|k}$ on $\theta \in \Theta(k)$, resulting in a two-level prior distribution $\pi$ on $\Theta $ defined by:
\begin{equation} \label{prior}
 k \sim \pi_{k} , \quad \theta | k \sim \pi_{|k} .
\end{equation}
 We denote the posterior distribution on $\Theta$ by $\pi( \theta | \mathbf{Y})$ and the conditional distribution of $\theta | (\mathbf Y, k)$ by $\pi_{|k}(\theta | \mathbf{Y})$.

In contrast to this in the empirical Bayes approach one constructs a frequentist estimator $\hat{k}_n$ for the hyperparameter $k$ and plugs it in into the conditional posterior distribution given $k$, i.e.
\begin{align*}
\pi_{|\hat{k}_n}(\theta| \mathbf{Y} )=\pi_{|k}(\theta | \mathbf{Y})\Big|_{k=\hat{k}_n},
\end{align*}
which is the empirical Bayes posterior distribution.

Models in the form of \eqref{def: model} are widely used in the Bayesian literature and under nonrestrictive assumptions the posterior distribution can optimally recover the true parameter $\theta_0$. In more details, it is common to assume that the true parameter belongs to some regularity class $\theta_0\in \Theta^{\beta}$ with some unknown regularity hyper-parameter $\beta$. Then it was shown for instance in \cite{arbel} and references therein that the hierarchical Bayes approach described above achieves optimal minimax contraction rate around the truth, in a large collection of cases,  without using any additional information about its unknown regularity, leading to an adaptive procedure, in the frequentist sense. In this article our focus is on the quality of Bayesian uncertainty quantification done via credible balls from a frequentist perspective. There are two main properties of interest  in a credible set from a frequentist perspective: the frequentist coverage and the expectation of its size under $P_{\theta_0}^{(n)}$, when $\theta_0$ is assumed to be the true value of the parameter. 
 In the literature the frequentist coverage properties of Bayesian credible sets constructed from sieve posteriors were only investigated for specific choice of priors and likelihoods, see for instance \cite{serra:2014,belitser:2014,yoo:2016}. In this article we present a general approach under which we can simultaneously investigate the frequentist properties of credible sets resulting from different choices of sieve priors and likelihoods.

We introduce some additional notations. Let  $B_k(\theta, u,d)$ denote the $d$-ball in $\Theta(k)$ with center $\theta$ and radius $u$ and $B_k^c(\theta, u,d)$ the complement of such a ball. Furthermore, let $\diam \big(S,d\big)$ denote the $d$-diameter of the set $S$, i.e. 
$$\diam \big(S,d\big)=\sup_{\theta,\theta'\in S} d(\theta,\theta').$$

We define the square distance of the truth from the set $\Theta(k)$ as
 $$b(k) = \inf\{ d^2(\theta_0, \theta) : \theta \in \Theta(k) \} .$$
  For simplicity we also extend the definition of the function $b$ on $[0, +\infty)$  by  $b(x) = b(k)$    for all $x \in [k, k+1)$ and $b(0) = +\infty$. Note that we allow $d(.,.)$ to depend on $n$, so that in this case $b(k)$ also depends on $n$. This will be the case in particular in the regression and in the classification examples, see Sections \ref{sec:reg} and \ref{sec:classif}, respectively. The normalized Kullback-Leibler divergence and variance of the log-likelihood-ratio are denoted by
$$KL(\theta_0, \theta) =\frac{1 }{ n }  E_{\theta_0}^{(n)}\left( \log\left( \frac{p_{\theta_0}^{(n)}}{p_{\theta}^{(n)} } \right)\right),\quad V( \theta_0, \theta) =\frac{1 }{ n }  V_{\theta_0}^{(n)}\left( \log\left( \frac{p_{\theta_0}^{(n)}}{p_{\theta}^{(n)} } \right)\right),$$
respectively. 
We denote by $N(\eps,A,d)$ the entropy, i.e. the number of $\eps$-radius $d$-balls needed to cover the set $A$. Throughout the paper, $c$ and $C$ denote global constants whose value may change one line to another.

  \section{Main results} \label{sec:newtrunc}
In this section we investigate the frequentist properties of Bayesian credible sets resulting from the hierarchical and the empirical Bayes procedures. We consider the general setting described in Section \ref{sec:notation} and introduce general, abstract conditions under which credible sets have honest frequentist coverage and rate adaptive size. The derived results will be applied in  Section \ref{sec:applis} for various choices of sampling and prior models.

Using the posterior distribution $\pi(\theta | \mathbf Y)$, be it hierarchical or empirical, we construct the Bayesian credible sets as balls centered around some estimator $\hat\theta_n$ (typically the posterior mean, mode or median) $\hat C(\alpha) = \{\theta:\,d(\theta,\hat\theta_n)\leq r_\a \}$ where $\a \in (0,1)$ and $r_\a$ is the radius of the ball and satisfies
\begin{align}
r_\a = \arg\inf_{r>0}\{\pi(\theta:\,d(\theta,\hat\theta_n)\leq r | \mathbf Y )\geq  1-\a\}. \label{def: HBcred}
\end{align}
 In our analysis we also introduce some additional flexibility to the credible sets by allowing them to be blown up by a factor $L>0$ resulting in 
\begin{align*}
\hat{C}(L,\a)=\{\theta:\, d(\theta,\hat\theta_n)\leq L r_\a\}.
\end{align*}
We show that these inflated sets (for sufficiently large blow up factor $L$) will have frequentist coverage tending to one and at the same time their size almost optimal in a minimax sense under some additional restrictions on the parameter $\theta_0$.

In the Gaussian white noise model with Gaussian prior, \cite{szabo:vdv:vzanten:13} shows that  a key idea to obtain good coverage is that a trade-off between bias and variance is realized, so that the \textit{correct } value of $k$ (or set of values) is selected either under the posterior $\pi_k(k | \mathbf Y) $ or the empirical estimator $\hat{k}_n$. 

To generalize this idea in non Gaussian setups, let us consider the index set $\mathcal{K}\subseteq\mathbb{N}$ and define for each $\theta_0 \in \bar\Theta$,
 \begin{equation}\label{kn-epsk}
\epsilon_n^{2} (k ) = b(k) + \frac{ k \log n}{ n } , \quad \mbox{and } \quad  k_n = \inf\{ k\in\mathcal{K} :\, b(k) \leq k (\log n)/n \},
 \end{equation}
 and 
$\mathcal K_n (M) = \{ k\in\mathcal{K}:\, \epsilon_n(k) \leq M \epsilon_{n}(k_n) \}$. Note that in these notations $\theta_0$ is implicit since  $b(k)$ depends on $\theta_0$.

To control the frequentist coverage of $\hat C(L, \a)$, we need to restrict ourselves to a subset of $\bar\Theta$, in a manner similar to \cite{szabo:vdv:vzanten:13}, generalizing their idea outside the white noise model with empirical Gaussian process  prior. We introduce below the general polished tail condition which determines the subclass of functions for which frequentist coverage can be obtained. 

\begin{definition}
Let $\theta \in \bar\Theta$, we say that $\theta$ (or equivalently  its associated bias function $b(.) $) satisfies the general polished tail condition associated to the semi-metric $d(.,.)$ if there exist 
integers  $k_0, R_0> 1$ and a real $0< \tau <1$ such that 
 \begin{equation}\label{cond:tail:L}
b(kR_0) \leq \tau b(k), \quad  \forall  k_0  \leq k \leq k_n .
\end{equation}
For given   $k_0, R_0 $ and $\tau$, we denote by $\Theta_{0,n} ( R_0, k_0,\tau)$ the class of $\theta \in \bar\Theta$ satisfying \eqref{cond:tail:L} and 
$\Theta_{0} ( R_0, k_0,\tau) = \bigcap_n \Theta_{0,n} ( R_0, k_0,\tau)$.  

\end{definition}

We note that in the case where $d(.,.)$ is the $\ell_2$-norm, for instance in the Gaussian white noise model, the bias function is $b(k) = \sum_{j=k+1}^\infty \theta_{0,j}^2$. The polished tail condition in \cite{szabo:vdv:vzanten:13} reads as
\begin{align}
\sum_{j=N+1}^{\infty}\theta_{0,j}^2 \leq L \sum_{j=N+1}^{\rho N}\theta_{0,j}^2,\quad \forall N\geq N_0, \label{def: PT:original}
\end{align}
for some $N_0,L,\rho>0$ which is equivalent with our definition  of $\Theta_{0} ( R_0, k_0,\tau)$ (with $k_0=N_0$, $\tau=L/(L+1)$ and $R_0=\rho$). Our new definition, however, extends also to the case where the semi-metric $d(.,.)$ is substantially different from the $\ell_2$-norm.

The generalization of the usual bias and variance trade-off is by obtaining a trade-off between the bias (or more precisely the approximation error) $nb(k)$ and a prior penalization term $k\log n$ induced by the prior mass of small neighbourhoods:
$\pi_{|k}(\theta:\, d( \theta_{[k]}^o, \theta) \leq u_n) $, where  $u_n = o(1)$ and $\theta_{[k]}^o\in\Theta(k)$ can be viewed as a projection of $\theta_0$ on $\Theta(k)$, typically with respect to the semi-metric $d$ or the KL-divergence.
 Then typically if $u_n \asymp n^{-H}$ for some $H>0$, then 
$\log \pi_{|k}(\theta:\, d(  \theta_{[k]}^o, \theta) \leq u_n)  \asymp  - k \log n$, so that the set $\mathcal K_n (M) $ corresponds to values of $k$ for which this trade-off is achieved.

\begin{lemma}\label{rem:Kn}
For any $\theta_0\in\bar\Theta$ and $k\in \mathcal{K}_n(M)$ we have that $k\leq 2M^2k_n$. Furthermore for any $\theta_0\in\Theta_{0,n}(R_0,k_0,\tau)$ let us assume that there exists an $A_0>1$ such that
\begin{align}
\text{for all $k<k_0$ there exists } k'\in\{k_0,k_0+1,...,A_0k_0\},\,  \text{such that }\, b(k)\geq b(k'). \label{eq: cond:bias}
\end{align}
Then for every $k\in \mathcal{K}_n(M)$  we have $k\geq ck_n$, with $c= (R_0^{-m}/2) \wedge (2R_0^{m+1}\vee R_0^mk_0 A_0)$, where $m>0$ is the smallest integer satisfying $\tau^m \leq (8M^2 R_0)^{-1}$.
 \end{lemma}
The proof of the lemma is deferred to Section \ref{sec:rem:Kn} in the supplementary material \cite{rousseau:szabo:16:supp}.

\begin{remark}
Condition \eqref{eq: cond:bias} is very mild. It is easy to see that it holds automatically for nested sets $\Theta(k)$, where the bias function $k\mapsto b(k)$ is monotone non-increasing. Furthermore it can also be verified for models where nestedness occures only on given geometric subsequences $\Theta(k)\subset\Theta(A_0 k)\subset\Theta(A_0^2k)\subset...$, for instance histograms with regular bins, see Section \ref{sec: histogram}.
\end{remark}

We will show in Section \ref{sec:hierarchic}  that in the hierarchical Bayes approach the posterior distribution concentrates on $\mathcal K_n(M)$ for $M$ large enough if the true parameter satisfies the general polished tail condition \eqref{cond:tail:L}. A similar phenomenon occurs for the empirical Bayes method, i.e. the maximum marginal likelihood estimator $\hat k_n$ belongs to the set $\mathcal K_n (M)$ with high probability, see Section \ref{sec:empiric}.

\noindent

In the hierarchical prior case we also consider the following condition on the prior on $k$:

\textbf{H} The prior on $k$ satisfies
 \begin{equation}\label{cond:prior:trunc}
e^{- c_2 k \log (k)}\lesssim \pi_k(k) \lesssim  e^{- c_1 k },   \qquad k\in\mathcal{K},
 \end{equation} 
for some positive constants $c_1,c_2>0$.




In order to bound from below the frequentist coverage of $ \hat{C}(L , \a) $, we restrict ourselves to a subset of parameters $\Theta_0\subseteq \Theta_0(R_0, k_0,\tau)$ for some positive $R_0, k_0, \tau $ on which we consider the following assumptions, used both for the empirical Bayes and for the hierarchical Bayes approaches.
  
  \begin{itemize}
\item [\textbf{A0}] The centering point $\hat \theta_n$ satisfies that for all $\epsilon>0$ there exists $M_\epsilon>0$
 \begin{equation} \label{freq}
  \sup_{\theta_0 \in \Theta_0} P_{\theta_0}^{(n)} \left( d( \theta_0 , \hat \theta_n)\leq  M_\epsilon \epsilon_n(k_n) \right) \geq 1-\eps. 
 \end{equation}

 \item [\textbf{A1}] There exist $c_{3}, c_{4}, C >0$ such that for all $ \theta_0 \in \Theta_0$
 \begin{equation*} \label{KLcond:kn}
 \begin{split}
\pi_{|k_n}\left( KL(\theta_0,\theta)  \leq c_{3} \eps_n^2(k_n), \, V(\theta_0,\theta)\leq  C\eps_n^2(k_n)\right) \geq C^{-1} e^{-  c_{4}  k_n \log n}.
\end{split}
 \end{equation*}

\item [\textbf{A2}$$] Assume that there exist constants $J_0, J_1, c_5>0$, $c_6\in(0,1)$ such that the following conditions hold for all $k\leq \bar{K}_n$ 
\begin{enumerate}[label=(\roman*)]
\item  There exist sets $\Theta_n(k) \subset \Theta(k)$ satisfying
\begin{equation*} \label{sieve} 
\pi_{|k}\left( \Theta_n(k)^c \right) \leq C e^{- (c_2+c_{3}+c_{4}+2) n\eps_n^2(k_n) }.
\end{equation*}
\item There exist measureable  (in $\mathbf Y$) functions $\varphi_n(\theta)\in [0,1]$  such that
\begin{align*}
\sup_{\theta\in  \Theta_n(k)}&E_{\theta_0}^{(n)}\left( \varphi_n(\theta) \right) \leq e^{- c_5 n d^2(\theta_0,\theta) } ,\\
\sup_{\theta\in  \Theta_n(k)}\mathop{\sup_{\theta'\in\Theta_n(k):}}_{d(\theta',\theta)\leq c_6 d(\theta_0, \theta)}
 &E_{\theta'}^{(n)}\left( 1 - \varphi_n(\theta) \right) \leq e^{-c_5 n  d^2(\theta_0,\theta) }.
\end{align*}
\item For all $u \geq \max( J_0 \epsilon_n(k_n), J_1 \sqrt{k(\log n)/n}) $
\begin{equation*}
\log N\big( c_6 u , \Theta_n(k)\cap \{ u \leq d(\theta_0, \theta)\leq 2 u\} , d \big) \leq c_5 n u^2/2.
\end{equation*}
 \end{enumerate}

\item [\textbf{A3}] 
 For all $\gamma >0$, there exists $M_0>0$ such that for all $M_0 k_n \leq k \leq \bar K_n$
$$ \pi_{|k}\big( B_k(\theta_0,J_1 \sqrt{k (\log n)/n} ,d )\cap\Theta_n(k)\big) \leq e^{- (c_2+c_{3}+c_{4}+\gamma) n\eps_n^2(k_n)}.$$

  \item [\textbf{A4}]  Assume that for all $M,\eps>0$ there exist $c_7, c_8, c_9,c_{10},  \delta_0,B_\eps >0$  and  $r \geq 2$ such that the following conditions hold for all $k\in\mathcal{K}_n(M)$
\begin{enumerate}[label=(\roman*)]
\item There exists a parameter $\theta_{[k]}^o\in\Theta(k)$ satisfying
 \begin{equation*}
 \begin{split}
B_k(\theta_{[k]}^o,\sqrt{k/n},d )\cap \Theta_n(k) \subset S_n(k, c_{7}, c_8,r),
\end{split}
 \end{equation*}
where $S_n(k, c,c', r) = \Big\{\theta\in\Theta(k):$
 $$E_{\theta_0}^{(n)} \log\frac{p_{\theta_{[k]}^o}^{(n)}}{p_{\theta}^{(n)}}  \leq c k ,\, E_{\theta_0}^{(n)}\Big(\log\frac{p_{\theta_{[k]}^o}^{(n)}}{p_{\theta}^{(n)}}-E_{\theta_0}^{(n)}\log\frac{p_{\theta_{[k]}^o}^{(n)}}{p_{\theta}^{(n)}}\Big)^r\leq  c' k^{r/2} \Big\}.$$
\item Let $\bar{B}=\Theta_n(k) \cap B_k( \theta_0, (M_\epsilon +1) \epsilon_n(k_n),d)$. Then for every $\theta_0\in \Theta_0$
\begin{equation*}
P_{\theta_0}\Big(\max_{k \in \mathcal K_n(M)}\sup_{\bar{B} }\big(\ell_n(\theta) - \ell_n(\theta_{[k]}^o) -  B_\eps k \big) \leq 0\Big)\geq 1-\eps.
\end{equation*}
\item For every $\delta_{n,k} \leq \delta_0$
\begin{equation*}
\sup_{\theta\in  \bar{B} }  \frac{\pi_{|k}\big( B_k(\theta, \delta_{n,k} \sqrt{k/n} ,d )\cap\Theta_n(k)\big)}{\pi_{|k}\big( B_k(\theta_{[k]}^o, \sqrt{k/n},d )\big)} \leq c_{10} e^{ c_9 k \log (\delta_{n,k}) }.
\end{equation*}
\end{enumerate}

\end{itemize}

\begin{remark} \label{rk:A2}
The parameter $\bar{K}_n$ in assumptions \textbf{A2} and \textbf{A3} is chosen to be $Ak_n\log n$ for some large enough constant $A>0$ for the hierarchical Bayes method. In case of the empirical Bayes method it is the upper bound of the interval where the maximum marginal likelihood estimator is taken, i.e. $\hat{k}_n\in\{1,2,...,\bar{K}_n\}$, see \eqref{def: MMLE}. In this case $\bar{K}_n$ is typically taken to be $n^{H}$ for some $H\in(0,1/2)$.
 \end{remark}

\begin{remark}
One can also handle dimensions $d_k$ not linear in $k$. Then the definition of $\eps_n(k)$ has to be modified to $\eps_n(k)=b(k)+d_k(\log n)/n$ and the conditions in \textbf{A1}-\textbf{A4} have to be given with $k$ replaced by $d_k$.
\end{remark}

A brief explanation of the above conditions is in order. Assumptions \textbf{A1}, \textbf{A2}  are the standard prior small ball probability,
remaining mass, testing and entropy conditions, routinely used in the literature for determining the contraction rates of the posteriors, see for instance \cite{ghosal:vdv:07}.  Assumption \textbf{A3} is commonly considered when upper bounding marginal likelihoods, see for instance \cite{rousseau:07, mcvinishetal:09, rousseau:szabo:2015:main, castillo08}. These conditions are used to describe the behaviour of the posterior distribution on the hyper-parameter $k$ and derive upper bounds for the posterior contraction rates. Proving \textbf{A3} is quite simple in case the semi-metric $d$ is locally equivalent to the $\ell_2$-norm, however, it is quite challenging  in the context of mixture models, where the geometry of the $L_1$ metric is complex.

Assumption \textbf{A4} gathers three conditions, which are required to hold only over the models $k\in\mathcal{K}_n(M)$, this assumption is used to derive lower bounds for the radius of the credible sets. 

Assumption \textbf{A4} (i) requires that locally the (slightly modified) Kullback - Leibler divergence can be bounded by  the distance $d(.,.)$ (up to a multiplicative constant). Note that due to the model misspecification, i.e. typically $\theta_0\notin \Theta(k)$ we consider a projection $\theta_{[k]}^o$ of $\theta_0$ on $\Theta(k)$ for controlling the prior penalization term, see the discussion below \eqref{kn-epsk}. This is a rather mild assumption, the main difficulty here lies in obtaining a sharp upper bound on 
$ KL(\theta_0, \theta)- KL(\theta_0, \theta_{[k]}^o)$ and not only on $ KL(\theta_0, \theta)$.  It can be weakened by considering $c_7$ going to infinity, this would however  induce a bigger inflation of the radius of the credible ball $\hat C_n (L, \alpha)$. 

In assumption \textbf{A4} (ii) the log-likelihood ratio is uniformly controlled in a neighbourhood of $\theta_0$ with high probability. This is not such a stringent condition since the required control is not sharp at all. Indeed it is required that the log-likelihood ratio $\ell_n(\theta) -\ell_n(\theta_{[k]}^o)$ is bounded from above by $O(k)$, but note that under $P_{\theta_0}$ its expectation is equal to $-n(KL(\theta_0, \theta)-KL(\theta_0, \theta_{[k]}^o))\leq 0$, which acts as a pull back force, see the proof of Propositions \ref{prop: hist}, \ref{prop: loglin} and \ref{prop: classification} in the supplementary 
material \cite{rousseau:szabo:16:supp}. 

In condition \textbf{A4} (iii) note that since $\theta_k\in  B_k(\theta_0, (M_\epsilon +1)\epsilon_n(k_n),d)$ and since (typically) $\theta_{[k]}^o\in  B_k(\theta_0, \epsilon_n(k),d)$,  
$d(\theta_k, \theta_{[k]}^o) \leq (M_\epsilon +1)\epsilon_n(k_n)+ \epsilon_n(k) \leq (M+M_\epsilon+1)\epsilon_n(k_n)$ for $k\in \mathcal{K}_n(M)$, so that 
 condition \textbf{A4} (iii) requires that in case the ball around  any point in the vicinity of $\theta_{[k]}^o$  has substantially smaller radius than a $\sqrt{k/n}$ ball centered around $\theta_{[k]}^o$ then the prior mass of the ball is also substantially smaller. This is verified in particular when the distance $d(.,.) $ behaves locally like the Euclidean distance and the prior densities are bounded from below and above, locally. The intuition behind this condition is the following. To achieve high frequentist coverage for the credible set the prior can not put substantially more mass around the centering point than on a small  neighbourhood of the truth. Else the posterior would be even more concentrated around the centering point resulting in overly confident uncertainty statements. Since the centering point is random, but living in a close neighbourhood of the truth we require this condition to hold uniformly over the ball $B_k(\theta_0, (M_\eps+1)\eps_n(k_n),d)$. Conditions \textbf{A4} (ii) and (iii) are the most demanding assumptions, because they require non trivial upper bounds on prior masses of $d$-balls.

Assumption \textbf{A0} is on the centering point and is satisfied typically for usual estimates such as the posterior mean, see the examples in Section \ref{sec:applis}.

In the literature different variations were considered of the standard conditions \textbf{A1} and \textbf{A2}, see for instance \cite{ghosal:ghosh:vdv:00,ghosal:vdv:07}.  Here we consider another version of \textbf{A2} (iii) (based on slicing of the sets $\Theta_n(k)$), which will be applied in the density estimation example with exponential families of priors. 

\begin{itemize}
\item [\textbf{A2}] (iiib) There exist a (possibly infinite) cover $B_{n,j}(k)$ of the set $\Theta_n(k)\cap \{\theta: d(\theta,\theta_0)\geq J_0(k) \eps_n(k_n)\}$ for $k\leq\bar{K}_n$,  such that
\begin{align}
&B_{n,j}(k)  \subset \Theta_n(k) \cap \{ d( \theta, \theta_0) > c(k,j) \epsilon_n(k_n) \}\quad\text{with}\label{cond: A4k'_1}\\
&\sum_{j} \exp\left( -\frac{ c_5}{2} n c(k,j)^2  \epsilon_n(k_n)^2 \right)  \lesssim e^{ -  (c_2+ c_{3} + c_4+2)n\epsilon_{n}^2(k_n) },\label{cond: A4k'_2} 
\end{align}
where $c_2 , c_{3} , c_4$ are defined in assumptions \textbf{H} and  \textbf{A1} and 
\begin{equation}
\log N( c_6 c(k,j) \epsilon_n(k_n), B_{n,j}(k) , d) \leq \frac{c_5  c(k,j)^2 n\epsilon_{n}(k_n)^2}{2}. \label{cond: A4k'_3}
\end{equation}
\end{itemize}

 In the next subsections we show that under the above assumptions together with the general polished tail restriction the credible sets resulting both from the hierarchical and the empirical Bayes procedures have optimal size and high frequentist coverage.
  
\subsection{Hierarchical Bayes approach}  \label{sec:hierarchic}

In this section we present the results for the hierarchical prior defined by \eqref{prior} satisfying assumption \textbf{H}. We show that under the general polished tail condition and the assumptions introduced in the preceding section the inflated credible set  $\hat C(L_n,\a)$ with $L_n\gtrsim \sqrt{\log n}$ has good frequentist properties, i.e. it has good frequentist coverage and we can characterize their size on $\Theta_0=  \Theta_0(R_0, k_0,\tau)$, $R_0>1$, $k_0\geq 1$ and $\tau<1$.

\begin{theorem}\label{th:coverage}
 Assume that conditions \textbf{H}, \textbf{A0}-\textbf{A4} and \eqref{eq: cond:bias} hold, with $\bar K_n = A k_n \log n$ and $A= c_2+ c_3+ c_4+1$ in assumption \textbf{A2}, then for every $\eps>0$ there exists a constant $L_{\eps,\a}>0$ such that
 \begin{equation}\label{eq:coverage}
\liminf_n \inf_{\theta_0 \in \Theta_0}P_{\theta_0}^{(n)} \left( \theta_0 \in \widehat C( L_{\eps,\a} \sqrt{\log n},\a)\right)  >1 - \epsilon.
 \end{equation} 
\end{theorem}

\begin{remark} \label{rk}
In Theorem \ref{th:coverage} (and also in Theorem \ref{th:EmpBayesCoverage} below), the inflation of the radius is of order $\sqrt{\log n}$, which is an unpleasant feature of the result. We believe that this is a necessary inflation, at least for centering points $\hat \theta_n$ leaving in the bulk of the posterior distribution, like the posterior mean. Indeed as appears in the proof, see also Lemmas  \ref{lem:post:Kn} and \ref{lem:rate}, the posterior mass essentially lives on the sets of $k$ that achieves the balance $ k\log n \asymp n \epsilon_n^2(k)$, while an optimal behaviour would be to achieve the balance $k \asymp n \epsilon_n^2(k)$. This is a typical feature of hierarchical (or empirical)  Bayesian approaches with a hyper-prior on the model $k$ and is strongly related to the $\log n$ penalty induced by the marginal likelihood, as expressed in the re-known BIC approximation. This results in having the posterior distribution concentrate on values of $k$ that are too small, so that the bias  $b(k) $ dominates the statistical error within each model $\Theta(k)$ which is $O(k/n)$. The necessity of the $\sqrt{\log n}$ factor is  demonstrated in the context of the nonparametric regression model. In Propositions \ref{prop: UB:radius} and \ref{prop: UB:radius:HB} it is shown that without a $\sqrt{\log n}$ blow up the credible sets have coverage tending to zero for certain representative (typical) elements of the polished tail class. There are two ways to temper this. One can either follow \cite{Gao:Zhou:13} using a block prior on the components of  $\theta$ which groups together models in blocks and within each block shrinks very strongly the coefficients to 0 to ensure that the selected models under the posterior have a large enough number of components. An alternative method could be to find a centering point $\hat \theta_n$ which is rougher than the posterior, see Lemma \ref{lem: Lep:center} in the supplement.
\end{remark}

The proof of Theorem \ref{th:coverage}  is deferred to Section \ref{sec:proof:covHB}. A key step in the proof is understanding the asymptotic behaviour of $\pi_k(k |\mathbf Y)$. In particular we show that the posterior distribution accumulates most of its mass on   $\mathcal K_n (M)$,  where a trade-off between bias and prior-penalization or complexity (equivalent to the variance term in the Gaussian setup) is achieved. This is presented in the following lemma:

 \begin{lemma}\label{lem:post:Kn}
Take any $\eps>0$ and assume that conditions \textbf{H} and \textbf{A1-A3} hold. Then there exists a large enough constant $\bar{M}_\eps>0$ such that
 \begin{equation*}
 \sup_{\theta_0\in \bar\Theta} E_{\theta_0}^{(n)}\big( \pi_k( k \notin \mathcal K_n(\bar{M}_\eps) |\mathbf Y )\big)\leq  \eps. 
 \end{equation*}
 \end{lemma}
 
 The proof is presented in Section \ref{sec:proof:lemKn} of the Supplementary material. Lemma \ref{lem:post:Kn} is similar in spirit to Theorem 2.1 of \cite{rousseau:szabo:2015:main}, however the definition of $\epsilon_n(k)$ being different in both paper, the proofs of both results are significantly different.

The following lemma states that $\epsilon_n(k_n)$ corresponds to the posterior concentration rates, hence $\hat \theta_n $ can be any random point of the posterior distribution or depending on $d(.,.)$ the posterior mean or other posterior summary.

 \begin{lemma}\label{lem:rate}
 Assume that  conditions  \textbf{H} and \textbf{A1-A3} hold. Then for every $\eps>0$ there exists $C_\eps>0$ such that 
  $$\sup_{\theta_0\in \bar\Theta} E_{\theta_0}^{(n)}\big(\pi\left( d(\theta, \theta_0)\geq C_{\eps}\epsilon_n(k_n)|\mathbf Y \right)\big)\leq\eps.$$
 \end{lemma}
 The proof of Lemma \ref{lem:rate} is presented in Section \ref{sec:pr:lemrate} in the supplement.

Finally we show that the radius of the credible set is bounded from above by a multiple of $\eps_n(k_n)$.

 \begin{corollary}\label{cor:radius}
Under the assumptions of Lemma \ref{lem:rate} and \eqref{freq} for all $\eps\in(0,1/2)$ there exists $K_{\eps}>0$ large enough such that
\begin{align*}
\inf_{\theta_0\in\bar\Theta}P_{\theta_0}^{(n)}\big(\diam(\widehat{C}(1,\alpha),d)\leq K_{\eps}\eps_n(k_n) \big)\geq 1-\eps.
\end{align*}
 \end{corollary}

The lemma is a straightforward consequence of assumption \eqref{freq} and Lemma \ref{lem:rate}. Note that a direct consequence of Corollary \ref{cor:radius}, together with Assumption \textbf{A0} is that $\eps_n(k_n)$ is also an upper bound on the posterior concentration rate.



\subsection{Empirical Bayes approach} \label{sec:empiric}

An alternative approach to endow the hyper-parameter $k$ by a prior is to estimate it from the data directly and plug in this estimator into the posterior distribution. One of the most commonly used approaches is the maximum marginal likelihood empirical Bayes method, where one estimates the hyper-parameter with the maximizer of the marginal likelihood function
\begin{align}
\hat{k}_n=\arg\max_{k \leq \bar K_n} \int_{\Theta(k)} e^{\ell_n(\theta)}\pi_{|k}(\theta)d\theta, \label{def: MMLE}
\end{align}
where $\ell_n(\theta)$ denotes the log-likelihood function. This empirical Bayes technique is closely related to the hierarchical Bayes approach \cite{rousseau:szabo:2015:main}, however, in certain situations they can have substantially different behaviour \cite{petrone:rousseau:scricciolo:14,castillo:mismer:2018}.

In the empirical Bayes approach we construct the (inflated) credible set similarly to the hierarchical Bayes case, i.e. we consider a $d$-ball around the centering point $\hat\theta_n$ (typically the empirical Bayes posterior mean or mode)
\begin{align}
\hat{C}_{\hat k_n}(L,\a)= \{\theta:\, d(\theta,\hat\theta_n)\leq L_n r_\a(\hat{k}_n)\},\label{def: EBcred}
\end{align}
where $L_n>0$ is a blow up factor and the radius $r_\alpha(\hat{k}_n)$ is defined as
\begin{align}
r_\alpha(\hat{k}_n)=\arg\inf_{r>0}\{\pi_{|\hat k_n}(d(\theta,\hat\theta_n)\leq r_\a(\hat{k}_n) | \mathbf Y )\geq 1-\a\},
\end{align}
for a typically small $\a\in(0,1)$. We show that these sets have similar size as the hierarchical Bayes credible sets and good frequentist coverage under the general polished tail condition \eqref{cond:tail:L} for appropriately large blow up factor $L_n$ of order $\sqrt{\log n}$.

\begin{theorem}\label{th:EmpBayesCoverage}
Assume that conditions \textbf{A0}-\textbf{A4} and \eqref{eq: cond:bias} hold with $\bar K_n\leq n^H$ for some $H\geq 0$. Then for every $\eps,\a\in(0,1)$ there exists a large enough constant $L_{\eps,\alpha}$ such that

  \begin{equation*}
\liminf_n \inf_{\theta_0 \in \Theta_0}P_{\theta_0}^{(n)} \left( \theta_0 \in \widehat C_{\hat k_n}( L_{\eps,\a} \sqrt{\log n},\a)\right)  \geq 1 - \epsilon.
 \end{equation*} 
  Furthermore, there exists $K_{\eps}>0$ such that 
  \begin{align*}
\inf_{\theta_0\in\bar\Theta}P_{\theta_0}^{(n)}\big(\diam(\widehat{C}_{\hat k_n}(1,\alpha),d)\leq K_{\eps}\eps_n(k_n) \big)\geq 1-\eps.
\end{align*}
\end{theorem}

The proof is deferred to Section \ref{sec:proof:covEB} in the supplement.

 \section{Application to various models} \label{sec:applis}
 
In this section we apply our abstract results (Theorems \ref{th:coverage} and \ref{th:EmpBayesCoverage} and Corollary \ref{cor:radius}) in four examples: nonparametric regression, density estimation with histogram priors and with exponential family priors, and nonparametric classification. To prove the contraction rate and coverage results in all of the examples we have shown that the considered semi-metrics are locally equivalent to the $\ell_2$ norm on the parameter space $\bar{\Theta}$. This result is of interest on its own right. Besides it also results in (nearly) optimal posterior contraction rates and coverage of the credible sets in terms of  the $\ell_2$ and other related norms.  Condition \textbf{A4} (ii) requires uniform and sharp control on the likelihood ratio. In the following examples we give a general strategy to prove such kind of statements, which can come handy in other nonparametric problems as well.

\subsection{Application to nonparametric regression} \label{sec:reg}
In this section we consider the fixed design regression model and investigate the behaviour of Bayesian credible sets based on sieve priors. Assume that we observe the sequence $\mathbf{Y}=(Y_1,Y_2,...,Y_n)$ satisfying
\begin{align}
Y_i=f_0(x_i)+\sigma Z_i,\quad x_i\in[0,1],\quad i=1,2,...,n,\label{def: reg}
\end{align}
where $Z_i$ are iid standard normal random variables, $\sigma=1$ for simplicity and $x_1,x_2,...,x_n$ are fixed (or random) design points. 

Next we consider the basis $\phi_1(x),\phi_2(x),...$ in $L_2[0,1]$. 
 Note that every $f\in L_2[0,1]$ can be written in the form $f(x)=f_{\theta}(x)=\sum_{i=1}^\infty \theta_i \phi_i(x)$ (with the convention $\theta=(\theta_1,...,\theta_k,0,0....)$ for $\theta\in\Theta(k)=\mathbb{R}^{k}$) and we assume that the true function $f_{\theta_0}$ belongs to a Sobolev-type smoothness class $S^{\beta}(L_0)$, defined as
\begin{align}
S^{\beta}(L_0)=\{f_{\theta}:\, \sum_{i=1}^{\infty}\theta_i^2 i^{2\beta}\leq L_0\},\quad\text{for some $\beta,L_0>0$,}\label{def: sobolev}
\end{align}
with typically unknown regularity parameter $\beta>0$. Note that depending on the basis functions $\phi_j$ this may or may not  refer to the classical Sobolev balls; note also that the Fourier basis satisfies the assumptions below. The minimax estimation rate, which typically coincides with the minimax size of confidence sets, see for instance \cite{robins:2006}, over $S^{\beta}(L_0)$ with respect to the $\ell_2$-norm is $n^{-\beta/(1+2\beta)}$.

Next, for any $k\leq n$ we introduce the notation $\Phi_k=(\phi_1,\phi_2,...,\phi_k)\in \mathbb{R}^{n\times k}$.  Let 
$d_n(\theta,\theta')^2=\frac{1}{n}\sum_{i=1}^n\big(f_{\theta}(x_i)-f_{\theta'}(x_i)\big)^2$ be the empirical $L_2$-norm between the functions $f_{\theta},f_{\theta'}\in L_2$. Let us introduce the notation  $f_{\theta,n} = (f_\theta(x_1), ..., f_\theta(x_n))$ and denote by $\theta_{[k]}^o$ the empirical $L_2$-norm projection of $f_{0,n} = (f_{\theta_0}(x_1), ..., f_{\theta_0}(x_n))^T$ to the space $\{\Phi_k \theta:\, \theta\in \mathbb R^k \}$ or in other words the $d_n$-projection of $\theta_0$ on $\mathbb{R}^k$.  Then defining $b(k)$ in terms of the semi-metric $d_n(.,.)$ leads to  $b(k)=d_n(
\theta_0,\theta_{[k]}^o)^2$  the approximation error of the true function with the $k$ dimensional projection.
Assume furthermore that there exists a constant $C_0\geq 1$ and a sequence $K_n$ going to infinity  such that
\begin{align}\label{assump: reg_design_mtx}
C_0^{-1}I_{K_n}\leq \frac{\Phi_{K_n}^T\Phi_{K_n} }{ n } \leq C_0 I_{K_n}.
\end{align}

\begin{remark}
The above assumptions on the choice of the basis functions $\phi_j(x)\in L_2[0,1]$ and the design points $x_1,x_2,...,x_n$ are very mild and standard. There are many suitable choice of basis satisfying these properties.  Orthonormal basis in $\mathbb{R}^n$, such as the discrete wavelet basis relative to the design points satisfy \eqref{assump: reg_design_mtx} with $K_n= n$, some orthonormal basis in $L_2$ will satisfy \eqref{assump: reg_design_mtx} for some finite value $K_n$. In the case of the Fourier basis for instance, \eqref{assump: reg_design_mtx} is valid as soon as  $K_n = o(n)$. 

\end{remark}

\begin{remark}
Let us consider a probability measure $\nu$ on $[0,1]$ and take an orthonormal $L_2(\nu)$ basis $\phi_j$. Then in view of Rudelson's inequality, \cite{rudelson:1999} 
\begin{equation} \label{app:design}
E_\nu\| \frac{\Phi_k^T\Phi_k }{ n } -I_k \|_2 \leq  M \sqrt{ \frac{k \log n }{n} } 
\end{equation}
for all $k \leq k_0 n/\log n $ and some $k_0$ small enough. Hence following from Lemma \ref{lem: help_loglin_2} in the supplementary material \cite{rousseau:szabo:16:supp}, if $K_n\log K_n =o( n ) $ assertion \eqref{assump: reg_design_mtx} is verified  with $\nu$-probability going to 1. 
\end{remark}

Due to the condition \eqref{assump: reg_design_mtx} we have to slightly modify the polished tail condition by assuming that the approximation error using the largest model $\Theta(K_n)$ is not too large, i.e we take 
$$\Theta_{0,n}   = \Theta_{0,n}(R_0,k_0,\tau) \cap \{ \theta_0: \,b( K_n ) \leq \delta K_n (\log n)/n \},$$
 for some $\delta < 1\wedge C_0$ and consider $\theta_0 \in \Theta_{0,n}$.  

\begin{remark}\label{rem:Kn2}
To understand better the meaning of the restriction $\theta_0 \in \Theta_{0,n}$, assume that $\sum_{j=1}^\infty|\theta_{0,j}| < +\infty$ and $\max_{j}\|\phi_j\|_{\infty}<\infty$ so that $\|f_{\theta_0}\|_{\infty}<\infty$. If \eqref{assump: reg_design_mtx} is true for all $1\leq k \leq Cn$, $C>0$,  then $\|\Delta_k \|_\infty = o(1)$ as $k$ goes to infinity, with
 $\Delta_k = f_{\theta_0} - \sum_{j=1}^k \theta_{0,j}\phi_j$. Since $b(k) \leq \|\Delta_k\|_\infty^2$ then there exists $K_n$, with $Cn\geq K_n\geq 1$, such that $b(K_n) \leq \delta K_n (\log n)/n$ for all $n\geq 2$ and $\delta >0$. Hence for all $L>0$,  $\{ \theta_0: \|\theta_0\|_1 \leq L \} \cap \Theta_0 \subset \Theta_{0,n}$, when $n$ is large enough, following from the inequality $\|f_{\theta}\|_{\infty}\leq  \|\theta\|_1 \max_{j}\|\phi_j\|_{\infty}$. However, if \eqref{assump: reg_design_mtx} is only true for $K_n=o(n)$, then $\Theta_{0,n}$ will typically be more constraint. 
For instance for $\theta_0\in S^{\beta}(L_0)$, $\beta >1/2$, we can bound
 $b(K_n) \leq \|\Delta_{K_n}\|_\infty^2 \lesssim K_n^{-2(\beta-1/2)}$ (using Cauchy-Schwarz inequality) so that $b(K_n) \leq \delta K_n (\log n)/n$ if $K_n \gtrsim (n/\log n)^{1/(2\beta)}$. In case $K_n  \asymp n/\log n$, $\beta >1/2$ is enough. The upper bound $K_n^{-2(\beta-1/2)}$ is independent of the design and the chosen basis and can be improved in particular cases.

For instance in the random design case with distribution $\nu$ and bounded orthonormal basis $\max_j \|\phi_j\|_\infty < +\infty $ and writing $\theta_{0,[k]}=(\theta_{0,1},...,\theta_{0,k})\in\mathbb{R}^k$ one has
\begin{equation*}
\begin{split}
\nu \big( d_n^2(\theta_0,\theta_{0,[K_n]}) > C \|\Delta_{K_n}\|_2^2 \big) &=  \nu \Big( \sum_{i=1}^n \big(\sum_{j=K_n+1}^\infty \theta_{0,j}\phi_j(x_i) \big)^2 > nC\|\Delta_{K_n}\|_2^2 \Big)\\
&\leq \frac{ E_\nu\Big( \big(\sum_{j=K_n+1}^\infty \theta_{0,j}\phi_j(X)\big)^2 \Big)}{ C \|\Delta_{K_n}\|_2^2} \leq  \frac{ 1}{ C }.
\end{split}
\end{equation*}
Therefore  $b(K_n)\leq d_n^2(\theta_0,\theta_{0,[K_n]})\leq C\|\Delta_{K_n}\|_2^2 \lesssim K_n^{-2 \beta }$ with large probability.
\end{remark}

\begin{remark}\label{rem: polished:tail:regression}
In the fixed design regression model with $K_n\geq n^{\frac{1}{2(\beta_0-1/2)}}$ (where $\beta_0>0$ is the smallest regularity level we are adapting to) the set $\Theta_{0,n}$ contains the set in $\{ \theta_0: \,b( K_n ) \leq \delta K_n (\log n)/n \}$ satisfying the $L_2$ polished tail condition of \cite{szabo:vdv:vzanten:13}, i.e. if $$\|\theta_{0,[R_0k]} - \theta_0\|_2^2\leq \tau_1\|\theta_{0,[k]} - \theta_0\|_2^2, \quad \tau_1 < 1/(5C_0^2)$$for all $k \geq k_0$, then $\theta_0 \in \Theta_{0,n}$. In the random design regression model (with arbitrary sequence $K_n$ tending to infinity) the above inclusion holds with $\nu$-probability arbitrary close to one. Therefore the discussion in \cite{szabo:vdv:vzanten:13} on the $L_2$ polished tail condition, in terms of the force of the restriction induced by this condition applies here, namely that the condition is non restrictive from a statistical complexity, topological and Bayesian perspective.
\end{remark}
The proof of the above remark is given in Section \ref{sec: polished:tail:regression} in the Supplementary material \cite{rousseau:szabo:16:supp}.


Then we define the prior distribution on the regression function $f$ by endowing the sequence of coefficients $\theta$  with the standard sieve prior, i.e.
\begin{align*}
\theta=(\theta_1,...,\theta_k) | k\sim \prod_{i=1}^k g(\theta_i),\\
 k\sim Geom(p)\, \text{or}\, Pois(\lambda),
\end{align*}
where $p\in(0,1)$ or $\lambda>0$ and $g(.)$ satisfies the standard assumption
\begin{align}
G_1 e^{-G_2 |x|^{q}}\leq g(x)\leq G_3 e^{-G_4 |x|^{q}}, \label{eq: sieve_prior}
\end{align}
for some positive constants $G_1,G_2,G_3,G_4$ and $q$. Alternatively we can also estimate $k$ by the MMLE \eqref{def: MMLE} and plug in the estimator $\hat{k}_n$ into the posterior. These type of priors were considered for instance in  \cite{arbel} and \cite{rousseau:szabo:2015:main}, where it was shown that the corresponding hierarchical and empirical Bayes posterior distributions achieve (up to a $\log n$ factor) adaptive contraction rate around the true function $f_0$. The frequentist behaviour of the Bayesian credible sets in the context of the regression model was investigated only in a few papers \cite{sniekers:2015,serra:2014,yoo:vaart:2017} for specific conjugate priors allowing direct computations, which can not be applied in the present setting due to the lack of explicit expression for the posterior.  Here we consider both the inflated hierarchical Bayes credible set
\begin{align*} 
 \hat{C}(L\sqrt{\log n}, \a)=\{\theta:\, d_n(\theta,\hat{\theta}_n)\leq L\sqrt{\log n}r_\a\},
\end{align*} 
with $\pi( \theta:\, d_n(\theta,\hat{\theta}_n)\leq r_\a|\mathbf{Y})\geq 1-\a$ and $\hat{\theta}_n$ satisfying assumption \textbf{A0} and the inflated MMLE empirical Bayes credible set defined along the same lines. By applying Theorems \ref{th:coverage} and \ref{th:EmpBayesCoverage} together with Corollary \ref{cor:radius} we can verify that both credible sets have good frequentist coverage and (almost) rate adaptive size under the general polished tail assumption.

\begin{proposition}\label{prop: reg}
Consider the fixed design regression model \eqref{def: reg} with $f_0\in S^{\beta}(L_0)$ for some $\beta\geq\beta_0>1/2$ and assume that condition \eqref{assump: reg_design_mtx} is satisfied with $K_n>n^{\frac{\beta_0}{(1+2\beta_0)(\beta_0-1/2)}}$.  Denote both the inflated hierarchical Bayes and empirical Bayes credible sets, centered around any estimator $\hat\theta_n$ satisfying \eqref{freq} by $\hat{C}_n(L\sqrt{\log n},\a)$. Then $\hat{C}_n(L\sqrt{\log n},\a)$ has (up to a $\log n$ factor) rate adaptive size and frequentist coverage tending to one under the general polished tail assumption  $\eqref{cond:tail:L}$, i.e. for every $\eps>0$ there exist a large enough $L,C>0$ such that
 \begin{align*}
 \liminf_n\inf_{\theta_0\in\Theta_{0,n}\cap S^{\beta_0}(L_0) } P_{\theta_0}^{(n)} \big( \theta_0\in \hat{C}_n(L\sqrt{\log n}, \a ) \big)\geq 1-\eps,\\
 \liminf_n\inf_{\beta\geq \beta_0}\inf_{\theta_0\in \mathcal{S}^{\beta}(L_0)}P_{\theta_0}^{(n)} \Big( \diam\big(\hat{C}_n(1, \a),d_n\big)\leq C \big(\frac{n}{\log n}\big)^{-\frac{\beta}{1+2\beta}}\Big)\geq 1-\eps.
 \end{align*}
\end{proposition}

The proof of the proposition is given in Section \ref{sec: proof_reg} of  the supplement \cite{rousseau:szabo:16:supp}.

\begin{remark}
Assumption \eqref{freq} on the estimator is very mild, for instance a typical draw from the posterior distribution satisfies it, see the comment above Lemma \ref{lem:rate}. Furthermore, standard estimators like the posterior mean also satisfies this assumption, see for instance \cite{arbel}. We also note that similar results hold for the random design regression as well.
\end{remark}

\begin{remark}
Recall that by assumption \eqref{assump: reg_design_mtx} the empirical $L_2$ semi-metric $d_n$  and the $\ell_2$-norm are equivalent over $\Theta(k),\,k\leq K_n$. Furthermore note that the inequality $\|\theta_0-\theta_{0,[K_n]}\|_2^2\lesssim K_n^{-2\beta}\leq 1/n$ holds for $K_n>n^{\frac{\beta_0}{(1+2\beta_0)(\beta_0-1/2)}}$. Hence we get that the same contraction rate and coverage statements as in Proposition \ref{prop: reg} hold for the metric $\ell_2$ as well. 
\end{remark}

The $\sqrt{\log n}$ blow up factor in the credible set is rather inconvenient and makes the procedure less appealing. The question naturally arises whether this blow up factor is just an artefact of the proof and can be removed or whether it is necessary to reach the desired frequentist coverage. We show below that without inflating the credible sets (centered at the posterior mean) with a multiple of $\sqrt{\log n}$ one would get coverage tending to zero for a large class of parameters satisfying the polished tail condition, justifying the presence of the inflating factor.

In view of \cite{szabo:vdv:vzanten:13} let us consider the class of self-similar functions 
\begin{align*}
\mathcal{H}_s^{\beta}(L)=\{f_{\theta}:\, L^{-1}i^{-\beta -1/2} \leq |\theta_i|\leq L i^{-\beta -1/2},\,i=1,2,... \},
\end{align*}
 where it was also shown that the present set is not substantially smaller than the entire hyper-rectangle (the set without the lower bound assumption on $|\theta_i|$) from a topological and statistical complexity point of view. Note also that $\mathcal{H}_s^{\beta}(L)\subset S^{\beta+\eps}(C)$, for arbitrary $\eps>0$ and some sufficiently large constant $C>0$.

\begin{proposition}\label{prop: UB:radius}
Consider the fixed design regression model  \eqref{def: reg} with $f_0\in \mathcal{H}_s^{\beta}(L)$ for some $\beta\geq\beta_0>1/2$ and orthogonal basis $\Phi_{K_n}^T \Phi_{K_n}=nI_{K_n}$ (where $K_n>n^{\frac{\beta_0}{(1+2\beta_0)(\beta_0-1/2)}}$). Furthermore take the prior $g(\theta)$ to be either the normal $N(\mu,\sigma)$ or Laplace $Lap(\mu,b)$ distribution. Then the empirical Bayes credible set centered around the posterior mean $\hat\theta_{\hat{k}_n}$ and inflated with a factor $m_n\log^{1/2} n$, for arbitrary $m_n=o(1)$, has frequentist coverage tending to zero, i.e. for every $\alpha>0$
\begin{align*} 
\limsup_n\sup_{f_0\in \mathcal{H}_s^{\beta}(L)}P_{\theta_0}^{(n)} \Big( \theta_0\in 
\{\theta:\, d_n(\theta,\hat\theta_{k_n})\leq m_n\sqrt{\log n}r_{\alpha}(\hat{k}_n) \} \Big) = 0
 \end{align*}
\end{proposition}

The proof of the proposition is given in Section \ref{sec: UB:radius} of the Supplementary material \cite{rousseau:szabo:16:supp}.

It is common or floklore knowledge that empirical Bayes procedures underestimate uncertainty compared to hierarchical Bayes procedures. However we prove below that under slightly more restrictive conditions on $\theta_0$ the same blow up factor is required for the hierarchical Bayes credible ball centered at the posterior mean. 
 More precisely let $\ell : \mathbb N \rightarrow \mathbb R_+$ be  a slowly varying function going to 0 at infinity and set
 $$  \mathcal{H}_{ss}^{\beta}(L,\ell ) = \{f_\theta \in  \mathcal{H}_s^{\beta}(L); \, \exists r_\infty \in [1/L,L] ; \,  |\theta_i^2i^{2\beta+1} - r_\infty^2| \leq \ell(i) , \, i\geq 1\}. $$
 
\begin{proposition}\label{prop: UB:radius:HB}
Consider the fixed design regression model  \eqref{def: reg} with $f_0\in \mathcal{H}_s^{\beta}(L)$ for some $\beta\geq\beta_0>1/2$ and orthogonal basis $\Phi_{K_n}^T \Phi_{K_n}=nI_{K_n}$ (where $K_n>n^{\frac{\beta_0}{(1+2\beta_0)(\beta_0-1/2)}}$). Furthermore assume that the log-prior $\log g(\theta)$ is continuously differentiable on $\mathbb R$. 
Then the hierarchical Bayes credible set centered around the posterior mean $\hat\theta$ and inflated with a factor $m_n\delta_n^{-1/2}$, $ \delta_n  = 1/\log n + \ell(k_n)$  for arbitrary $m_n=o(1)$, has frequentist coverage tending to zero, i.e. for every $\alpha>0$
\begin{align*} 
\limsup_n\sup_{f_0\in \mathcal{H}_{ss}^{\beta}(L,\ell )}P_{\theta_0}^{(n)} \Big( \theta_0\in 
\{\theta:\, d_n(\theta,\hat\theta)\leq m_n \delta_n^{-1/2}r_{\alpha} \} \Big) = 0
 \end{align*}
In particular if $\ell(i) \lesssim 1/\log (i) $, then $\delta_n^{-1/2} \asymp\sqrt{\log n}$. 
\end{proposition}

The proof of the proposition is given in Section \ref{sec: UB:radius:HB} of the Supplementary material \cite{rousseau:szabo:16:supp}.

\begin{remark}\label{rem:counter}
From the proofs of Propositions \ref{prop: UB:radius} and \ref{prop: UB:radius:HB} it appears that the necessity of the logarithmic blow up follows from (1) the sub-optimal behaviour of the posterior mean and (2) the concentration of the posterior  (or the empirical Bayes) distribution  on values of $k$ that are too small. We note, however, that any other summary statistics of the posterior distribution has similar behaviour as the bulk of the posterior is located at a sub-optimal place. One can of course choose other centering points, not related to the posterior, to avoid the $\sqrt{\log n}$ blow up factor. For instance one can consider Lepski's estimator in the sequence model, see Lemma \ref{lem: Lep:center} in the supplement, but this brings us outside of the Bayesian framework and we are hesitant recommending such a solution. 
\end{remark}

\subsection{Application to density estimation using histogram priors} \label{sec: histogram}
 In this section we consider the density estimation model, i.e. we assume to observe $\textbf{Y}=\{Y_1,Y_2,...,Y_n\}$ iid samples from a true density function $p_0$ and our goal is to recover this density. We assume that $p_0$ is continuous, bounded from below by $c_0$ and from above by $C_0$. Furthermore we assume that it belongs to a H\"older smoothness class $\mathcal{H}^{\beta}(L_0)$ for some $\beta\in(0,1]$.
 
We investigate the Bayesian approach using histogram prior distributions, see for instance \cite{SCRICCIOLO:2007,castillo:rousseau:2015,rousseau:szabo:2015:main}. In other words let $\Theta(k)$ denote the collection of $k$-bins random histogram where the bins are regular : $[(j-1)/k, j/k)$, $j =1,..., k$,
 \begin{equation}
p_{\theta} (x) = k\sum_{j=1}^k \theta_j\1_{I_j}(x), \quad \theta_j \geq 0 ,\quad \sum_{j=1}^k \theta_j = 1.\label{def: hist}
 \end{equation}
 We therefore identify $\Theta(k)$ with the $k$-dimensional simplex $\mathcal S_k  = \{ x\in [0,1]^k ; \sum_{i=1}^k x_i = 1\}$.  First we endow the hyper-parameter $k$ with either a Poisson $ Pois(\lambda) $ or a Geometric  $Geom(p)$ hyper-prior with $\lambda >0$ and $0<p<1$. Given $k$ consider a Dirichlet prior $\mathcal D(\alpha_{1,k}, ..., \alpha_{k,k})$  on $(\theta_1,..., \theta_k)$, i.e. the hierarchical prior $\pi$ on the densities takes the form
 \begin{align*}
\theta= (\theta_1,..., \theta_k) | k \sim  \mathcal D(\alpha_{1,k}, ..., \alpha_{k,k}),\quad c_1k^{-a} \leq \alpha_{j,k} \leq C_1\\
 k\sim Geom(p)\, \text{or}\, Pois(\lambda).
 \end{align*}
 for some $a\geq 0$ and $c_1, C_1 >0$. 
Alternatively we apply the MMLE $\hat{k}_n$ for the hyper-parameter $k$ and then consider the Dirichlet prior $\mathcal D(\alpha_{1,\hat{k}_n}, ..., \alpha_{\hat{k}_n,\hat{k}_n})$ on $(\theta_1,..., \theta_{\hat{k}_n})$.

Then we consider the inflated hierarchical Bayes credible set 
$$\hat{C}(L\sqrt{\log n},\a)=\{p_\theta:\, h(p_{\theta},p_{\hat\theta_n})\leq L\sqrt{\log n}r_{\a} \},$$
with $h(.,.)$ the Hellinger distance, $\hat\theta_n$ satisfying assumption \eqref{freq} with $d(\theta,\theta')=h(p_{\theta},p_{\theta'})$ and the radius $r_{\alpha}$ satisfies $\pi(\theta:\, h(p_{\theta},p_{\hat\theta_n})\leq r_{\a}|\textbf{Y})\geq 1-\a$. Note that since the Hellinger metric is bounded and convex, and the posterior distribution contracts around the truth with the optimal rate $\eps_n(k_n)$ the posterior mean satisfies condition \eqref{freq}, see page 507 of \cite{ghosal:ghosh:vdv:00}. The inflated empirical Bayes credible set $\hat{C}_{\hat{k}_n}(L\sqrt{\log n},\a)$ is defined along the same lines. Applying again Theorems \ref{th:coverage} and \ref{th:EmpBayesCoverage} together with Corollary \ref{cor:radius} we can verify that both credible sets  have high frequentist coverage and (almost) rate adaptive size under the general polished tail assumption.

\begin{proposition}\label{prop: hist}
Consider the density estimation model with histogram priors \eqref{def: hist} and assume that $p_0\in \mathcal{H}^{\beta}(L_0)$ for some $\beta\in[\beta_0,1]$, $\beta_0>1/2$, and it is bounded away from zero and infinity. Then both the inflated hierarchical Bayes and empirical Bayes credible sets with centering point $p_{\hat\theta_n}$ satisfying \eqref{freq} have (up to a $log n$ factor) rate adaptive size and frequentist coverage tending to one under the polished tail assumption  $\eqref{cond:tail:L}$, i.e. for every $\eps>0$ there exist $L_\eps,C_\eps>0$ such that
 \begin{align*}
 \liminf_n\inf_{p_0\in \Theta_0\cap \mathcal{H}^{\beta_0}(L_0) } P_{p_0}^{(n)} \big( p_0\in\hat{C}_n(L\sqrt{\log n},\a)) \big)\geq1-\eps,\\
  \liminf_n\inf_{\beta\in[\beta_0,1]}\inf_{p_0\in\mathcal{H}^{\beta}(L_0)}P_{p_0}^{(n)} \Big(  \diam\big(\hat{C}_n(1,\a),h\big)\leq C\big(\frac{n}{\log n}\big)^{-\frac{\beta}{1+2\beta}}\Big)\geq 1-\eps,
 \end{align*}
where $\hat{C}_n(L\sqrt{\log n},\a)$ could either denote the  hierarchical or the empirical Bayes credible sets inflated by an $L\sqrt{\log n}$ multiplier.
\end{proposition} 
The proposition is verified in Section \ref{sec: hist} of the Supplementary material \cite{rousseau:szabo:16:supp}.

\begin{remark}
Using Lemma 3 of the Supplementary material, we have  $h(p_0, p_\theta) \asymp \| p_0- p_\theta\|_2$ in a neighbourhood of $p_0$ if $k$ is not too large, so that the polished tail condition in the Hellinger distance is equivalent to the polished tail condition in the $L_2$-norm (associated to different constants). To understand the latter note that if $p_{0,[k]}$ is the $L_2$ projection of $p_0$ and $b_2(k)$ is the $L_2$ bias,  then for any positive integer  $R_0 $,
$b_2(k)  = b_2(2R_0 k) + \| p_{0,[k]} - p_{0,[2R_0k]}\|_2^2 $ so that  the $L_2$ polished tail condition is equivalent to 
$$\| p_{0,[k]} - p_{0,[2R_0k]} \|_2^2 \geq (1-\tau) \| p_0 - p_{0,[k]}\|_2^2,$$
which has a similar flavour to the polished tail condition of \cite{szabo:vdv:vzanten:13}.
\end{remark}
\subsection{Application to density estimation with exponential families of prior}
In this subsection we consider again the density estimation problem on $[0,1]$, i.e. we assume that we observe independent and identically distributed draws  $\textbf{Y}=\{Y_1,Y_2,...,Y_n\}$ from a distribution with density function $f_0$ (with respect to the Lebesgue measure). Then we
assume that the true density can be written as an infinite dimensional exponential distribution 
\begin{align}
f_0(x) = \exp\Big( \sum_{j=1}^\infty \theta_{0,j} \phi_j(x) - c(\theta_0) \Big), \quad x\in[0,1],\label{def: loglin}\\
\text{with}\quad e^{c(\theta_0)} = \int_0^1 \exp\Big( \sum_{j=1}^\infty\theta_{0,j} \phi_j(x) \Big)dx,\nonumber
\end{align}
for some $\theta_0=(\theta_{0,1},\theta_{0,2},...)\in\ell_2$. For any $\theta\in\ell_2$ we define $f_{\theta}=\exp(\sum \theta_j\phi_j-c(\theta))$ and hence $f_0=f_{\theta_0}$. This model is also known as the log-linear model. Furthermore we also assume that $\|\log f_0 \|_\infty <+\infty $, that $\phi_j(x)$, $j=1,2,...$ forms an orthonormal basis (together with $\phi_0(x)\equiv 1$ and therefore satisfies  $\int_0^1\phi_j(x)dx = 0$ for all $j\geq 1$)  and that $\theta_0 \in \mathcal S^\beta (L_0)$ for some $\beta,L_0>0$ as in \eqref{def: sobolev}. 

Then we define the prior distribution on the densities with hyper-parameter $k$ by endowing the sequence $\theta$  in the log-linear model with the standard sieve prior, i.e.
\begin{align*}
\theta=(\theta_1,...,\theta_k) | k\sim \prod_{i=1}^k g(\theta_i),\\
k\sim Geom(p)\, \text{or}\, Pois(\lambda),
 \end{align*}
 for some fixed $p\in(0,1)$ or $\lambda>0$ and $g(.)$ satisfying \eqref{eq: sieve_prior}. Alternatively one can estimate $k$ from the data by the MMLE and plug in the estimator $\hat{k}_n$ into the posterior distribution. Similarly to Section \ref{sec:reg}, here $\Theta(k) = \mathbb R^k$. 

These type of priors were considered for instance in \cite{verdinelli:wasserman:1998,vvvz08,rivoirard:rousseau:09,rivoirard:rousseau:12,arbel,rousseau:szabo:2015:main}, where rate adaptive posterior contraction rates were derived. However, the reliability of Bayesian uncertainty quantification in this model was not investigated yet in the literature.

By using the corresponding posterior distribution we construct the inflated hierarchical credible set as
\begin{align*}
\hat{C}(L\sqrt{\log n},\a)=\{f_{\theta}:\, h(f_{\theta},f_{\hat\theta_n})\leq L\sqrt{\log n} r_{\a} \},
\end{align*}
where $h(.,.)$ denotes the Hellinger distance, the radius $r_\alpha$ satisfies $\pi (\theta:\, h(f_{\theta},f_{\hat\theta_n})\leq  r_{\a}|\textbf{Y})\geq 1-\a $ and  $\hat\theta_n$ is an arbitrary estimator satisfying \eqref{freq} with $d(\theta,\theta')=h(f_{\theta},f_{\theta'})$. We note that similarly to the histogram example above the posterior mean satisfies condition \eqref{freq} hence can be used as a centering point of the credible set. The construction of the inflated empirical Bayes credible set  $\hat{C}_{\hat{k}_n}(L\sqrt{\log n},\a)$ goes similarly.
Using again Theorems \ref{th:coverage} and \ref{th:EmpBayesCoverage} together with Corollary \ref{cor:radius} we can verify that the preceding credible sets have high frequentist coverage and (almost) rate adaptive size under the general polished tail assumption.

\begin{proposition}\label{prop: loglin}
Consider the log-linear model \eqref{def: loglin}. Then both the inflated hierarchical and empirical Bayes credible sets have (up to a $\log n$ factor) rate adaptive size and frequentist coverage tending to one under the general polished tail assumption  $\eqref{cond:tail:L}$, i.e. for every $\beta_0>1/2$ and $\eps>0$ there exist $L_\eps,C_\eps>0$ such that
 \begin{align*}
\liminf_{n} \inf_{\theta_0\in \Theta_0\cap \mathcal{S}^{\beta_0}(L_0) } P_{\theta_0}^{(n)} \big( f_{\theta_0}\in \hat{C}_n(L_\eps\sqrt{\log n},\a) \big)\geq 1-\eps,\\
\liminf_{n} \inf_{\beta\geq\beta_0}\inf_{\theta_0\in \mathcal{S}^{\beta}(L_0)}P_{\theta_0}^{(n)} \Big( \diam\big(\hat{C}_n(1,\a) ,h\big)\leq C_\eps \big(\frac{n}{\log n}\big)^{-\frac{\beta}{1+2\beta}}\Big)\geq 1-\eps,
 \end{align*}
 where $\hat{C}_n(L_\eps\sqrt{\log n},\a)$ denotes either the inflated hierarchical or empirical Bayes credible set with a blow up factor $L_\eps\sqrt{\log n}$.
\end{proposition}

The proof of the proposition is given in Section \ref{sec: proof_loglin} of the Supplementary material \cite{rousseau:szabo:16:supp}.

 \begin{remark}
In view of Lemma \ref{lem: help_loglin_1} in the supplement we note that the rate and coverage statements of Proposition \ref{prop: hist} also hold for the $\ell_2$-metric. 
\end{remark}

\begin{remark}
Again, similarly to before, if $f_{\theta_0} \in \mathcal S^{\beta_0}(L)$ with $\beta_0 >1/2$ and if $k \leq \bar K_n$ then for all $k_0\leq k \leq k_n$  we have
$\|\theta_0 - \theta_{0,[k]}\|_2 \leq Lk^{-\beta_0}$, where $\theta_{0,[k]}=(\theta_{0,1},...,\theta_{0,k},0,0,...)$, and if $k_0 \geq (L/\epsilon)^{1/\beta_0}$ with $\epsilon >0$ arbitrarily small, using Lemma 5 in the Supplementary material \cite{rousseau:szabo:16:supp}, 
$$b(k ) \asymp \|\theta_0 - \theta_{0,[k]}\|_2^2,\qquad \text{for $k\geq k_0$}.$$
Therefore the parameters $\theta$ satisfying the $L_2$ polished tail condition of \cite{szabo:vdv:vzanten:13}, see also \eqref{def: PT:original}, is a subset of $\Theta_{0}$. 
\end{remark}

\subsection{Application to nonparametric classification}  \label{sec:classif}
In this section we apply our general theorem to the nonparametric classification (or also known as binary regression) model. We assume to observe the sequence $\textbf{Y}=(Y_1,Y_2,...,Y_n)\in\{0,1\}^n$ satisfying
\begin{align}
P(Y_i=1|x_i )=q_0(x_i),\quad \text{for some $q_0: [0,1]\mapsto (0,1)$},\label{def: class}
\end{align}
with $x_i\in[0,1]$, $i=1,...,n$ fixed design points. We also take  $\mu(x)=e^{x}/(1+e^x)$ to be the logistic link function.

We assume that  under the true distribution associated to  $q_0$,   $f_0=\mu^{-1}(q_0)\in \mathcal{S}^{\beta}(L_0)$, with unknown smoothness parameter $\beta>0$. Minimax estimation rates with respect to the $L_2$-norm , i.e. $n^{-\beta/(1+2\beta)}$ for $S^{\beta}(L_0)$, and an adaptive estimator achieving these rates in the random design case was derived for instance in \cite{yang:1999, yang:1999b}. Similarly to the density estimation and classification models it was also shown that with respect to the $L_2$-norm it is impossible to construct adaptive confidence sets over a full scale of regularity classes $\cup_{\beta\geq\beta_0}S^{\beta}(L_0)$, for any $\beta_0>0$, see \cite{mukherjee:2018}.

 In the Bayesian approach one endows the nonparametric function $f$ with a prior distribution resulting in a prior on the binary regression function $q$.  The theoretical properties of the Bayesian approach in the present model were investigated for instance in \cite{ghosal:vdv:07} with linear function $f$, in \cite{vvvz08} with Gaussian process priors on the nonparametric function $f$ and in \cite{kirichenko:2015} in context of classification of the nodes of large graphs. In the preceding papers adaptive posterior contraction rates were derived. However, the coverage properties of Bayesian credible sets remained unknown. Due to the lack of an explicit formula for the posterior distribution direct computations are not feasible to quantify the reliability of Bayesian credible sets. Therefore, we tackle this until now unanswered question by applying our general, abstract theorem.

In our analysis we consider again the popular sieve prior. For given $k$ we introduce the parametrization
\begin{align*}
f_{\theta}(x_i)=\sum_{j=1}^{k}\theta_j \phi_j(x_i)= \Phi_k(x_i)\theta,
\end{align*}
with $\theta=(\theta_1,...,\theta_k)^T\in \Theta(k) = \mathbb{R}^k$ and $\Phi_k(x_i)=\big(\phi_1(x_i),\phi_2(x_i),...,\phi_k(x_i)\big)$, as in Section \ref{sec:reg}, satisfying assumption \eqref{assump: reg_design_mtx}. We work with the average (empirical) Hellinger semi-metric
 \begin{equation*}
 \begin{split}
 h_n^2(q_1, q_2) &= \frac{1}{  n } \sum_{i=1}^n h_b^2( q_1(x_i),  q_2(x_i)) , \qquad\text{where}\\
 h_b( q_1(x_i),  q_2(x_i)) &= (\sqrt{q_1(x_i)} -\sqrt{ q_2(x_i)})^2 + (\sqrt{1 -q_1(x_i)} -\sqrt{1- q_2(x_i)})^2.
 \end{split}
 \end{equation*}

\begin{remark}\label{rem: UBforEps_n}
Since assumption \eqref{assump: reg_design_mtx} is in a general and weak form, similarly to the nonparametric regression example we have to slightly strengthen our polished tail assumption. To see this first note that $ h_n^2(q_1, q_2) \leq d_n^2(f_1, f_2)$ with $f_j(x) = \mu^{-1}(q_j(x))$, $j=1,2$.
 Similarly to before, to understand the coverage properties of the credible balls, we need to study the bias function $b(k)$ with respect to the semi-metric  $h_n$.   Assume  $\theta_0 \in \mathcal S^\beta (L_0) $ for $\beta \geq \beta_0>1/2$ and $L_0>0$. Denote by $\tilde b( .  ) $ the bias function associated to $d_n(f_{\theta_0}, f_\theta)$ and studied in Section \ref{sec:reg}. Assume that $K_n$ satisfies $\tilde b(K_n) \leq \delta K_n (\log n)/n$ for some small enough $\delta$. Then since $b(K_n) \leq \tilde b(K_n)$ we get $b(K_n) \leq \delta K_n (\log n)/n$.  The discussion on the feasibility of the constraint $\tilde b(K_n) \leq \delta K_n (\log n)/n$ is similar to that of Section  \ref{sec:reg}.  As in the case of the regression model, using \eqref{eq: equiv:metric} of the Supplementary material \cite{rousseau:szabo:16:supp}, if $f_{\theta_0} \in \mathcal S^{\beta_0}(L)$ with $\beta_0>1/2$, $d_n(\theta, \theta_0) \asymp h_n(\theta, \theta_0)$ locally. Using the same arguments as in Section \ref{sec:reg}, if $\theta_0$ satisfies the $L_2$ polished tail condition of  \cite{szabo:vdv:vzanten:13}, then it satisfies the generalized polished tail condition. 

\end{remark}

In this example we consider the prior 
\begin{align*}
\theta=(\theta_1,...,\theta_k) | k\sim \prod_{i=1}^k g(\theta_i),\\
k\sim Geom(p)\, \text{or}\, Pois(\lambda),
\end{align*}
with $g(.)$ satisfying \eqref{eq: sieve_prior}, and $p\in(0,1)$ or $\lambda>0$, resulting in the two level hierarchical prior $\pi(.)$. Alternatively, we estimate $k$ using the MMLE and plug it in into the posterior for $\theta$ given $k$. Then we consider credible balls in terms of $q(x) = \mu(f(x) ) $, and the empirical Hellinger semi-metric $h_n(.,.)$.

The inflated hierarchical Bayes credible balls are defined as
 $$\hat{C}(L\sqrt{\log n}, \a) = \{ q_\theta (.):\, h_n( q_\theta , \hat q_{n}) \leq L\sqrt{\log n} r_\a\},$$ 
with radius $r_{\alpha}$ given by $\pi(\theta:\, h_n( q_\theta , \hat q_{n}) \leq r_\a| \mathbf{Y})\geq 1-\a$ and taking the posterior mean $\hat q_{n}=E_{\pi(.|\textbf{Y})}(q_{\theta})$ as the centering point.  Note that by convexity and boundedness of $q\rightarrow h_n^2(q,q_0)$, the posterior mean $\hat q_n$ satisfies condition \eqref{freq}. Alternatively we can use any centering point satisfying condition \eqref{freq}. The inflated  empirical Bayes credible ball $\hat{C}_{\hat k_n}(L\sqrt{\log n},\a)$ is defined similarly.

By applying our main Theorems \ref{th:coverage} and \ref{th:EmpBayesCoverage} and Corollary \ref{cor:radius} we show that under the general polished tail assumption $\eqref{cond:tail:L}$ both of the inflated credible sets have (nearly) optimal frequentist behaviour.

\begin{proposition}\label{prop: classification}
Consider the classification model given in \eqref{def: class} with $q_0=\mu(f_{\theta_0})$ satisfying  $\theta_0\in S^{\beta}(L_0)$, $\beta\geq\beta_0>1/2$ and $K_n\gg n^{\frac{1}{2(\beta_0-1/2)}}$. Then both the inflated hierarchical and empirical Bayes credible sets $\hat{C}_n(L\sqrt{\log n}, \a)$ - denoting either $\hat{C}(L\sqrt{\log n}, \a)$ in the hierarchical approach or $\hat{C}_{\hat{k}_n}(L\sqrt{\log n}, \a)$ in the empirical approach - have (up to a $\log n$ factor) rate adaptive size and frequentist coverage arbitrary close to one under the polished tail assumption, i.e.  for every $\eps>0$ there exist constants $L, C>0$ such that
 \begin{align*}
 \liminf_n\inf_{\theta_0 \in \Theta_{0,n}\cap  \mathcal S^{\beta_0}(L_0)} P_{\theta_0}^{(n)} \big( q_{\theta_0}\in \hat{C}_n(L_\eps\sqrt{\log n}, \a)  \big)\geq 1-\eps,\\
 \liminf_n\inf_{\beta\geq\beta_0}\inf_{ \theta_0 \in \mathcal S^{\beta}(L_0)}P_{\theta_0}^{(n)} \Big( \diam\big(\hat{C}_n(1, \a) ,h_n\big)\leq C_\eps \big(\frac{n}{\log n}\big)^{-\frac{\beta}{1+2\beta}}\Big)\geq 1-\eps.
 \end{align*}
\end{proposition}
The proof of the proposition is deferred to Section \ref{sec: proof_class} of  the supplementary material \cite{rousseau:szabo:16:supp}.

 \begin{remark} We note that the empirical Hellinger semi-metric $h_n$ is locally equivalent to the empirical $L_2$ semi-metric $d_n(f_{\theta_1},f_{\theta_2})$ and the $\ell_2$-norm $\|\theta_1-\theta_2\|_2$, see assertions \eqref{eq: equiv:metric} and \eqref{eq: help1_classification} in the Supplementary material \cite{rousseau:szabo:16:supp}. Therefore the same coverage and contraction rate results can be shown for these semi-metrics as well.
 \end{remark}
 
 \subsection{Parametric models and the BIC formula} \label{parametric} 
 
 Interestingly the lower bound obtained in Propositions \ref{prop: UB:radius} and \ref{prop: UB:radius:HB} also hold for simple regular parametric models. Consider the same structure as model  \eqref{def: model} with $k \leq K <+\infty$, and $\Theta(k+1) \subset \Theta(k)$.  
Now assume that there exist $k_1< k_0\leq K$ such that $\theta_0 \in \Theta(k_0)$ and $\inf_{\theta\in\Theta(k_1)} \|\theta_0 - \theta\|_2 = \delta \sqrt{(\log n)/n} $  and  $\inf_{\theta\in\Theta(k_1-1)} \|\theta_0 - \theta\|_2 \geq C \sqrt{(\log n)/n} $, for some small $\delta >0$ and large $C>0$. In other words $\theta_0$ is close to $\Theta(k_1) $ but not close to any $\Theta(k)$ with $k <k_1$.  Then under usual regularity assumptions, see Section \ref{pr:LB:para} of the supplement,   for any $\alpha \in (0,1)$,  
\begin{equation}\label{LB:para}
\pi(k \leq k_1|\mathbf Y) = 1 + o_{P_{\theta_0}}(1), \quad 
P_{\theta_0}( \theta_0 \in \hat C( L_n,\alpha) ) = o(1) , \quad \forall L_n = o(\sqrt{\log n}), 
\end{equation} 
where $\hat C(1,\alpha) = \{ \|\theta - \hat \theta \|_2 \leq r_\alpha\} $ is the $\alpha$ credible ball centered at the posterior mean $\hat \theta$. 

In other words, because of the $\log n $ penalization induced by the integration over the parameter spaces $\Theta(k)$, signals of order $\sqrt{(\log n)/n}$ may  be estimated at 0 and the posterior concentrates on a smaller dimensional parameter set, thus underestimating the uncertainty.

\section{Discussion}\label{sec:disc}

In this paper we have provided some general tools to study the frequentist properties of inflated credible balls in infinite dimensional models based on sieve priors. We have also studied three types of models: regression, density estimation and classification. As we can see from our results a key condition for the good behaviour of these inflated balls is the fact that the posterior distribution concentrates on the values of $k$ for which $b(k) \asymp k (\log n)/n$ and this is verified under the generalized polished tail condition, together with some other technical conditions. An intriguing feature of our result is the fact that we had to inflate the credible balls by a factor of order $\sqrt{\log n}$. In the case of the regression model,  under both  the empirical and hierarchical Bayes posteriors we have shown that this inflation is necessary, in order to obtain good frequentist coverage. The reason behind it is that the marginal maximum likelihood estimator $\hat k_n$ corresponds to a value $k$ such that the bias $b(k) \asymp k (\log n)/n$, while the estimation error (and thus the radius $r_\alpha^2$ ) is of order $k/n$. We believe that this (negative) result remains valid for the other models (density estimation and classification). 

We believe that this is in fact an important take away message from our results, i.e. the model selection priors, induce a penalization of order $d_k \log n $, where $d_k $ is the dimension of the parameter space reminiscent of the BIC formula, which in turn induce a loss in uniform coverage. This is even still true in simple regular parametric models, as discussed in Section \ref{parametric}.

 From a practical perspective, these credible balls can be approximately visualized by plotting the curves under the posterior distribution which satisfy the constraint $d(\theta, \hat  \theta_n) \leq L \sqrt{\log n} r_\alpha$, as was done for instance  in \cite{ray2017} and in \cite{szabo:vdv:vzanten:13}. In general visualization of confidence sets outside of the $L_{\infty}$ or point-wise case is challenging and we are not aware of any practical solution for instance for $L_2$- or Hellinger-confidence balls.

 The paper focuses on priors based on the structure \eqref{def: model}. This represents a general family of prior models but of course does not cover every possible prior. In particular hierarchical priors based on a continuous hyperparameter, such as hierarchical Gaussian processes, are not tackled by the present approach. There is so far no general theory for such priors and the only existing results so far are based of particular models and particular priors for which explicit computations can be derived, as in \cite{szabo:vdv:vzanten:13}.

\section{Proof of Theorem \ref{th:coverage}}\label{sec:proof:covHB}

Theorem \ref{th:coverage} is a simple consequence of the following lemma which allows to control the prior mass of neighbourhoods of $\hat \theta_n$.
\begin{lemma}\label{th:prior}
Under the same assumptions as in Theorem  \ref{th:coverage} for every $\eps>0$ there exists a small enough $\delta_\eps>0$ such that for $\rho_n = \delta_{\eps}/\sqrt{\log n}$
 \begin{equation*}
 \sup_{\theta_0\in \Theta_0} E_{\theta_0}^{(n)}\left( \pi(d(\theta, \hat \theta_n)\leq \rho_n  \epsilon_n(k_n)|\mathbf Y ) \right) \leq \eps.
 \end{equation*}
\end{lemma}

The proof of Lemma \ref{th:prior} is presented in Section \ref{sec:pr:thprior}.
 We now give the proof of Theorem \ref{th:coverage}.


\begin{proof}[Proof of Theorem \ref{th:coverage}]
Let $L_n = L_{\eps,\a} \sqrt{\log n}$ (for some 
$L_{\eps,\a}>0$ to be specified later) and $\epsilon_n = \epsilon_n(k_n)$. Then by assumption \eqref{freq} and definition \eqref{def: HBcred} we have for every $\eps>0$ that
\begin{equation*}
\begin{split}
P_{\theta_0}^{(n)} \left( \theta_0\in  \widehat C(L_n, \a) \right)  & = P_{\theta_0}^{(n)} \left( d(\theta_0,\hat \theta_n) \leq  L_nr_\a \right) \\
&\geq P_{\theta_0}^{(n)} \left[ \pi\left(  d(\theta,\hat \theta_n) \leq  d(\theta_0,\hat \theta_n)/L_n  \big| \mathbf Y \right) \leq 1 -\alpha \right]\\
& \geq P_{\theta_0}^{(n)}\Big[ \pi\Big(\theta:\, d(\theta,\hat\theta_n)\leq M_\eps \eps_n/L_n| \mathbf Y\Big)\leq1-\alpha \Big]-\eps.  
\end{split}
\end{equation*}
We show below that the first term on the right hand side is bounded from below by $1-\eps$.
In view of Lemma \ref{th:prior} there exists $\delta_{\eps,\a} >0$ small enough such that 
$$ \sup_{\theta_0\in \Theta_0}  E_{\theta_0}^{(n)}\left( \pi\big( d(\theta, \hat \theta_n) \leq \delta_{\eps,\a} \epsilon_n/\sqrt{\log n} | \mathbf Y \big) \right) \leq \eps(1-\a),$$
and therefore by taking $L_{\eps,\a}=M_{\eps}/\delta_{\eps,\a}$  and applying  Markov's inequality
\begin{align*}
P_{\theta_0}^{(n)}\Big[ \pi\Big( d(\theta,\hat\theta_n)\leq \frac{M_\eps \eps_n}{L_n}|\mathbf Y\Big)>1-\alpha \Big]\leq \frac{E_{\theta_0}^{(n)}\left( \pi\big( d(\theta, \hat \theta_n) \leq \frac{\delta_{\eps,\a} \epsilon_n}{\sqrt{\log n}} | \mathbf Y \big) \right)}{1-\alpha}\leq\eps,
\end{align*}
finishing the proof of our statement.
\end{proof}

\subsection{Proof of Lemma \ref{th:prior}} \label{sec:pr:thprior}
For notational convenience let $\epsilon_n =\epsilon_n(k_n)$ and $\Theta_n=\cup_k\Theta_n(k)$. Then in view of Lemma \ref{lem:post:Kn}, for large enough choice of $M>0$

  \begin{equation*}
  \begin{split}
  &E_{\theta_0}^{(n)} \pi\left(d(\theta, \hat \theta_n)\leq  \rho_n\eps_n |\mathbf Y \right)\\
&\quad\leq E_{\theta_0}^{(n)}\Big( \sum_{k\in \mathcal K_n(M)}   \pi_{|k}\big( \{d(\theta, \hat \theta_n)\leq \rho_n\eps_n \}\cap \Theta_n(k) |\mathbf Y \big)\pi_k(k|\mathbf Y)\Big)\\
&\quad\quad +E_{\theta_0}^{(n)}\Big( \sum_{k\in \mathcal K_n(M)}   \pi_{|k}\big(  \Theta_n(k)^c |\mathbf Y \big)\pi_k(k|\mathbf Y)\Big)+\eps
   \end{split}
  \end{equation*}
for all $\theta_0\in\Theta_0$.  Let $\Omega_{n,0}= \left\{m_n(k_n)>  e^{- (c_{3}+c_4+1) n\eps_n^2} \right\}$, with $m_n(k)  = \int_{\Theta(k)} e^{\ell_n(\theta) - \ell_n(\theta_0) }\pi_{|k}(d\theta)$. In view of Lemma 10 of \cite{ghosal:vdv:07} (with $n\eps^2=n\eps_n^2(k_n)\leq 2 k_n\log n$, $k=r$ and $C=1$) we get following \textbf{A1} that $P_{\theta_0}^{(n)}(\Omega_{n,0}^c)\leq (k_n\log n)^{-r/2} =o(1)$. Furthermore, note that for $k\in\mathcal{K}_n(M)$ we have by assumption \textbf{A2} (i) that
\begin{align*}
E_{\theta_0}^{(n)}\pi_{|k}(\Theta_n(k)^c|\textbf{Y})=\pi_{|k}(\Theta_n(k)^c)\lesssim e^{-(c_2+c_3+c_4+2)n\eps_n^2}.
\end{align*}
By combining the above inequalities and in view of Lemma \ref{rem:Kn} we get that
\begin{align*}
E_{\theta_0}^{(n)}\Big( \sum_{k\in \mathcal K_n(M)}   \pi_{|k}\big(  \Theta_n(k)^c |\mathbf Y \big)\pi_k(k|\mathbf Y)\Big)\lesssim k_n e^{-n\eps_n^2}=o(1). 
\end{align*}

Next we show that with probability at least $1-\tilde{C}\eps$, for some universal $\tilde{C}>0$, we have for every $k\in\mathcal{K}_n(M)$
\begin{align}
\pi_{n,k} := \pi_{|k}\left( \{d(\theta, \hat \theta_n)\leq  \rho_n\eps_n  \}\cap \Theta_n(k) |\mathbf Y \right)\leq \eps .\label{def:hyper_post_prob}
\end{align}
Then the statement of the lemma follows by noting that
$$ E_{\theta_0}^{(n)} \pi\left( d(\theta, \hat \theta_n)\leq \rho_n\eps_n  |\mathbf Y \right)\leq (\tilde{C}+1+o(1))\eps+ \sum_{k\in\mathcal{K}_n(M)}\eps \pi_k(k|Y)\leq (\tilde{C}+3)\eps.$$

It remained to prove \eqref{def:hyper_post_prob}. As a first step we introduce the notations, for $C,B>0$
\begin{align}  
\Omega_{n}( C ) =\left\{\max_{k\in\mathcal{K}_n(M)} e^{ C k }\frac{\int_{\Theta(k)} e^{\ell_n(\theta) - \ell_n(\theta^o_{[k]})}\pi_{|k}(d\theta)}{\pi_{|k}\left( d(\theta, \theta^o_{[k]})^2 \leq  k/n\right)}\geq 1\right\},\\
\Gamma_n(B) = \{\max_{k\in \mathcal{K}_n(M)}\sup_{\Theta_n(k) \cap B_k(\hat \theta_{n},  \rho_n\eps_n,d ) }\big(\ell_n(\theta) - \ell_n(\theta_{[k]}^o )-B k\big)<0\}.
\end{align}
Using assumption \textbf{A0} we have with probability greater than $1-\epsilon$, $d(\hat \theta_n, \theta_0)\leq M_\epsilon \epsilon_n$, therefore as soon as $\rho_n \leq 1$, 
$$B_k(\hat \theta_{n},  \rho_n\eps_n,d ) \subset B_k(\theta_0,  (M_\epsilon + 1)\eps_n,d ).$$
Hence in view of assumption  \textbf{A4} (ii) there exists a large enough constant $B_{\eps}>0$ such that 
$\inf_{\theta_0 \in \Theta_0} P_{\theta_0}^{(n)}(\Gamma_n(B_{\eps}) ) \geq 1 - \epsilon$.
  Also note that following from \textbf{A4} (i) and by using the standard technique for lower bound for the likelihood ratio (e.g. Lemma 10 of \cite{ghosal:vdv:07} with $1+C=c_7+1/\sqrt{\eps}$ and $n\eps^2=k$) we have, for any $k\in\mathcal{K}_n(M)$,  with $P_{\theta_0}^{(n)}$-probability bounded from below by $1 - (\epsilon/k)^{r/2}\geq 1-\eps/k$ that
\begin{align}
\int_{\Theta(k)} e^{\ell_n(\theta) - \ell_n(\theta^o_{[k_n]})}\pi_{|k}(d\theta) &\geq e^{-(c_7+1/\sqrt{\epsilon})k}\pi_{|k}\big(S_n(k,c_{7},c_8,r)\big)\nonumber\\
&\geq e^{-(c_7+1/\sqrt{\epsilon}) k }\pi_{|k}\big(B_k(\theta^o_{[k]}, \sqrt{k/n},d)\big),\label{eq: hulp1}
  \end{align}
hence in view of Lemma \ref{rem:Kn}, $P_{\theta_0}^{(n)}\big( \Omega_{n}^c( c_{7}+1/\sqrt{\epsilon} )\big)\leq C\epsilon$. 

Then  we have, on $ \Omega_{n}(c_7+1/\sqrt{\epsilon})\cap \Gamma_n(B_{\eps})$, that for any $k\in\mathcal{K}_n(M)$
\begin{align*}
\pi_{n,k} &\leq  e^{(c_{7}+B_{\eps}+1/\sqrt{\epsilon})k} \frac{ \pi_{|k}\left( \Theta_{n}(k)\cap\{d(\theta, \hat \theta_{n})\leq  \rho_n \epsilon_n  \}\right) }{ \pi_{|k}\left(d(\theta, \theta^o_{[k]}) \leq  \sqrt{k/n}\right)}.
\end{align*}
We recall that Lemma \ref{rem:Kn} implies that $k_n\leq C k$ for all $k\in \mathcal K_n(M)$ and by definition of $\epsilon_n(k_n)$, $n\epsilon_n^2 \leq 2 k_n\log n$. Therefore we have $ \rho_n\epsilon_n \leq \delta_{\eps}\sqrt{2k_n/n}\leq C^{1/2} \delta_{\eps}\sqrt{k/n}$ for all $k\in \mathcal K_n(M)$. In view of assumption  \textbf{A4} (iii) (with $\delta_{n,k}=C^{1/2}\delta_{\eps}$ )
\begin{align}
\pi_{n,k}\lesssim  e^{(c_{7}+B_{\eps}+1/\sqrt{\epsilon}+c_{9}\log(C^{1/2} \delta_{\eps}) )  k}\leq   \eps,\label{eq: UB_hyper_post}
\end{align}
for small enough choice of $\delta_{\eps}>0$ (the choice $\log(\delta_{\eps}) \leq -c_9^{-1}(1/\sqrt{\epsilon} + c_7 + B_\epsilon +\log \epsilon^{-1})-\log C^{1/2} $ is sufficiently small).


 \section*{Acknowledgements}
The authors would like to thank the Associate Editor and the Referees for their useful comments which lead to an improved version of the manuscript.
This work has been partially funded by the Chaire Havas.

\bibliographystyle{apalike}
\bibliography{biblio}

\appendix
\section{Proofs of the Propositions }
 In this section we prove the propositions of Section  3 of  \cite{rousseau:szabo:2015:main}. Before giving the proofs we introduce some additional notations. Throughout the supplementary material, \textbf{A0}-\textbf{A4} denote assumptions \textbf{A0}-\textbf{A4} of  \cite{rousseau:szabo:16:main}. Furthermore we use the abbreviation $\eps_n=\eps_n(k_n)$ in the whole manuscript.  Along the lines we also use the notation $\Phi_k(x_i)=\big(\phi_1(x_i),..,\phi_k(x_i)\big)$ and denote by  $c$ and $C$ global constants whose value may change one line to another.

\subsection{Proof of Proposition \ref{prop: reg}}\label{sec: proof_reg} 
The first assertion is a direct consequence of Theorem \ref{th:coverage} and Theorem \ref{th:EmpBayesCoverage}, hence it is sufficient to verify the corresponding conditions. As a first step we note that the models $\Theta(k)$ are nested.  Also note that for $\theta,\theta'\in\mathbb{R}^k$ we have that 
$$d_n^2(\theta,\theta')=  (\theta-\theta')^T[\frac{1}{n}\Phi_k^T\Phi_k](\theta-\theta'),$$
and therefore in view of assumption \eqref{assump: reg_design_mtx} for all $k\leq K_n$
\begin{align}
C_0^{-1}\|\theta-\theta'\|_2^2\leq d_n(\theta,\theta')^2\leq C_0\|\theta-\theta'\|_2^2.\label{eq: norms}
\end{align}


First we consider condition \textbf{A1}. By easy and standard calculations we get
\begin{align*}
&2 KL(\theta_0,\theta)=V(\theta_0,\theta)= d_n^2(\theta_0,\theta), 
\end{align*}
see for instance \cite{ghosal:vdv:07}, hence by the definition of $\theta_{[k]}^o$ and assertion \eqref{eq: norms}, for every $k_n\leq K_n$
\begin{align*}
&\{KL(\theta_0,\theta)\leq \eps_n^2/2,V(\theta_0,\theta)\leq \eps_n^2  \}=B_{k_n}(\theta_0,\eps_n,d_n)\\
&\qquad\supset B_{k_n}(\theta_{[k_n]}^o,k_n(\log n)/n,d_n)\supset  B_{k_n}(\theta_{[k_n]}^o,C_0^{-1}k_n(\log n)/n,\|.\|_2).
\end{align*}
Condition \textbf{A1} then follows by Lemma \ref{lem: prior small} with $\tilde\theta=\theta_{[k_n]}^o$.

Next, let us define the sieve $\Theta_n(k)$ as
\begin{align*}
\Theta_n(k)\equiv\{\theta\in\mathbb{R}^k:\, \|\theta\|_2\leq C_1\sqrt{k}(n\eps_n^2)^{1/q}\},
\end{align*}
for some sufficiently large constant $C_1>0$ and $q$ given in \eqref{eq: sieve_prior}.
Then assumption \textbf{A2} (i) follows from the second assertion of Lemma \ref{lem: prior small}.
Condition \textbf{A2} (ii) is verified in Corollary 2 on page 149 of \cite{birge:1983}.  For assumption \textbf{A2} (iii) we note that following from \eqref{eq: norms}
for every $0<c_{6}<1$, $c_{5}>0$ and $u^2\geq 2(1/2+ 1/q)k\log n/(c_5n)$, $k\leq K_n$
\begin{align*}
\log N\big(c_{6} u,\Theta_n(k),d_n(.,.)\big)&\leq \log N\big(c_{6} u/C_0,\Theta_n(k),\|.\|_2\big)\\
&\leq k (1/2+ 1/q)\log n  \leq  c_{5} u^2n/2.
\end{align*}

As a next step we prove that condition \textbf{A3} holds. First note that since $g$ is bounded the density corresponding to $\pi_{|k}$ is also bounded from above by $c_{\max}^k$ for some sufficiently large $c_{\max}>0$. Then in view of \eqref{eq: norms}
\begin{align}
\pi_{|k}\big(B_k(\theta_0,J_1\sqrt{k(\log n)/n},d)\big)&\leq 
\pi_{|k}\big(B_k(\theta_{[k]}^o,2J_1\sqrt{k(\log n)/n},d)\big))\nonumber\\
&\leq \pi_{|k}\big(B_k(\theta_{[k]}^o,2J_1C_0\sqrt{k(\log n)/n},\|.\|_2)\big)\nonumber\\
&\leq c_{\max}^k \mbox{Vol}\big( B_k(\theta_{[k]}^o,2J_1C_0\sqrt{k(\log n)/n},\|.\|_2)\big).\label{eq: condA4:reg}
\end{align}
Then using the formula $\mbox{Vol}( B_k(0,r,\|.\|_2))=\pi^{k/2}r^k/ \Gamma(k/2+1)$ for the volume of a $k$-dimensional $\ell_2$-ball we get that the right hand side of the preceding display is bounded from above by $(Ck(\log n)/n)^{k/2}\leq \exp\{-Ck\log n \}\leq \exp\{-CM_0k_n\log n \}$, hence we get the condition for large enough choice of $M_0>0$.

Then we show that assumption \textbf{A4} (i) also holds. For every $\theta\in\Theta(k)$
\begin{align*}
\|\mathbf{Y}-\Phi_k\theta\|_2^2-\|\mathbf{Y}-\Phi_k\theta_{[k]}^o\|_2^2
=\|\Phi_k\theta_{[k]}^o-\Phi_k\theta\|_2^2+2\big(\mathbf{Y}-\Phi_k\theta_{[k]}^o\big)^{T}\big(\Phi_{k}(\theta_{[k]}^o-\theta)\big).
\end{align*}
Besides, $\mathbf{Y}=f_{0,n}+\mathbf{Z}$, where $\mathbf{Z}=(Z_1,Z_2,...,Z_n)^T$ and $f_{0,n}=(f_0(x_1),...,f_0(x_n))^T$, and hence the $P_{\theta_0}^{(n)}$-expected value of the second term on the right hand is zero following from the orthonormality of the $d_n(.,.)$-projection of $f_0$ into the sub-space $\{\Phi_k\theta:\, \theta\in\Theta(k)\}$ and $E_{\theta_0}^{(n)}\mathbf{Z}=0$. Therefore
\begin{align*}
E_{\theta_0}^{(n)}\log\frac{p_{\theta_{[k]}^o}^{(n)}}{p_{\theta}^{(n)}}=nd_n^2(\theta_{[k]}^o,\theta)/2.
\end{align*}
Similarly to the preceding display we can also show that
$ V_{\theta_0}^{(n)}\big(\log\frac{p_{\theta_{[k]}^o}^{(n)}}{p_{\theta}^{(n)}}\big)=d_n^2(\theta_{[k]}^o,\theta), $
resulting in $B_k(\theta_{[k]}^o,\sqrt{k/n},d_n)= \mathcal{S}_n(k,2,1,2)$.

Next we deal with condition \textbf{A4} (ii). Note that for all  $\theta\in\Theta(k)$, by writing $f_{\theta,n} =( f_\theta(x_1), ..., f_\theta(x_n) )^T$, we get by Cauchy-Schwarz
 \begin{equation}
 \begin{split}
\ell_n(\theta)-\ell_n(\theta_{[k]}^o) &= - \frac{ n d_n(\theta_{[k]}^o,\theta)^2 }{ 2 } +  \mathbf Z^{T}(f_{\theta,n}-f_{\theta_{[k]}^o,n}) \\
&\leq - \frac{ n d_n(\theta_{[k]}^o,\theta)^2 }{ 2 } + \| \mathbf Z^T \Phi_k\|_2 \|  \theta_{[k]}^o - \theta\|_2. ,
\label{eq: UB:loglikeratio}
\end{split}
\end{equation}
Given that $ \| \mathbf Z^T \Phi_k\|_2 $ is increasing in $k$, in view of Lemma \ref{rem:Kn} of \cite{rousseau:szabo:16:main}.
$$\max_{ k \in \mathcal K_n(M) } \|  \mathbf Z^T \Phi_k\|_2 = \|  \mathbf Z^T \Phi_{2M^2k_n}\|_2 \leq C'\sqrt{k_n n}, $$ for some large enough $C'>0$ with probability going to 1. Since 
$ \|  \theta_{[k]}^o - \theta\|_2 \leq C_0d_n(\theta_{[k]}^o,\theta)$, we obtain that with probability tending to 1, uniformly over $k\in \mathcal K_n(M)$,
$$
\ell_n(\theta) - \ell_n(\theta_{[k]}^o) \leq \sqrt{n} d_n(\theta_{[k]}^o,\theta) \left( C'C_0\sqrt{k_n} - \sqrt{n} d_n(\theta_{[k]}^o,\theta)/2\right) \leq 2C_0^2C'^2k_n,$$
terminating the proof of condition \textbf{A4} (ii).

Finally, to prove condition \textbf{A4} (iii), we note that following from \eqref{eq: norms}, for every $\theta\in\Theta(k)$, $k\leq K_n$
\begin{align}
\frac{\pi_{|k}\big( B_k( \theta,\delta_{n,k} \sqrt{k/n},d_n)\big)}{\pi_{|k}\big(B_k(\theta_{[k]}^o,\sqrt{k/n},d_n) \big)}
&\leq \frac{\pi_{|k}\big( B_k(\theta,C_0 \delta_{n,k} \sqrt{k/n},\|.\|_2) \big)}{ \pi_{|k}\big( B_k(\theta_{[k]}^o,C_0^{-1} \sqrt{k/n},\|.\|_2)\big)}.\label{eq: regr_A5}
\end{align} 
Since $\|\theta_{[k]}^o \|_2 \leq C_0\|\theta_0\|_2  $, we know that in $B_k(\theta_{[k]}^o,C_0^{-1} \sqrt{k/n},\|.\|_2)$ the prior density is bounded from below by $c_{\min}^{k}$ for some sufficiently small $c_{\min}>0$. Recall also that the prior density is bounded from above by $c_{\max}^k$, for some sufficiently large $c_{\max}>0$, for $\theta\in\Theta(k)$. Hence the right hand side of the preceding display is bounded from above by $(c_{\max}/c_{\min})^k$ times the fraction of the volumes of the $\ell_2$-balls with radius  $C_0\delta_{n,k} \sqrt{k/n}$ and $C_0^{-1} \sqrt{k/n}$, respectively. Using again the formula for the volume we get that the right hand side of the preceding display is bounded from above by $(C_0^2 c_{\max}/c_{\min})^k \delta_{n,k}^{k}$, finishing the proof of our statement.


It remains to deal with the second assertion of the proposition. Let us introduce first the notation $\theta_{0,[k]}=(\theta_{0,1},\theta_{0,2},...,\theta_{0,k})\in\mathbb{R}^k$. Then for $\theta_0\in S^{\beta}(M)$ with $\beta\geq\beta_0>1/2$ we have that $d_n(\theta_0,\theta_{0,[K_n]}) \leq \|\Delta_{K_n}\|_\infty\lesssim \sum_{i=K_n+1}^{\infty}|\theta_{0,i}|\lesssim K_n^{-(\beta-1/2)}$, where $\Delta_k=f_{\theta_0}-\sum_{j=1}^k\theta_{0,j}\phi_j$. Therefore by triangle inequality and assertion \eqref{eq: norms}
\begin{align}
b(k)^{1/2}&\leq d_n(\theta_{0,[k]},\theta_{0,[K_n]})+ d_n(\theta_{0},\theta_{0,[K_n]})\nonumber\\
&\lesssim \big(\sum_{i=k+1}^{K_n}\theta_{0,i}^2\big)^{1/2}+K_n^{-(\beta-1/2)}\lesssim k^{-\beta}+K_n^{-(\beta-1/2)}.\label{eq: UBforBias_reg}
\end{align}
Hence by taking $\bar{k}_n=C(n/\log n)^{1/(1+2\beta)}$ (with large enough constant $C>0$) and in view of assumption $K_n\geq n^{\frac{\beta_0}{(1+2\beta_0)(\beta_0-1/2)}}$ we get that $b(\bar{k}_n)\leq \bar{k}_n\log n/ n$ and as a consequence $k_n\leq \bar{k}_n$.
This leads to $k_n=o( n^{1/(1+2\beta_0)})$ and
\begin{align*}
\eps_n^2\leq 2 k_n(\log n)/n\leq 2 \bar{k}_n(\log n)/n\lesssim  (n/\log n)^{-2\beta/(1+2\beta)},
\end{align*}
finishing the proof of the proposition.

\begin{lemma}\label{lem: prior small}
For every $c>0$, $k\in\{1,...,n\}$,  $\tilde\theta\in\Theta(k) $ satisfying $|\tilde{\theta}_i|\leq C$ for every $i=1,...,k$, 
\begin{align*}
\pi_{|k}( \|\theta-\tilde\theta\|_2\leq c\sqrt{k(\log n)/n})\gtrsim e^{-Ck\log n},\\
\pi_{|k}\big( \|\theta\|_2\geq  C_1\sqrt{k}(n\eps_n^2)^{1/q}\big) \leq e^{-CC_1^qn\eps_n^{2}}.
\end{align*}
\end{lemma}
\begin{proof}
The proof of the first and second assertions are basically given for instance as part of the proof of Lemma 3 of \cite{arbel} and the proof of Theorem 2.1 of \cite{rivoirard:rousseau:12}, but for completness we give a sketch of the proof here as well.

First of all note that by assumption \eqref{eq: sieve_prior} we have
\begin{align*}
\pi_{|k}\big(\|\theta-\tilde\theta\|_2^2\leq ck(\log n)/n\big)&=\int_{\|\theta-\tilde\theta\|_2^2\leq ck(\log n)/n}\prod_{j=1}^k g(\theta_j)d\theta\\
&\geq G_1^k \int_{\|\theta-\tilde\theta\|_2\leq cn^{-1/2}} \prod_{j=1}^k e^{-G_2|\theta_j|^{q}}d\theta,
\end{align*}
for any $k\geq 1$. Note that $|\theta_j|^{q}\leq 2^{q} ( |\tilde\theta_j|^{q}+|\tilde\theta_j-\theta_{j}|^q)$. Distingushing the cases $q\geq 2$ and $q<2$ and using H\"older's inequality in the latter, one can easily derive that $\sum_j |\tilde\theta_j-\theta_{j}|^q\leq  k\|\tilde\theta-\theta\|_2^{q}\leq  c^{q}k n^{-q/2} $, see for instance page 28 of \cite{arbel}. 
Also note that the volume of a $k$-dimensional ball with radius $cn^{-1/2}$ is bounded from below by a multiple of $e^{-Ck\log n}$ and $\sum_{j=1}^{k}|\tilde\theta_j|^{q}\leq C^qk$, hence the right hand side of the preceding display is also bounded from below by a multiple of $e^{-Ck\log n}$.

Let $w_n= C_1\sqrt{k}(n\eps_n^2)^{1/q}$. For the second assertion note that for sufficiently large $n$
\begin{align*}
\pi_{|k}(\theta\in\Theta(k):\, \|\theta\|_2\geq w_n)&\leq \sum_{i=1}^{k}\pi_{|k}( \theta_i^2\geq  w_n^2/k)\\
&\leq k \Big(\int_{w_n/\sqrt{k}}^{\infty}+\int_{-\infty}^{-w_n/\sqrt{k}}\Big)g(x)dx\\
&\leq k G_3 \int_{w_n/\sqrt{k}}^{\infty}e^{-G_4 x^q}dx\\
&\leq C e^{-C' (w_n/\sqrt{k})^q},
\end{align*}
where the last inequality follows from
\begin{align*}
\int_{y}^{\infty}e^{-cx^q}dx = q^{-1}\int_{y^q}^{\infty} e^{-cz}z^{\frac{1-q}{q}}dz
\leq (cq)^{-1} e^{-cy^q}y^{1-q}\lesssim  e^{-(c/2)y^q}.
\end{align*}

\end{proof}

\subsection{Proof of Proposition \ref{prop: UB:radius}}\label{sec: UB:radius}
We show below that 
\begin{align}
\mbox{limsup}_n \sqrt{n/\hat{k}_n}\diam \big(\hat{C}_n(1,\alpha),d_n\big)< +\infty ,\quad \text{$P_{f_0}$-almost surely}.\label{eq: UB:radius}
\end{align}
Furthermore we know that in view of assertion \eqref{eq: MMLE} (see the proof of Theorem \ref{th:EmpBayesCoverage}) and Lemma \ref{rem:Kn} of \cite{rousseau:szabo:2015:main} we have that 
$\inf_{\theta_0\in\Theta_0} P_{\theta_0}(C^{-1}k_n\leq \hat{k}_n\leq C k_n)\geq 1-\eps$ holds under the polished tail condition.
Finally note that in view of the monotonicity of $b(k)$ we have with probability larger than $1-\eps$ for every $f_0\in \mathcal{H}_s^{\beta}(L)$ that
\begin{align*}
d_n(\theta_0,\hat\theta_{\hat{k}_n})&\geq b(\hat{k}_n)^{1/2}\geq b(Ck_n)^{1/2}\gtrsim \|\theta_0-\theta_{0,[Ck_n]}\|_2\\
&\geq L_0^{-1} \big(\sum_{i=Ck_n}^{\infty}i^{-1-2\beta}\big)^{1/2}\gtrsim k_n^{-\beta}\gtrsim \sqrt{k_n(\log n)/n}.
\end{align*}
Therefore we can conclude that for any $m_n = o(1) $, combining the above inequality with \eqref{eq: UB:radius} implies that for all $\epsilon>0$, when $n$ is large enough
\begin{align*}
\inf_{f_0\in \mathcal{H}_s^{\beta}(L)}P_{f_0}^{(n)}\Big( d_n(\theta_0,\hat\theta_{\hat{k}_n})\geq  m_n \sqrt{\log n}r_{\alpha}\Big)\geq 1-\eps,
\end{align*}
which proves Proposition \ref{prop: UB:radius} .

It remains to prove assertion \eqref{eq: UB:radius}. Note that the posterior distribution can be written in the form
\begin{align*}
\pi_{|\hat{k}_n}(\theta|\textbf{Y})\propto \exp\Big\{\sum_{j=1}^{\hat{k}_n}\Big( -n\theta_j^2+\log g(\theta_j)+2\sum_{i=1}^n Y_i\phi_j(x_i) \theta_j\Big)\Big\},
\end{align*}
hence due to Chebyshev's inequality it is sufficient to verify that the random variable with density function proportional to $z\mapsto \exp\big(  -nz^2+\log g(z)+2\sum_{i=1}^n Y_i\phi_j(x_i) z\big)$ has variance bounded from above by a multiple of $1/n$. By elementary conjugate computations one can see that this holds for $g$ equal to the normal distribution with fixed parameters $\mu\in\mathbb{R}$ and $\sigma>0$ and the Laplace distribution with parameters $\mu\in\mathbb{R}$ and $b>0$.

\subsection{Proof of Proposition \ref{prop: UB:radius}}\label{sec: UB:radius}
We show below that 
\begin{align}
\mbox{limsup}_n \sqrt{n/\hat{k}_n}\diam \big(\hat{C}_n(1,\alpha),d_n\big)< +\infty ,\quad \text{$P_{f_0}$-almost surely}.\label{eq: UB:radius}
\end{align}
Furthermore we know that in view of assertion \eqref{eq: MMLE} (see the proof of Theorem \ref{th:EmpBayesCoverage}) and Lemma \ref{rem:Kn} of \cite{rousseau:szabo:2015:main} we have that 
$\inf_{\theta_0\in\Theta_0} P_{\theta_0}(C^{-1}k_n\leq \hat{k}_n\leq C k_n)\geq 1-\eps.$
Finally note that in view of the monotonicity of $b(k)$ we have with probability larger than $1-\eps$ for every $f_0\in \mathcal{H}_s^{\beta}(L)$ that
\begin{align*}
d_n(\theta_0,\hat\theta_{\hat{k}_n})&\geq b(\hat{k}_n)^{1/2}\geq b(Ck_n)^{1/2}\gtrsim \|\theta_0-\theta_{0,[Ck_n]}\|_2\\
&\geq L_0^{-1} \big(\sum_{i=Ck_n}^{\infty}i^{-1-2\beta}\big)^{1/2}\gtrsim k_n^{-\beta}\gtrsim \sqrt{k_n(\log n)/n}.
\end{align*}
Therefore we can conclude that for any $m_n = o(1) $, combining the above inequality with \eqref{eq: UB:radius} implies that for all $\epsilon>0$, when $n$ is large enough
\begin{align*}
\inf_{f_0\in \mathcal{H}_s^{\beta}(L)}P_{f_0}^{(n)}\Big( d_n(\theta_0,\hat\theta_{\hat{k}_n})>  m_n \sqrt{\log n}r_{\alpha}\Big)\geq 1-\eps,
\end{align*}
which proves Proposition \ref{prop: UB:radius} .

It remains to prove assertion \eqref{eq: UB:radius}. Note that the posterior distribution can be written in the form
\begin{align*}
\pi_{|\hat{k}_n}(\theta|\textbf{Y})\propto \exp\Big\{\sum_{j=1}^{\hat{k}_n}\Big( -n\theta_j^2+\log g(\theta_j)+2\sum_{i=1}^n Y_i\phi_j(x_i) \theta_j\Big)\Big\},
\end{align*}
hence due to Chebyshev's inequality it is sufficient to verify that the random variable with density function proportional to $z\mapsto \exp\big(  -nz^2+\log g(z)+2\sum_{i=1}^n Y_i\phi_j(x_i) z\big)$ has variance bounded from above by a multiple of $1/n$. By elementary conjugate computations one can see that this holds for $g$ equal to the normal distribution with fixed parameters $\mu\in\mathbb{R}$ and $\sigma>0$ and the Laplace distribution with parameters $\mu\in\mathbb{R}$ and $b>0$.

\subsection{Proof of Proposition \ref{prop: UB:radius:HB}} \label{sec: UB:radius:HB}

In this section we prove that inflation by a term going to infinity is also necessary in the case of hierarchical Bayes posteriors, in the context of nonparametric regression. 
We assume that 
 $$ \Phi_{K_n}^T \Phi_{K_n} = I_{K_n} , \quad K_n = n/\log n .$$
We also assume that the parameter $\theta_0$ belongs to the hyper-rectangle $\mathcal H_s^\beta(L_0)$ with $\beta >1$ and verifies that
\begin{equation}\label{cond:subHR}
\exists r_\infty \in [1/L_0, L_0]; \, \quad \mbox{s.t.} \,  \lim_{i \rightarrow \infty} i^{(2\beta+1)}\theta_{0,i}^2 =r_\infty^2.
 \end{equation}
From Proposition \ref{prop: reg} in the main document 
we then have that the posterior concentration rate is of order $\epsilon_n \lesssim  (n/\log n)^{-\beta/(2\beta+1)} $, that $\mathcal K_n(M) \subset [m_1 k_n, m_2 k_n]$ with $k_n = (n/\log n)^{1/(2\beta+1)} $ for some $m_1,m_2>0$ and for all $\| \theta - \theta_0\|_2  \lesssim \epsilon_n = o(1/\sqrt{k_n}) $, $\theta \in \cup_{k \in \mathcal K_n(M)} \Theta_k$, the log-likelihood is equal to
\begin{equation}\label{like:simp}
\begin{split}
\ell_n(\theta ) - \ell_n(\theta_0)  &= -\frac{  n \|\theta - \theta_{0, [K_n]} \|_2^2 }{ 2 }  + \sqrt{n} \langle Z, \Phi_k \theta -\Phi_{K_n} \theta_0\rangle + o_p(1)   \\
& = -\frac{  n \|\theta - \theta_{0} \|_2^2 }{ 2 }  +  \sqrt{n}\langle Z, \Phi_k( \theta - \theta_{[k]}^o) \rangle  +\sqrt{n} \langle Z , P_k^\perp \theta_0\rangle + o_p(1)   \\
\end{split} 
\end{equation}
where $P_k^\perp = I - \Phi_k  \Phi_k^T $ is the projection on the orthogonal of the space generated by $\Phi_k$ and $\theta_{[k]}^o $ is the orthogonal projection of $f_{\theta_0}$ on $ \Phi_k$ and $Z = Y - f_{0,n}$. We also denote by $P_k =  \Phi_k  \Phi_k^T $ the orthogonal projection on $ \Phi_k$. 
We have that 
$$\| \theta_{[k]}^o -\theta_{0,[k]} \|_2^2 \leq \|\theta_0 - \theta_{0, [K_n]}\|_2^2 \lesssim (n/\log n)^{-2\beta} =   o(1/n). $$
Then  for all $k \in \mathcal K_n(M)$,  since $\log g$ is continuously differentiable, 
\begin{equation*}
\pi_k(k| Y)  \propto \pi_k(k) \prod_{j=1}^k g(\theta_{0,j}  ) e^{ \frac{  - n \|\theta_{0,[k]} - \theta_{0} \|_2^2 + \|\Phi_k^T Z\|_2^2 }{ 2 }    + \sqrt{n}\langle Z , P_k^\perp \theta_0\rangle  - k \log (\frac{n}{2\pi})/2  +o_p(1) }.
\end{equation*} 

To control the radius of the credible balls we provide tighter bounds for the concentration of the hyper-posterior $\pi_k(.|Y)$, i.e. we show  that there exists  $(k_n^*)^{2/3}\lesssim h_n  = o(k_n^*)$, satisfying  $ \ell(k_n^*) \leq h_n /k_n^*$, such that 
\begin{align}
\pi_k ( k \in (k_n^* - h_n, k_n^*+h_n) | Y) = 1 + o_p (1) .\label{eq: hyper:post:conc:HBcase}
\end{align}
Then in view of \eqref{like:simp}, we also have that uniformly on $k \leq  m_2 k_n$, 
$$ \mathbb E( \theta | Y, k) = \theta_{0,[k]} - \frac{ P_k Z}{\sqrt{n} } + o_p(1/\sqrt{n})   .$$ 
Let $\hat \theta $ be the posterior mean, then 
$$1 - \alpha \leq \Pi( \| \theta - \hat \theta\|_2^2 \leq r_\alpha^2 |Y) \leq \frac{ \sum_{k= k_n^*-h_n}^{k_n^* + h_n} \pi_k(k| Y) [ \|\hat \theta -  \mathbb E( \theta | Y, k)\|_2^2 + \frac{ k}{n}  + o(k/n) ] }{ r_\alpha^2} 
$$ 
so that 
 \begin{equation*}
\begin{split} 
r_\alpha^2 \lesssim  \frac{1}{1-\alpha}[ 2 h_n (k_n^*)^{-2\beta-1} + \frac{ k_n^* }{n} ] \lesssim (k_n^*)^{-2\beta}\left( \frac{h_n}{  k_n^*} + \frac{ 1 }{ \log n} \right) 
\end{split} 
\end{equation*} 
and it is necessary to inflate by a term going to infinity (as the squared bias is bounded from below by $k_n^{-2\beta}\asymp (k_n^*)^{-2\beta}$). This term depends on convergence rate of $r_i^2 $ to $r_\infty^2$.

It remained to prove assertion \eqref{eq: hyper:post:conc:HBcase}. Before proceeding to the proof we introduce several new notations. Define $$L_n (k)  =  n \|\theta_{0,[k]} - \theta_{0} \|_2^2  + k \log \left(\frac{n}{2\pi e g(0)^2 } \right) - 2 \log \pi_k(k),$$ 
  where $2\log \pi_k(k) = -  a k - b k \log (k) + a'\log k + O(1) $,  where $b=a' = 0 $ in the case of a Geometric distribution and $a = -2(\log \lambda +1) $, $b = 2$, $a' = 1$ if $\pi_k $ is the  Poisson distribution. Write $\theta_{0,i}^2 = r_i^2i^{-2\beta-1} $ 
then $r_i \rightarrow r_\infty$, as $i$ goes to infinity; define    $c    = e^{a} /(2\pi e g(0)^2)$ and 
$$ k_n^* =\left\lfloor \left(\frac{ n r_\infty^2(2\beta+1) }{ (b+2\beta+1)\log n + (2\beta+1)\log c+ b\log r_\infty^2 -b\log \log (nc) } \right)^{1/(2\beta+1)}\right\rfloor,$$  
when $n$  is large enough so that it exists. Then 
\begin{equation} \label{eq:kn}
 nr_\infty^2(k_n^*)^{-(2\beta+1)}  = \log(nc) + b\log(k_n^*) + b + O( \log n/k_n^*). 
 \end{equation}

First  let $k > k_n^*$,  write $k = k_n^* + h$ with $h>0$, then   
\begin{align*}
\frac{ \pi_k(k| Y) }{  \pi_k(k_n^*| Y) } 
 & = e^{ -\frac{  L_n(k) - L_n (k_n^*) }{2} }  e^{ \sqrt{n} \langle Z , P_k P_{k_n^*}^{\perp}\theta_0\rangle +  \frac{\|P_k P_{k_n^*}^{\perp} Z\|^2 - h }{2}}  ( 1 + o_p(1))\\
 &=  e^{ \frac{   n \|\theta_{0,[k_n^*]} - \theta_{0,[k_n^*+h]} \|_2^2 - h  \log (n c ) - bh \log (k_n^*)-  [b (k_n^*+h) + a'] \log ( 1 + h/k_n^*)  }{ 2 }    +\sqrt{n} \langle Z , P_k P_{k_n^*}^{\perp}\theta_0\rangle} \\
&\qquad\times  e^{ \frac{\|P_k P_{k_n^*}^{\perp} Z\|^2 - h }{2}}  ( 1 + o_p(1)). \\
\end{align*} 
Note that we have
\begin{align*}
\|\theta_{0,[k_n^*]} - \theta_{0,[k_n^*+h]} \|_2^2 &= \sum_{i=1}^h r_{k_n^*+ i}^2 (k_n^*+ i)^{-2\beta-1}\\
&= (k_n^*)^{-2\beta-1}r_\infty^2 \sum_{i=1}^h \frac{r_{k_n^*+ i}^2}{ r_\infty^2 } (1+ i/k_n^*)^{-2\beta-1}.
\end{align*}
Furthermore, set 
\begin{align*}
\delta_{n,h}  &:=  r_{\infty}^{-2}\frac{ \sum_{i=1}^h (r_{k_n^*+i}^2- r_{\infty}^2)(1 + i/k_n^*)^{-2\beta-1}}{ \sum_{ i = 1}^{h}  (1 + i/k_n^*)^{-2\beta-1} }
\end{align*}
then, for $h = o(k_n^*)$,  
\begin{align*}
|\delta_{n,h} |
 &\lesssim \frac{ \sum_{i=1}^h \ell(k_n^*+i)}{ h } + o ( h/k_n^*)  = \ell(k_n^*)+ o ( h/k_n^*),
\end{align*}
since $\ell $ is slowly varying. This leads to, for all $h = o(k_n^*)$, 
\begin{equation*}
\begin{split} 
   - L_n(k) + L_n (k_n^*) 
 & = 
 n (k_n^*)^{-2\beta-1} r_{\infty}^2(1 + \delta_{n,h} ) \sum_{ i = 1}^{h}  (1 + i/k_n^*)^{-2\beta-1} - h \log (n c )\\
&\qquad - bh \log (k_n^*)-  [b(k_n^*+h)-a'] \log ( 1 + h/k_n^*)  \\
 & =  [\log(nc) + b\log(k_n^*)+b ] \left( (1 + \delta_{n,h})\sum_{ i = 1}^{h}  (1 + i/k_n^*)^{-2\beta-1} -h \right) \\
&\qquad - \frac{ b h^2}{ 2k_n^*} + O\left( \frac{h \log n  }{k_n^*}+ \frac{h^3}{(k_n^*)^2} \right)  + o(1) \\
& = [\log(nc) + b\log(k_n^*) ]h  \left(\delta_{n,h}  (1 + o(1)) - (2\beta+1) \frac{ h }{ 2k_n^*}  \right)\\
&\qquad  +  O\left( \frac{h^2 }{k_n^*}\right) +o(1).
 \end{split} 
 \end{equation*}
So that, choosing  
$$ (2\beta+1) \frac{ h }{ 2k_n^*} \geq 2\delta_{n,h},  $$
leads to
\begin{equation*}
\frac{ \pi_k(k| Y) }{  \pi_k(k_n^*| Y) } \leq   e^{-\frac{ (2\beta+1) [\log(nc) + b\log(k_n^*) ]h^2 }{4 k_n^* }  +   \sqrt{n}\langle Z , P_k P_{k_n^*}^{\perp}\theta_0\rangle  }   e^{ \frac{\|P_k P_{k_n^*}^{\perp} Z\|^2 - h  }{2}}( 1 + o_p(1)).  
\end{equation*} 
With probability going to 1, for all $h \lesssim k_n^*$, 
$$\sqrt{n}|\langle Z , P_k P_{k_n^*}^{\perp}\theta_0\rangle | \lesssim \sqrt{n} (k_n^*)^{-\beta-1/2} \sqrt{h} \sqrt{\log n} \lesssim \sqrt{h} \log n = o(h^2 \log n/ k_n^*)$$ 
as soon as $ (k_n^*)^{2/3} = o(h) $. In this case, 
\begin{equation*}
\frac{ \pi_k(k| Y) }{  \pi_k(k_n^*| Y) } \leq   e^{- \frac{ h^2 C \log n  }{ k_n^*}} , \quad \text{when } \quad  h \geq\max\left(  \frac{  4 \ell(k_n^*)  k_n^*}{(2\beta+1 )} , z_n(k_n^*)^{2/3} \right), \, z_n \rightarrow +\infty.
\end{equation*} 
Also, if $h \geq \delta k_n^*$, then there exists $\tau >0$ such that 
$$ \sum_{ i = 1}^{h}  (1 + i/k_n^*)^{-2\beta-1} \leq (1-\tau)h $$ and for $n$ large enough
\begin{equation*}
\begin{split} 
 - L_n(k) + L_n (k_n^*) \leq  - \frac{ \tau h [\log(nc) + b\log(k_n^*) ]}{2} ,
 \end{split} 
 \end{equation*}
 and the same result holds. 
Let $k = k_n^*- h $, with $0<h = o(k_n^*)$,  then neglecting all terms of order $o(1)$, 
\begin{align*}
- L_n(k) + L_n (k_n^*) & = - n \|\theta_{0,[k_n^*]} - \theta_{0,[k]} \|_2^2 + (k_n^*-k)  \log (n c ) + b( k_n^*- k) \log (k_n^*)\\
&\qquad+  [bk-a'] \log ( k_n^*/k) \\ 
& = -[\log(nc) + b\log(k_n^*) + b](1 + \delta_{n,h}) \sum_{ i = 1}^{h}  (1 - i/k_n^*)^{-2\beta-1} \\
& \qquad+h\log (n c ) +bh \log (k_n^*) -b(k_n^*-h) \log ( 1 - h/k_n^*)\\
 & \leq  -[\log(nc) + b\log(k_n^*)] h \left( (2\beta+1) \frac{ h }{k_n^*} +  \delta_{n,h}(1+o(1))\right) + O\left(\frac{ h^2 }{2k_n^*}  \right)
 \end{align*}
and similarly to before if $h(2\beta+1) \geq -2\delta_{n,h} k_n^*$ then  
 \begin{equation*}
\begin{split} 
- L_n(k) + L_n (k_n^*) \leq -c [\log(nc) + b\log(k_n^*)]h^2 /k_n^*
\end{split} 
 \end{equation*}
for some $c >0$. 
This proves that for all  $(k_n^*)^{2/3}\lesssim h_n  = o(k_n^*)$ such that $ \ell(k_n^*) \leq h_n /k_n^*$ 
assertion \eqref{eq: hyper:post:conc:HBcase} holds.


\subsection{Proof of Proposition \ref{prop: hist}}\label{sec: hist} 
As a first step we introduce some additional notation which will be used throughout the proof. Let us denote by $n_j$ the number of observations falling into the $j$th bin $I_j=[(j-1)/k,j/k)$, $1\leq j\leq k$, let 
$\theta_{[k]}^o = ( \int_{I_j}p_0(x) dx, j \leq k ) \in \Theta(k)=\mathcal{S}_k$ and we use the abbreviation $\theta_j^o=\theta_{[k],j}^o$ for the $j$th coefficient of the vector $\theta_{[k]}^o$. Note that $c_0/k\leq\theta_{j}^o\leq C_0/k$ following from $c_0\leq p_0(x)\leq C_0$. Let $\rho_n=\delta/\sqrt{\log n}$ for some sufficiently small $\delta>0$ and $h(.,.)$ denote the Hellinger distance.

First note that by the mean value theorem for $p_0\in\mathcal{H}^{\beta}(L_0)$ we have $\|p_0-p_{\theta_{[k]}^o}\|_{\infty}\lesssim k^{-\beta}$. This, combined with \eqref{Hell:Hist} (with $\theta=\theta_{[k]}^o$), and the inequality $\|p_0-p_{\theta_{[k]}^o}\|_2\leq \|p_0-p_{\theta_{[k]}^o}\|_{\infty}$ implies that $h(p_0,p_{\theta_{[k]}^o})\lesssim k^{-\beta}$. Then it is easy to see (using similar arguments to the one below \eqref{eq: UBforBias_reg}) that $k_n\lesssim (n/\log n)^{1/(1+2\beta)}=o(n^{1/(1+2\beta_0)})$ and $\eps_n\lesssim (n/\log n)^{-\beta/(1+2\beta)}$.

Next we deal with the first assertion of the proposition. Since the density $p_0$ is bounded from below and above by some positive constants, the Kullback-Leibler divergence and second moment of the likelihood ratio are both bounded by the square Hellinger distance, establishing condition \textbf{A1}, see Lemma 8.2 of \cite{ghosal:ghosh:vdv:00}.

Next note that by taking $\Theta_n(k) = \Theta(k)$ condition \textbf{A2} (i) automatically holds. Then for all $u$ such that $u^2 \gtrsim k/n$ the entropy condition \textbf{A2} (iii) is verified with $d(.,.)$ the Hellinger metric, see for instance the proof of Proposition 3.6 in \cite{rousseau:szabo:15:supp}. The existence of tests , i.e. condition \textbf{A2} (ii),  with respect of the Hellinger distance  (with $c_6 =1/18$) follows from see \cite{birge:83}.
 
For condition \textbf{A4} (i) we note that in view of Lemma \ref{lem: help_hist} (with $\mu_n=1$ and $\tilde\theta=\theta_{[k]}^o$) for $\theta\in B_k(\theta_{[k]}^o,\sqrt{k/n},h)$  by Taylor's series expansion we get that
\begin{equation}\label{eq: hist:UBforKL}
\begin{split}
\int p_{0}\log \left( \frac{ p_{\theta_{[k]}^o} }{ p_\theta } \right) & = \sum_{j=1}^k \theta_{j}^o \log ( \theta_{j}^o/ \theta_j)  \leq k c_0^{-1}\sum_{j=1}^k (\theta_{j}^o- \theta_j)^2 \leq \frac{ 9C_0 k}{ c_0n },\\
\int  p_{0}\log^2 \left( \frac{ p_{\theta_{[k]}^o} }{ p_\theta } \right)   & \lesssim \frac{ 9C_0 k}{ c_0n },
\end{split}
\end{equation}
resulting in $B_k(\theta_{[k]}^o,\sqrt{k/n},h)\subset \mathcal{S}(k,9C_0/c_0,9CC_0/c_0,2)$.

Next we verify \textbf{A4} (ii).  Let $n_j = \sum_{i=1}^n \1_{Y_i \in I_j} $ and note that
\begin{equation*}
\begin{split}
\ell_n(\theta)-\ell_n(\theta_{[k]}^o)
&= \sum_{j=1}^k(n_j-n\theta_{j}^o)(\log \theta_j-\log \theta_{j}^o) + n \sum_{j=1}^k\theta_{j}^o(\log \theta_j-\log \theta_{j}^o) \\
&\leq \left( \sum_{j=1}^k \frac{(n_j-n\theta_{j}^o)^2}{\theta_j^o} \right)^{1/2}  \left(\sum_{j=1}^k \theta_j^o (\log \theta_j-\log \theta_{j}^o)^2 \right)^{1/2}\\
&\qquad - n \sum_{j=1}^k\theta_{j}^o(\log \theta_j^o-\log \theta_{j}).
\end{split}
\end{equation*}
Note that in view of \eqref{Hell:Hist} (with $\theta=\theta_{[k]}^o$) and Lemma \ref{lem: histogram_projection} we have that $h^2(p_0,p_{\theta_{[k]}^o})\leq C'\inf_{\theta\in\mathbb{R}^k}h^2(p_0,p_{\theta})=C' b(k)$.
Therefore in view of  Lemma \ref{lem: help_hist} (with $\tilde\theta=\theta_{[k]}^o$ and $\mu_n$ taken to be a large enough constant) for $\theta\in B_k\big(\theta_0,(M_{\eps}+1)\eps_n(k_n), h\big)\subset B_k\big(\theta_{[k]}^o,\tilde{C}\sqrt{k_n(\log n)/n}, h\big)$, $k\in\mathcal{K}_n(M)$ (for some large enough $\tilde{C}>0$) we have $c_0/(2k) \leq \theta_j \leq 4C_0/k$ and using Taylor series expansion of $\log\theta_j$ around $\theta_j^o$ for every $j=1,...,k$,
$$\sum_{j=1}^k \theta_j^o (\log \theta_j-\log \theta_{j}^o)^2 \leq \sum_{j=1}^k\frac{\theta_j^o (\theta_j- \theta_{j}^o)^2}{ 2( \theta_j^o \wedge \theta_j)^2} \leq \frac{ 2C_0k\|\theta_{[k]}^o - \theta\|_2^2 }{ c_0^2} $$
and 
$$\sum_{j=1}^k\theta_{j}^o(\log \theta_j^o-\log \theta_{j})\geq \sum_{j=1}^k\frac{\theta_j^o (\theta_j- \theta_{j}^o)^2}{ 2( \theta_j^o \vee \theta_j)^2}  \geq \frac{ c_0k}{ 2^5C_0^2} \|\theta_{[k]}^o - \theta\|_2^2.$$
So that there exist $C_1, C_2>0$ such that
\begin{equation*}
\begin{split}
\ell_n(\theta)-\ell_n(\theta_{[k]}^o)
&\leq \left( \sum_{j=1}^k \frac{(n_j-n\theta_{j}^o)^2}{\theta_j^o} \right)^{1/2}  \sqrt{k} C_1 \|\theta_{[k]}^o - \theta\|_2 - n k C_2\|\theta_{[k]}^o - \theta\|_2^2 .
\end{split}
\end{equation*}
We show below that for all $\epsilon>0$ there exists $B_\epsilon>0$ such that 
\begin{equation}\label{control:stoch}
 P_{p_0}^{(n)}\left( \sup_{k\in\mathcal{K}_n(M)}\sum_{j=1}^k (n_j-n\theta_{j}^o)^2 > nB_\epsilon \right) \lesssim \eps.
 \end{equation}
Then on the event  $\sum_{j=1}^k (n_j-n\theta_{j}^o)^2 \leq nB_\epsilon$, for all $k\in\mathcal{K}_n(M)$
\begin{equation*}
\ell_n(\theta)-\ell_n(\theta_{[k]}^o)
\leq \sqrt{nk} \|\theta_{[k]}^o - \theta\|_2 \left(  \sqrt{B_\epsilon}   C_1  - \sqrt{nk}C_2\|\theta_{[k]}^o - \theta\|_2\right)\leq B_\epsilon C_1^2/C_2.
\end{equation*}

We now prove \eqref{control:stoch}. First we note that in view of Lemma \ref{rem:Kn} of \cite{rousseau:szabo:16:main} it is sufficient to show for every $k\in\mathcal{K}_n(M)$
\begin{align*}
 P_{p_0}^{(n)}\left( \sum_{j=1}^k (n_j-n\theta_{j}^o)^2 > nB_\epsilon \right) \leq C\frac{\eps}{k}.
\end{align*}
By the properties of the categorical random variable
\begin{equation*}
\sum_{j=1}^k E_{p_0}^{(n)}(n_j-n\theta_{j}^o)^2=n\sum_{j=1}^k\theta_{j}^o(1-\theta_{j}^o) \leq n.
\end{equation*}
Using Lemma \ref{lem:Variance:histo}, for $k\in\mathcal{K}_n(M)$ and Chebyshev's inequality, if $ B_\epsilon>1$
\begin{equation*}
P_{p_0}^{(n)}\left( \sum_{j=1}^k (n_j-n\theta_{j}^o)^2 > n B_\epsilon  \right) \lesssim \frac{ 1 }{ k (B_\epsilon-1)^2 }.
\end{equation*}
finishing the proof of condition  \textbf{A4} (ii) for sufficiently large choice of $B_{\eps}$.

Next we prove assumption \textbf{A4} (iii). Let us denote by $\tilde \theta_{[k]}$ the $h(.,.)$-projection of $\tilde \theta$ onto $\Theta(k)$, for $\tilde \theta $ satisfying $h(p_{0}, p_{\tilde \theta}) \leq M_\epsilon \epsilon_n$. We first show  that there exist $C_1,C_2>0$ such that
\begin{align}
\frac{\pi_{|k}\big( B_k(\tilde \theta, \delta_{n,k} \sqrt{k/n},h )\big)}{\pi_{|k}\big( B_k(\theta_{[k]}^o, \sqrt{k/n},h )\big)}\leq \frac{\pi_{|k}\big( B_k(\tilde\theta_{[k]},C_1\delta_{n,k} /\sqrt{n},\|.\|_2 )\big)}{\pi_{|k}\big( B_k(\theta_{[k]}^o, C_2 /\sqrt{n},\|.\|_2) \big)}.\label{eq:help_hist_fraction}
\end{align}
For $k\in\mathcal{K}_n(M)$ we have
$$h(p_{\tilde \theta_{[k]}},p_{\theta_{[k]}^o})\leq h(p_{\tilde \theta},p_{\theta_{[k]}^o}) \leq h(p_{\tilde \theta},p_0)+h(p_{0},p_{\theta_{[k]}^o})\lesssim k_n(\log n)/n.$$
 Therefore by applying Lemma  \ref{lem: help_hist} (with $\mu_n=C\log n$ and $\tilde\theta=\theta_{[k]}^o$) we get that $\tilde \theta_{[k],j}\asymp k^{-1}$, for $j=1,...,k$.  Then by applying again 
 Lemma  \ref{lem: help_hist} (with $\mu_n=C\log n$ and $\tilde\theta=\tilde \theta_{[k]}$) we get that 
on $\theta\in  B_k(\tilde \theta_{[k]},2\delta_{n,k} \sqrt{k/n},h )\supset B_k(\tilde \theta,\delta_{n,k} \sqrt{k/n},h )$, $\theta_j\asymp k^{-1}$, $j=1,...,k$. Therefore as a consequence of assertion \eqref{eq: hist_help_1} 
$$h(p_{\tilde \theta_{[k]}}, p_{ \theta})\asymp \sqrt{k}\|\tilde\theta_{[k]}-\theta\|_2.$$
The same argument with $\mu_n=1$ gives the preceding display also for $\theta\in B_k(\theta_{[k]}^o, \sqrt{k/n},h)$, leading to \eqref{eq:help_hist_fraction}. 

Since $\pi_{|k}$ is a Dirichlet prior with parameters $(\alpha_{1,k},...,\alpha_{k,k})$ on the $k$-dimensional simplex  $\Theta(k)=\mathcal{S}_k$, with 
$k^{-a}c_1 \leq \alpha_{j,k} \leq C_1$, there exists a constant $C>0$ such that 
\begin{equation*}
\begin{split}
\frac{\pi_{|k}\big( B_k(\tilde\theta_{[k]},C_1\delta_{n,k}  /\sqrt{n},\|.\|_2 )\big)}{\pi_{|k}\big( B_k(\theta_{[k]}^o, C_2 /\sqrt{n},\|.\|_2)\big)} &  \leq  C^k\frac{  \text{Vol}\left(B_k(\tilde\theta_{[k]},C_1\delta_{n,k}/ \sqrt{n},\|.\|_2 )\cap \mathcal{S}_k\right) }{ \text{Vol}\left(B_k(\theta_{[k]}^o, C_2/\sqrt{n},\|.\|_2)\cap \mathcal{S}_k \right) }.
\end{split}
\end{equation*}
Moreover, since $\tilde \theta_{[k]} \in \Theta(k)$ and since $\tilde \theta_{[k],j} \gtrsim 1/k $ with $1/\sqrt{n} = o(1/k)$ we can re-express in a bijective way 
any $\theta \in B_k(\tilde\theta_{[k]},C_1\delta_{n,k}/ \sqrt{n},\|.\|_2 )\cap \mathcal{S}_k$  as $\tilde \theta_{[k]} + u $ with $u \in \mathbb R^k$, $\|u\|_2 \leq C_1\delta_{n,k}/ \sqrt{n}$ and $\1^T u = 0$. Moreover a $k$ dimensional ball with radius $r$ and centered at 0 intersected with a hyperplane (containing 0) is a $k-1 $ dimensional ball with the same radius so that 
$$ \text{Vol}\left(B_k(\tilde\theta_{[k]},C_1\delta_{n,k}/ \sqrt{n},\|.\|_2 )\cap \mathcal{S}_k \right) = \text{Vol}\left( B_{k-1}(0, C_1\delta_{n,k}/ \sqrt{n},\|.\|_2)\right).$$
The same argument implies that 
$$ \text{Vol}\left(B_k(\theta_{[k]}^o,C_2/ \sqrt{n},\|.\|_2 )\cap \mathcal{S}_k \right) = \text{Vol}\left( B_{k-1}(0, C_2/ \sqrt{n},\|.\|_2)\right)$$
and therefore
\begin{equation*}
\begin{split}
\frac{\pi_{|k}\big( B_k(\tilde\theta_{[k]},C_1\delta_{n,k} /\sqrt{n},\|.\|_2 )\big)}{\pi_{|k}\big( B_k(\theta_{[k]}^o, C_2/\sqrt{n},\|.\|_2)\big)} &  \leq  C^k\left( \frac{  C_1 \delta_{n,k} }{ C_2} \right)^{k-1} \leq e^{ k \log (\delta_{n,k} )/2}
\end{split}
\end{equation*}
for $\delta_{n,k}$ small enough. 


Finally we prove \textbf{A3}. By triangle inequality and $h^2(p_0,p_{\theta_{[k]}^o})\lesssim \|p_0 - p_{\theta^o_{[k]}}\|_2^2\asymp b(k)$ (see Lemma \ref{lem: histogram_projection} and assertion \eqref{Hell:Hist} with $\theta=\theta_{[k]}^o$) we have that there exists $C>0$ such that
\begin{align*}
B_k(\theta_0,J_1 \sqrt{k(\log n)/n},h)&\subset  B_k(\theta_{[k]}^o,C\sqrt{k(\log n)/n},h)\\
&\subset B_k(\theta_{[k]}^o,C_3\sqrt{(\log n)/n},\|.\|_2).
\end{align*}
Furthermore, in view of Lemma \ref{lem: help_hist} (with $\mu_n=\log n$) we have $c/k\leq |\theta_j|\leq C/k$, $j=1,...,k$,
 for $\theta\in B_k(\theta_0,J_1 \sqrt{k(\log n)/n},h)$. Therefore
\begin{align*}
\pi_{|k}&\big(B_k(\theta_0,J_1 \sqrt{k(\log n)/n},h)\big)\\
&\leq \frac{\Gamma(\sum_j \alpha_{j,k})}{\prod_j \Gamma(\alpha_{j,k})} (C/k)^{\sum_j (\alpha_{j,k} -1)}\text{Vol}\left( B_{k-1}(0, C_2\sqrt{(\log n)/ n},\|.\|_2)\right)
\end{align*}
From the assumption $\alpha_{j,k}\leq C$ we get that in view of Stirling's approximation that $\Gamma(\sum_j \alpha_{j,k})k^{-\sum_{j}\alpha_{j,k}}\lesssim e^{Ck}$. We conclude the proof by noting that $\text{Vol}\left( B_{k-1}(0, C_2\sqrt{\log_2 n/ n},\|.\|_2)\right)\lesssim e^{-ck\log n}$ and taking $M_0$ large enough.


\begin{lemma}\label{lem: help_hist}
Assume that $\tilde\theta\in\Theta(k)$ satisfies $c_0 k^{-1}\leq \tilde\theta_j\leq C_0 k^{-1}$, $j=1,...,k$, for some $0<c_0<C_0$. Then for every $\theta\in B_{k}(\tilde\theta, \sqrt{\mu_n k/n},h )$ with  $\mu_n\leq n/(4k^2)$ we have 
\begin{align*}
&c_0^2/(2k)\leq \theta_j\leq 4C_0/k,\qquad \text{for every $j=1,...,k$,}\\
 &\sum_{j=1}^k   (\theta_j- \tilde\theta_j)^2\leq \frac{ 9C_0\mu_n}{n}.
 \end{align*}
\end{lemma}

\begin{proof}
First note that
\begin{equation}\label{eq: hist_help_1}
\begin{split}
\frac{ k\mu_n }{ n } &\geq h^2(p_{\tilde\theta}, p_\theta)  = \sum_{j=1}^{k} (\sqrt{\theta_j}- \sqrt{\tilde\theta_j})^2 =  \sum_{j=1}^k  \frac{ (\theta_j-\tilde\theta_j)^2}{(\sqrt{\theta_j}+ \sqrt{\tilde\theta_j})^2}.
\end{split}
\end{equation}
As a consequence for all $j$, $$\sqrt{\theta_j}\leq  \sqrt{\tilde\theta_j} + |\sqrt{\theta_j}-\sqrt{\tilde\theta_j}| \leq \sqrt{C_0}/\sqrt{k} + \sqrt{k\mu_n}/\sqrt{n} \leq 2\sqrt{C_0}/\sqrt{k}$$
and similarly $\theta_j \geq c_0/(2k)$.
We get the second statement by combining \eqref{eq: hist_help_1} with the preceding upper bound.
\end{proof}

\begin{lemma}\label{lem: histogram_projection}
For some sufficiently large $C>0$ we have
\begin{align*}
C^{-1}\inf_{\theta\in\mathbb{R}^k} h^2(p_0, p_\theta)\leq  \|p_0 - p_{\theta_{[k]}^o} \|_2^2\leq C\inf_{\theta\in\mathbb{R}^k} h^2(p_0, p_\theta).
\end{align*}
\end{lemma}

\begin{proof}
Note that 
 \begin{equation} \label{Hell:Hist}
\begin{split}
\|p_0 - p_\theta\|_2^2 &= \int_0^1(p_0 - p_\theta)^2(x) dx 
\geq \int_0^1(\sqrt{p_0} -\sqrt{ p_\theta})^2(x) p_0(x) dx\geq c_0 h^2(p_0, p_\theta),
 \end{split}
 \end{equation}
where $c_0$ is the lower bound for the density $p_0$. Furthermore, for all $k$
 $$\inf_{\theta \in \mathbb R^k} \|p_0 -p_\theta\|_2^2 = \|p_0 - p_{\theta_{[k]}^o} \|_2^2$$
 with $\theta_{[k]}^o = ( \int_{I_j}p_0(x) dx, j \leq k ) \in \Theta(k) $ so that $b(k) \lesssim  \|p_0 - p_{\theta_{[k]}^o}\|_2^2$. Moreover, set $\eta_j^o = \int_{I_j} \sqrt{p_0}(x)dx \asymp 1/k $, then 
 \begin{equation*}
\begin{split}
  h^2(p_0, p_\theta) &\geq \sum_j \int_{I_j} ( \sqrt{p_0}(x) - k\eta_j^o )^2  dx  
   =  \sum_j \int_{I_j} \frac{( p_0(x) - k^2(\eta_j^o)^2 )^2}{ (\sqrt{p_0}(x) + k\eta_j^o)^2}   dx\\
    & \geq  \frac{1}{ 2C_0}   \sum_j \int_{I_j} ( p_0(x) - k^2(\eta_j^o)^2 )^2  dx \geq \frac{1}{ 2C_0} \|p_0 - p_{\theta_{[k]}^o}\|_2^2, 
 \end{split}
 \end{equation*}
hence  $b(k) \asymp  \|p_0 - p_{\theta_{[k]}^o}\|_2^2$. 
\end{proof}

 \begin{lemma}
 \label{lem:Variance:histo}
 If $k \leq \sqrt{n}$ then 
$$
V_{p_0}^{(n)}\left( \sum_{j=1}^k(n_j-n\theta_{j}^o)^2\right) \lesssim \frac{ n^2 }{ k }.
$$
 
 \end{lemma}
 
\begin{proof}
The variance term in the statement is equal to
 \begin{equation*}
 \begin{split}
 \sum_{j_1,j_2\leq k}& \sum_{i_1,..., i_4\leq n} E_{p_0}^{(n)}\left(\prod_{l=1}^2(\1_{X_{i_l}\in I_{j_1}} - \theta_{j_1}^o) \prod_{l=3}^4(\1_{X_{i_l}\in I_{j_2}} - \theta_{j_2}^o)\right) \\
 & \qquad - n^2\sum_{j_1,j_2=1}^k\theta_{j_1}^o(1-\theta_{j_1}^o)\theta_{j_2}^o(1-\theta_{j_2}^o)\\
 &=  \sum_{j_1,j_2\leq k}\theta_{j_1}^o(1-\theta_{j_1}^o)\theta_{j_2}^o(1-\theta_{j_2}^o) \left(n(n-1) - n^2\right) \\
 & \qquad   +n \sum_{j=1}^k E_{p_0}^{(n)}\left((\1_{X_{i}\in I_j}- \theta_{j}^o)^4\right)+  2 n(n-1) \sum_{j=1}^k\left(\theta_{j}^o(1-\theta_{j}^o)\right)^2\\
&\lesssim \frac{ n^2}{ k } + n \lesssim  \frac{ n^2}{ k }.
 \end{split}
 \end{equation*}
 \end{proof}

\subsection{Proof of Proposition \ref{prop: loglin}}\label{sec: proof_loglin} 
We need to verify assumptions \textbf{A1}-\textbf{A4} (with \textbf{A2} (iiib)). 
Before that we note that we take $d(.,.)$ to be the Hellinger distance $h(.,.)$ and we
choose
$$\Theta_n(k) = \{ \theta \in \RR^k ; \|\theta\|_2 \leq R_n(k)\}, \quad R_n(k) = C_1 \sqrt{k} (n\epsilon_n^2)^{1/q},$$
for some large enough $C_1>0$ and $q$ given in \eqref{eq: sieve_prior}. Define $\theta_{[k]}^o$ to be the Kullback-Leibler projection of $\theta_0$ onto $\Theta(k)$, i.e.
\begin{align}
\theta_{[k]}^o=\arg\inf_{\theta\in\Theta(k)}KL(\theta_0,\theta),\label{def: loglin:project}
\end{align}
 which exists and is unique by convexity of $\theta \rightarrow KL(\theta_0,\theta)$. Denote also $\theta_{0,[k]} = (\theta_{0,1}, ... , \theta_{0,k})$  
and let $\bar K_n \leq n^{1/2-\epsilon}$ for an arbitrarily small $\epsilon>0$. Note that since $f_0 \in \mathcal S^{\beta_0}(L)$ with 
$\beta_0 >1/2$, $k_n \lesssim (n/\log n)^{1/(2\beta_0+1)} \leq \bar K_n$, choosing $\eps$ small enough.
Condition \textbf{A1} is verified in the proof of Condition (C) of \cite{rivoirard:rousseau:12}.

The proof of assumption \textbf{A2} (i) is given in the proof of Proposition \ref{sec: proof_reg}, while for Hellinger tests satisfying \textbf{A2} (ii) we refer for instance to \cite{ghosal:ghosh:vdv:00}. To prove \textbf{A2} (iiib) we need to construct a covering of 
$$\bar \Theta_n(k)  = \Theta_n(k) \cap \{ \theta, h(f_0, f_\theta) \geq J_0(k) \epsilon_n \}.$$
Define 
 \begin{align}
B_{n,j}(k)=\Theta_n(k)\cap\{j\eps_n\leq \|\theta-\theta_{0}\|_2\leq (j+1)\eps_n \},\label{eq: loglin_Bnj}
\end{align}
with the notation $\|\theta-\theta_0\|_2^2=\sum_{i=1}^k(\theta_i-\theta_{0,i})^2+\sum_{i=k+1}^{\infty}\theta_{0,i}^2$ for $\theta\in\Theta(k)$, $k\leq\bar{K}_n$. If $\theta\in B_{n,j}(k)$ with $j\leq J_n:= J_1\sqrt{n} / \sqrt{k k_n \log n} $ and arbitrary $J_1>0$, then $\|\theta-\theta_0\|_2\lesssim 1/\sqrt{k}$ and therefore in view of Lemma \ref{lem: help_loglin_1}, $\|\theta - \theta_0\|_2 \asymp h(f_0, f_\theta)$. As a consequence condition   \eqref{cond: A4k'_1}  is satisfied with $c(k,j) = c j$ for some $c>0$ and
\begin{align*}
\log N( c_6c(k,j)\eps_n,   B_{n,j}(k), h(. , . ) ) &\leq\log N\big( \delta j \epsilon_n,   B_{n,j}(k), \| . \|_2\big) \\
&\leq C k = o(j^2 n\epsilon_n^2),
\end{align*}
for some sufficiently small $\delta>0$ resulting in condition \eqref{cond: A4k'_3}, for $j \leq J_n$. 

For $j>J_n$ define $\bar B_{n,J_n+1}(k) = \cup_{j>J_n} B_{n,j}(k)$. Since $\|\theta-\theta_0\|_2\gtrsim  1/\sqrt{k}$ for $\theta\in \bar{B}_{n,J_{n+1}}(k)$, note that in view of assertions (17) and (18) of \cite{rivoirard:rousseau:12} we have that
\begin{align*}
\|\theta-\theta_0\|_2^2\lesssim V(\theta_0,\theta)\lesssim h^{2}(f_\theta,f_{\theta_0})\big( k \|\theta-\theta_0\|_2^2+\log^2h(f_\theta,f_{\theta_0}) \big).
\end{align*}
Therefore, $h(f_\theta,f_{\theta_0})\gtrsim k^{-1/2}/\log n$ and $\theta\in \{h(f_{\theta},f_{\theta_0})>c(k,J_n+1)\eps_n\}$ holds for $c(k, J_n+1)= c k^{-1/2}\epsilon_n^{-1}/\log n$, for some sufficiently small constant $c>0$, hence condition \eqref{cond: A4k'_1} is verified. 
The entropy condition will follow from the second assertion of Lemma \ref{lem: help_loglin_1} 
\begin{equation*}
\begin{split}
  \log N\big( &c_6 c(k,J_n+1)\eps_n,  \bar B_{n,J_n+1}(k), h(. , . ) \big) \\
 &\leq  \log N\big(  (c_6/C) c(k,J_n+1)\eps_n/\sqrt{k}, \Theta_n(k), \| . \|_2\big)\\
 &\leq k\log (C R_n(k)\sqrt{k n}) \leq C' k \log n
\end{split}
\end{equation*}
for some $C'>0$.  Since $nc(k,J_n+1)^2 \eps_n^2  = c^2 n/(k \log n^2)$ and $k\leq \bar K_n \leq n^{1/2-\eps}$ for some $\eps>0$, we have $k\log n = o(nc(k,J_n+1)^2 \eps_n^2)$ verifying condition \eqref{cond: A4k'_3}, for $j = J_n+1$. Finally condition \eqref{cond: A4k'_2} is verified, noting that for all $k \leq \bar K_n$ 
\begin{align*}
\sum_{j \leq J_n }&e^{- c_5c(k,j)^2 n\epsilon_n^2/2}  + e^{-c_5c(k,J_n+1)^2 n\epsilon_n^2 /2} \\
&\leq 
\sum_{j \leq J_n} e^{-c_5c^2nj^2 \epsilon_n^2/2} + e^{ -  c_5 J_0^2 k^{-1}n/\log^2 n}
  \leq 3 e^{ -  c_5 J_0^2 k_n \log n/2},
\end{align*}
since $kk_n (\log n)^3\leq \bar{K}_n^2\log^3 n=o(n)$.

Next we deal with condition \textbf{A3}. Note that in view of Lemma  \ref{lem: help_loglin_1}  we have that $h(f_{\theta},f_{\theta_0})\asymp \|\theta-\theta_0\|_2$ over $\theta\in B_k( \theta_0, CJ_1\sqrt{k\log n /n} , \| . \|_2)\subset B_k( \theta_0, 1/\sqrt{k} , \| . \|_2)$, for arbitrary $C>0$, when $n$ is large enough. Furthermore, recall from the proof of condition \textbf{A2} (iiib) that
the inequality $\|\theta - \theta_0\|_2 > 1/\sqrt{k}$ implies $\sqrt{k\log n /n}\ll h(f_0, f_\theta)$. Therefore
$$B_k(\theta_0, J_1\sqrt{k\log n /n} ,h ) \subset B_k( \theta_0, CJ_1\sqrt{k\log n /n} , \| . \|_2),\quad  k \leq \bar K_n$$
and the proof of condition \textbf{A3} follows from the proof of Proposition \ref{prop: reg} (see assertion \eqref{eq: condA4:reg} and the argument below).

Then we verify condition \textbf{A4} (i). From \eqref{eq: help:loglike:loglin} and \eqref{hell-L2} we have that 
$$ E_{f_0}^{(n)}\log \frac{f_{\theta_{[k]}^o}}{f_{\theta}} \asymp  \|\theta  -\theta_{[k]}^o\|_2^2 \lesssim h^2(f_{\theta_{[k]}^o}, f_\theta) \leq C k/n$$
for some $C>0$ and all $\theta \in B_k(\theta_{[k]}^o, \sqrt{k/n},h)$, $k\in\mathcal{K}_n(M)$.  Denote  $\Phi(\textbf{Y}) = (\sum_{i=1}^n \phi_j(Y_i), j=1, ... , k)^T$. Then we also have in view of \eqref{eq: UB:normconst} and Lemma \ref{lem: help_loglin_2} that
\begin{align*}
V_{f_0}^{(n)} \log \frac{f_{\theta_{[k]}^o}}{f_{\theta}}  &\leq 2 n  \Big(E_{f_0}[(\theta - \theta_{[k]}^o)^T\Phi(Y_1)]^2+E_{f_0}[c(\theta)-c(\theta_{[k]}^o)]^2  \Big)\\
&\leq 2\big(\|f_{\theta_0}\|_\infty+\|f_{\theta_{[k]}^o}\|_\infty+C \big)\|\theta-\theta_{[k]}^o\|_2^2\lesssim k/n
\end{align*}
so \textbf{A4} (i) holds.

Next we verify  condition \textbf{A4} (ii).  Let $k \in \mathcal K_n(M)$.  By Cauchy-Schwarz inequality we get that
\begin{align*}
 \Big|\ell_n(\theta_{[k]}^o)-\ell_n(\theta)-nE_{f_0}^{(n)} \log \frac{f_{\theta_{[k]}^o}}{f_{\theta}}\Big|
&= \Big|(\theta_{[k]}^o-\theta)^T \big(\Phi(\textbf{Y})-E_{f_{\theta_0}}^{(n)}\Phi(\textbf{Y})\big)\Big|\\
&\leq \|\theta-\theta_{[k]}^o\|_2\|\Phi(\textbf{Y})-E_{f_{\theta_0}}^{(n)}\Phi(\textbf{Y})\|_2.
\end{align*}
Also note that
\begin{align*}
 E_{f_0}^{(n)}\|\Phi(\textbf{Y})-E_{f_0}^{(n)}\Phi(\textbf{Y})\|_2^2
=tr \Var_{f_0}^{(n)} \Phi(\textbf{Y}) \leq  kn\|f_{\theta_0}\|_{\infty}
\end{align*}
and by Markov's inequality  $\{ \|\Phi(\textbf{Y})-E_{f_0}^{(n)}\Phi(\textbf{Y})\|_2 \leq  \sqrt{kn\|f_{0}\|_{\infty}/\epsilon} \}$ holds with probability greater than $1-\epsilon$.
Therefore in view of Lemma \ref{lem: equiv:metric:loglin}
\begin{equation*}
\begin{split}
\ell_n(\theta) - \ell_n(\theta_{[k]}^o)& \leq \|\theta-\theta_{[k]}^o\|_2\sqrt{kn\|f_{0}\|_{\infty}/ \epsilon}  -   nc_0 \|\theta - \theta_{[k]}^o\|_2^2\\
& \leq \|\theta-\theta_{[k]}^o\|_2\sqrt{kn} \left( \sqrt{\frac{ \|f_0\|_\infty }{\epsilon} } - \frac{ \sqrt{n} \|\theta - \theta_{[k]}^o\|_2c_0}{ \sqrt{k}} \right) \leq \frac{ k \|f_0\|_\infty }{ c_0\epsilon },
\end{split}
\end{equation*}
holds with probability greater than $1-\eps$.

Finally, to prove \textbf{A4} (iii), first note that (similarly to the proof of condition \textbf{A3}) in view of Lemma \ref{lem: help_loglin_1}, for $k\in \mathcal K_n(M)$, if $\tilde\theta \in \Theta(k)$ satisfies $h(f_{\tilde\theta}, f_{\theta_0}) \leq (M_\epsilon+1)\epsilon_n$ then $\|\tilde\theta-\theta_0\|_2\lesssim \epsilon_n$ and therefore 
$$\|\tilde\theta-\theta_0\|_1\leq \|\tilde\theta-\theta_{0,[k]}\|_1+\|\theta_0-\theta_{0,[k]}\|_1\leq \sqrt{k}\|\tilde\theta-\theta_0\|_2+O(1)=O(1).$$ 
Hence for all $\theta \in B_k(\tilde\theta, \delta_{n,k} \sqrt{k/n} , h) $ again in view of Lemma \ref{lem: help_loglin_1}
$$h(f_{\tilde\theta}, f_\theta) \asymp \| \tilde\theta - \theta \|_2.$$
This is in particularly true for $\tilde\theta = \theta_{[k]}^o$ and we can bound 
\begin{equation*}
\frac{ \pi_{|k}\left( B_k(\tilde\theta,\delta_{n,k}\sqrt{k/n},h)\right) }{\pi_{|k}\left( B_k(\theta_{[k]}^o,\sqrt{k/n},h) \right)} \leq \frac{ \pi_{|k}\left( B_k(\tilde\theta,C_1\delta_{n,k}\sqrt{k/n},\| . \|_2\right) }{\pi_{|k}\left( B_k(\theta_{[k]}^o,C_2\sqrt{k/n}, \| . \|_2)\right)},
\end{equation*}
for some positive constants $C_1, C_2$. Moreover for all $\theta \in B_k(\tilde\theta,C_1\delta_{n,k}\sqrt{k/n},\| . \|_2)$ 
 $\|\theta\|_1 \leq \|\tilde\theta\|_1 + o(1)\leq \|\tilde\theta-\theta_0\|_1+\|\theta_0\|_1+O(1) \leq  \| \theta_0\|_1+O(1)$ and 
 $$G_1^k e^{-G_2 \sum_{j=1}^k |\theta_j|^q} \leq \prod_jg(\theta_j) \leq G_3^k \quad\text{where}\quad \sum_{j=1}^k |\theta_j|^q \leq \|\theta\|_1^q k^{(1-q)_+}.$$
 We thus obtain that 
\begin{align*}
\frac{ \pi_{|k}\left( B_k(\tilde\theta,\delta_{n,k}\sqrt{k/n},h)\right) }{\pi_{|k}\left( B_k(\theta_{[k]}^o,\sqrt{k/n},h) \right)} &\leq  e^{k \big(2G_2\|\theta_0\|_1^q+\log(G_3/G_1) \big)}\frac{\mbox{Vol}\big(B_k(\tilde\theta,C_1\delta_{n,k}\sqrt{k/n},\|.\|_2)\big)}{\mbox{Vol}\big(B_k(\theta_{[k]}^o,C_2\sqrt{k/n},\|.\|_2) \big)}\\
& \leq e^{ k \log (\delta_{n,k})/2}
\end{align*}
 as soon as $\delta_{n,k}$ is small enough.

We finally verify the second statement of  Proposition \ref{prop: loglin}. Following Corollary \ref{cor:radius} it is sufficient to show that for $\theta_0\in S^{\beta}(L)$ we have $\eps_n\lesssim  (n/\log n)^{-\beta/(1+2\beta)}$. Note that in view of \eqref{hell-L2} in Lemma \ref{lem: help_loglin_1} and Lemma \ref{lem: help_loglin_2}
\begin{align*}
h^2(f_{\theta_0},f_{\theta_{[k]}^o})&\lesssim e^{c_1 (\|\theta_0\|_1+ \|\theta_0  -\theta_{[k]}^o\|_1)}\|\theta_0 -\theta_{[k]}^o\|_2^2\\
&\lesssim e^{c_1  (\|\theta_0\|_1 +\|\theta_0  -\theta_{0,[k]}\|_1+ \|\theta_{0,[k]}  -\theta_{[k]}^o\|_1)}(\|\theta_0 -\theta_{0,[k]}\|_2^2+\|\theta_{0,[k]}  -\theta_{[k]}^o\|_2^2)\\
&\lesssim k^{-2\beta}\sum_{i=k+1}^{\infty}\theta_{0,i}^2i^{2\beta}\lesssim k^{-2\beta}.
\end{align*}
 Hence by choosing $\bar{k}_n=C(n/\log n)^{1/(1+2\beta)}$ we get that $b(\bar{k}_n)<\bar{k}_n(\log n)/n$ for sufficiently large $C>0$ and therefore $k_n\leq\bar{k}_n$. We conclude the proof by noting that
 $$\eps_n\leq 2\sqrt{k_n(\log n)/n}\leq   2\sqrt{\bar{k}_n(\log n)/n}\lesssim (n/\log n)^{-\beta/(1+2\beta)}.$$

\subsubsection{Technical Lemmas}
\begin{lemma}\label{lem: help_loglin_1}
Over $B_k(\theta_0,c/\sqrt{k},\|.\|_2)$, where $c>0$ is arbitrary and $\|\theta_0\|_1=O(1)$, we have that
\begin{align*}
h(f_\theta,f_{\theta_0})\asymp \|\theta-\theta_0\|_2.
\end{align*}
Furthermore, for any $\theta,\theta'\in\Theta(k)$, $\|\theta-\theta'\|_2\leq \delta/\sqrt{k}$, with some sufficiently small $\delta>0$ we have 
\begin{align*}
h(f_{\theta},f_{\theta'})\leq C\sqrt{k}\|\theta-\theta'\|_2,
\end{align*}
for some universal constant $C>0$.
\end{lemma}

\begin{proof}
First we deal with the first statement. We show below that there exists $\tilde{c}>0$ such that for all $A>0$,  there exists $C_A>0$ such that for any $\theta_1,\theta_2\in\ell_2 \cap \ell_1$ satisfying $\|\theta_1 - \theta_2\|_2 \leq A$ we have
\begin{equation}\label{hell-L2}
\begin{split}
h^2(f_{\theta_1}, f_{\theta_2}) &\leq C_A e^{\tilde{c}(\|\theta_1\|_1+  \|\theta_1  -\theta_2\|_1)}\|\theta_1 -\theta_2\|_2^2,\\
h^2(f_{\theta_1}, f_{\theta_2}) & \geq C_A^{-1} e^{- \tilde{c}  (\|\theta_1\|_1+  \|\theta_1  -\theta_2\|_1)}\|\theta_1 -\theta_2\|_2^2.
\end{split}
\end{equation}
Then the first statement of the lemma simply follows by noting that for $\theta\in B_k(\theta_0,c/\sqrt{k},\|.\|_2)$
$$\|\theta-\theta_0\|_1\leq \|\theta_0\|_1+\|\theta-\theta_{0,[k]}\|_1\leq   \|\theta_0\|_1+\sqrt{k}\|\theta-\theta_{0,[k]}\|_2=O(1)$$
and $\|\theta\|_1\leq\|\theta_0\|_1+\|\theta_0-\theta\|_1$,
where we use the following (slight abusement) of our notation $\|\theta-\theta_{0}\|_1=\sum_{i=1}^k|\theta_i-\theta_{0,i}|+\sum_{i> k}|\theta_{0,i}|$.

The lower bound in \eqref{hell-L2} is given in Lemma F.1 of \cite{rousseau:szabo:15:supp}. For the upper bound in \eqref{hell-L2} we use similar computations. Using the inequality $|e^{v}-e^w|= e^v|1-e^{w-v}|\leq e^v e^{|w-v|}|w-v|$ we get that
\begin{align*}
 h^2(f_{\theta_2}, f_{\theta_1}) &\leq \int f_{\theta_1} e^{|(\theta_2-\theta_1)^T\Phi(x) - c(\theta_2) + c(\theta_1)|}  \big((\theta_2-\theta_1)^T\Phi(x) - c(\theta_2) + c(\theta_1)\big)^2dx.
\end{align*}
Furthermore note that the following inequalities hold $\|(\theta_2-\theta_1)^T\Phi(x)\|_\infty\leq \|\theta_2-\theta_1\|_1 \|\Phi(x)\|_\infty$, $e^{|c(\theta_1)-c(\theta_2)|}\leq e^{\|\theta_1-\theta_2\|_1 \|\Phi\|_\infty}$, $\|f_{\theta_1}\|_\infty\lesssim e^{\|\theta_1\|_1\|\Phi\|_{\infty}}$ and
\begin{align}
|c(\theta_1)-c(\theta_2)|&=\log \int f_{\theta_1}(x) e^{(\theta_2-\theta_1)^T \Phi(x)}dx\nonumber\\
& \lesssim   \|f_{\theta_1}\|_\infty  \|\theta_1-\theta_2\|_2+ O( \|\theta_1-\theta_2\|_2^2),\label{eq: UB:normconst}
\end{align}
where the last display follows from the Taylor expansion of the functions $\log(1+x)$ and $e^x$ around zero (see also the first display after (F.2) in \cite{rousseau:szabo:15:supp}).
The proof of the statement concludes by noting that the bases $\phi_1,...,\phi_k$ are orthogonal.

For the second statement of the lemma  we note that following from equation (8) of \cite{rivoirard:rousseau:12} for $\|\theta - \theta'\|_2 \leq  \delta/\sqrt{k}$ (for some sufficiently small $\delta>0$),
\begin{align*}
h(f_\theta, f_{\theta'}) \leq 4\|\sum_{j=1}^k (\theta_j-\theta_j')\phi_j\|_\infty^2
\leq 4 \|\theta-\theta'\|_1\max_{j=1,..,k}\|\phi_j\|_{\infty}\leq C\sqrt{k}\|\theta-\theta'\|_2.
\end{align*}
\end{proof}

\begin{lemma}\label{lem: help_loglin_2}
For all $c\leq f\leq C $ and all $k \geq 1$, the matrices 
\begin{align*} 
\bar \Gamma(i,j)= E_f\phi_i(X)\phi_j(X),\quad \Gamma(i,j) = \bar \Gamma(i,j)-E_f\phi_i(X) E_f\phi_j(X),
\end{align*}
 satisfy
  \begin{equation}\label{cov:lb}
c I_k\leq \Gamma \leq \bar\Gamma\leq CI_k.
   \end{equation}
\end{lemma}

\begin{proof}
Let $\tilde \phi_{j}  = \phi_j -  E_f(  \phi_{j}) $,  and recall that $0<c\leq f\leq C$ and $\int \phi_j=0$. Then we have for all $u \in \mathbb R^k$,
\begin{equation*}
\begin{split}
u^T \Gamma u &= E_f\left( (\sum_{j\leq k} u_j \tilde \phi_j)^2\right) \geq c \|\sum_{j\leq k} u_j \tilde \phi_j\|_2^2\\
& = c \|u\|_2^2 + c(\sum_{j=1}^kE_f(  \phi_{j}) u_j)^2 \geq  c\|u\|_2^2.
\end{split}
\end{equation*}
Similarly
\begin{align*}
u^T \bar\Gamma u =  E_f\left( (\sum_{j\leq k} u_j  \phi_j)^2\right) \leq C \|u\|_2^2.
\end{align*}
Also $\bar \Gamma = \Gamma + E_f(\Phi_k)E_f(\Phi_k)^T$, terminating the proof of \eqref{cov:lb}.
\end{proof}

\begin{lemma}\label{lem: equiv:metric:loglin}
Let $\theta_{[k]}^o $ be the Kullback-Leibler projection of $\theta_0$ onto $\Theta(k)$, given in \eqref{def: loglin:project}, then $\theta_{[k]}^o$ satisfies 
\begin{equation*}\label{carac:thetako}
E_{f_{\theta_0}}(\phi_j) = E_{f_{\theta_{[k]}^o}}(\phi_j) , \quad \forall j \leq k,\quad \mbox{and} \quad \theta_{[k]}^o = \theta_{0,[k]} + \delta 
\end{equation*}
with
\begin{equation*}
\|\delta\|_2^2 \leq C_1 \|\theta_0 - \theta_{0,[k]}\|_2^2, \quad \|\delta\|_1 \leq C_1 k \sqrt{(\log n)/n}
\end{equation*}
as soon as $k \in \mathcal K_n(M)$, where $C_1$ depends on $M$, $\|\theta_0\|_1$ and $\|\theta_0\|_2$. For $\theta\in B_k(\theta_0,C\eps_n,h)$ with arbitrary $C>0$ we have
\begin{align}
E_{f_{\theta_0}}^{(n)}\log \frac{f_{\theta_{[k]}^o}}{f_{\theta}} \asymp \|\theta-\theta_{[k]}^o\|_2^2.\label{eq: help:loglike:loglin}
\end{align}
\end{lemma}

\begin{proof}

For convenience we introduce the notations $E_0=E_{f_{\theta_0}}$ and $E_{\theta}=E_{f_{\theta}}$. Then by definition, $\theta_{[k]}^o$ satisfies 
\begin{equation}
0=\frac{ \partial \int f_0(x) \log f_\theta(x) dx   }{ \partial \theta_j }\vert_{\theta = \theta_{[k]}^o}   = -E_0( \phi_j) +E_{\theta_{[k]}^o}(\phi_j) .\label{eq: help_deriv}
\end{equation}
Write $\theta_{[k]}^o = \theta_{0,[k]} +\delta$ with $\theta_{0,[k]} = (\theta_{0,1}, ..., \theta_{0,k})$ where $\delta \in \mathbb R^k  $ and  $\Delta(x) = \sum_{j\leq k } \delta_j \phi_j(x)$. 
We have
\begin{equation}\label{KL1}
h^2(f_{\theta_0}, f_{\theta_{[k]}^o}) \leq KL(\theta_0, \theta_{[k]}^o) \leq KL(\theta_0, \theta_{0,[k]}) 
\end{equation}
and using a Taylor expansion of $e^{x}$ followed by a Taylor expansion of $\log (1 + x)$, combined with Lemma  \ref{lem: help_loglin_2}, we obtain that 
\begin{align*}
KL(\theta_0,& \theta_{0,[k]})  = \sum_{j=k+1}^\infty \theta_{0,j}  E_0( \phi_j) - \log E_0\left( e^{\sum_{j>k}\theta_{0,j}\phi_j}\right)\\
& = \sum_{j=k+1}^\infty \theta_{0,j}  E_0( \phi_j) - \log \Big( 1 + \sum_{j=k+1}^\infty \theta_{0,j}  E_0( \phi_j)  + \frac{(\sum_{j=k+1}^\infty \theta_{0,j}  E_0( \phi_j) )^2 }{2}\\
&\qquad\qquad+ o( \sum_{j=k+1}^\infty \theta_{0,j}^2) \Big) \\
&=\frac{ \sum_{j,j'=k+1}^\infty \theta_{0,j} \theta_{0,j'}E_0(\phi_j)E_0(\phi_{j'})}{2} + o(\|\theta_0 -\theta_{0,[k]}\|_2^2)\\
&\asymp \|\theta_0 -\theta_{0,[k]}\|_2^2 \asymp h^2(f_{\theta_0}, f_{\theta_{0,[k]}}) \lesssim \epsilon_n^2.
\end{align*}
Moreover using Lemma 3.1 of \cite{rivoirard:rousseau:12}
 $$h^2(f_{\theta_0}, f_{\theta_{[k]}^o}) \gtrsim \|\theta_0 - \theta_{[k]}^o \|_2^2 (\log n)^{-2}$$
 so that 
  $$\|\theta_0 - \theta_{[k]}^o \|_2^2\lesssim \epsilon_n^2(\log n)^2 = o(1/k_n) , \quad \|\theta_{0,[k]} - \theta_{[k]}^o \|_1 \leq \sqrt{k} \epsilon_n \log n = o(1).$$
We can then work similarly to $KL(\theta_0, \theta_{0,[k]}) $: let $\bar \Gamma_{0,k}(i,j) = E_0(\tilde \phi_i(X)\tilde \phi_j(X))$, $i,j\leq k$. 
\begin{align*}
KL(\theta_0,& \theta_{[k]}^o)  = \sum_{j=0}^\infty( \theta_{[k],j}^o -\theta_{0,j})  E_0( \phi_j) - \log E_0\left( e^{\sum_{j=0}^\infty( \theta_{[k],j}^o -\theta_{0,j})   \phi_j}\right)\\
&\asymp  \|\theta_0 -\theta_{[k]}^o\|_2^2.
\end{align*}
Hence 
$$\|\theta_0 -\theta_{[k]}^o\|_2^2  \lesssim  KL(\theta_0, \theta_{[k]}^o) \leq KL(\theta_0, \theta_{0,[k]}) \lesssim  \|\theta_0 -\theta_{0,[k]}\|_2^2 \leq \|\theta_0 -\theta_{[k]}^o\|_2^2$$ 
so that 
$$ \|\theta_{0} - \theta_{[k]}^o \|_2^2\asymp \|\theta_0- \theta_{0,[k]}\|_2^2.$$
Using 
$$\|\theta_{0} - \theta_{[k]}^o \|_2^2 = \|\theta_0- \theta_{0,[k]}\|_2^2+ \|\theta_{[k]}^o- \theta_{0,[k]}\|_2^2$$
we obtain that 
$$\|\delta\|_2^2= \|\theta_{[k]}^o- \theta_{0,[k]}\|_2^2 \lesssim \|\theta_0- \theta_{0,[k]}\|_2^2 \lesssim \epsilon_n^2 $$
which in turn implies that 
$$\|\delta\|_1 \leq \sqrt{k}\|\delta\|_2 \lesssim \sqrt{k} \epsilon_n \lesssim k \sqrt{(\log n)/n}. $$

We show below that
\begin{align}
\sup_{\theta\in \Theta_n(k)\cap B_k(\theta_0,C\eps_n,h) }\|\theta-\theta_{[k]}^o\|_2 \lesssim \epsilon_n\label{eq: dist:project:loglin}
\end{align}
and as a consequence we have $\|\theta-\theta_{[k]}^o\|_1\leq \sqrt{k}\|\theta-\theta_{[k]}^o\|_2=o(1)$.
 Next note that
\begin{align*}
 \frac{\int e^{\theta^T \Phi_k(x)}dx }{\int e^{(\theta_{[k]}^o)^T \Phi_k(x)}dx }=\frac{\int e^{(\theta-\theta_{[k]}^o)^T \Phi_k(x)} e^{(\theta_{[k]}^o)^T \Phi_k(x)}dx }{\int e^{(\theta_{[k]}^o)^T \Phi_k(x)}dx }=E_{f_{\theta_{[k]}^o}}( e^{(\theta - \theta_{[k]}^o)^T \Phi_k(Y_1)})
\end{align*}
and by Taylor series expansion of $e^{(\theta - \theta_{[k]}^o)^T \Phi_k(Y_1)}$ around zero
\begin{align*}
E_{f_{\theta_{[k]}^o}} e^{(\theta - \theta_{[k]}^o)^T \Phi_k(Y_1)}=&1+E_{f_{\theta_{[k]}^o}}^{(n)}(\theta - \theta_{[k]}^o)^T \Phi_k(Y_1)\\
&\qquad+\frac{ (\theta - \theta_{[k]}^o)^T    E_{f_{\theta_{[k]}^o}}\Phi_k(Y_1)\Phi_k(Y_1)^T(\theta - \theta_{[k]}^o)}{ 2 }\\
&\qquad+O(  \| \theta - \theta_{[k]}^o \|_1 \| \theta - \theta_{[k]}^o \|_2^2).
\end{align*}
Since $\log(1+x)=x-x^2/2+O(x^3)$ for small $x>0$ we get
\begin{align*}
\log E_{f_{\theta_{[k]}^o}} &e^{(\theta - \theta_{[k]}^o)^T \Phi_k(Y_1)}=  E_{f_{\theta_{[k]}^o}}(\theta - \theta_{[k]}^o)^T \Phi_k(Y_1)\\
&+\frac{ (\theta - \theta_{[k]}^o)^T  E_{f_{\theta_{[k]}^o}}\Phi_k(Y_1)\Phi_k(Y_1)^T(\theta - \theta_{[k]}^o)}{ 2 }
-\frac{\Big(E_{f_{\theta_{[k]}^o}}(\theta - \theta_{[k]}^o)^T \Phi_k(Y_1)\Big)^2}{2}\\
&  +o(   \| \theta - \theta_{[k]}^o \|_2^2).
\end{align*}
This leads to 
\begin{equation*}
\begin{split}
E_{f_{\theta_0}}\log \frac{f_{\theta_{[k]}^o}}{f_{\theta}} &= ( \theta_{[k]}^o-\theta)^T E_{f_{\theta_0}}(\Phi_k(Y_1)) + \log E_{f_{\theta_{[k]}^o}}( e^{(\theta - \theta_{[k]}^o)^T \Phi_k(Y_1)}) \\
& = \frac{ (\theta - \theta_{[k]}^o)^T   \text{Cov}_{f_{\theta_{[k]}^o}}( \Phi_k(Y_1))(\theta - \theta_{[k]}^o)}{ 2 }+ o(   \| \theta - \theta_{[k]}^o \|_2^2)\\
&\asymp  \|\theta - \theta_{[k]}^o\|_2^2,
\end{split}
\end{equation*}
 by Lemma \ref{lem: help_loglin_2},  concluding the proof of assertion \eqref{eq: help:loglike:loglin}.
 
To finish the proof of the lemma it remains to verify \eqref{eq: dist:project:loglin}. We have seen that  $\|\theta_{[k]}^o-\theta_{0,[k]}\|_2\leq C\|\theta_{0,[k]}-\theta_0\|_2$, therefore for any $\theta\in\Theta(k)$ 
$$\|\theta-\theta_{[k]}^o\|_2\leq \|\theta_{0,[k]}-\theta_{[k]}^o\|_2+\|\theta-\theta_{0,[k]}\|_2\leq
(1+C)\|\theta-\theta_0\|_2.$$
Combined with  $h(f_{\theta_0},f_{\theta})\asymp \|\theta-\theta_0\|_2$ (see Lemma \ref{lem: help_loglin_1}) this concludes the proof of \eqref{eq: dist:project:loglin} and as a consequence the lemma.


\end{proof}

\subsection{Proof of Proposition \ref{prop: classification}}\label{sec: proof_class}
As a first step we introduce the following notations which will be used throughout the whole proof. Let $\theta_{0,[k]}=(\theta_{0,1},...,\theta_{0,k})$ and $\theta_{[k]}^o=\arg\min_{\theta\in\Theta(k)} KL(\theta_0,\theta)$ denote the Kullback-Leibler projection of $\theta_0$ onto $\Theta(k)$. The corresponding binary regression function is given as
$$q_{\theta_{[k]}^o}(x) := \mu(f_{\theta_{[k]}^o})(x)= \frac{ e^{(\theta_{[k]}^o)^T \Phi_k(x)}}{ 1+ e^{(\theta_{[k]}^o)^T \Phi_k(x)}}.$$
For notational convenience we also introduce the abbreviations $q^o_i=q_{\theta_{[k]}^o}(x_i)$, $q_i=\mu(f_{\theta})(x_i)$, $q_{0,i}=\mu(f_{\theta_0})(x_i)$, and $q_{0[k],i}=\mu(f_{\theta_{0,[k]}})(x_i)$. Finally by slightly abusing our notations we write 
$$h_n^2(\theta,\tilde{\theta})=n^{-1}\sum_{i=1}^n \big( \sqrt{q_{\theta}(x_i)}- \sqrt{q_{\tilde\theta}(x_i)}\big)^2+\big( \sqrt{1-q_{\theta}(x_i)}- \sqrt{1-q_{\tilde\theta}(x_i)}\big)^2$$
 for the empirical Hellinger distance and $d_n^2(\theta,\theta_0)=n^{-1}\sum_{i=1}^n (f_{\theta}(x_i)-f_{\theta_0}(x_i))^2$ for the empirical $L_2$-norm. Similarly to the preceding sections we use the shorthand notation $\eps_n=\eps_n(k_n)$. Then the proof of the first assertion consists of verifying the conditions of Theorem \ref{th:coverage} and \ref{th:EmpBayesCoverage}.  

First we deal with condition \textbf{A1}. Note that in view of Lemma 3.2 of \cite{vvvz08} (with $G(x)=n^{-1}\sum_{i=1}^n 1_{x\leq x_i}$ and uniformly bounded $S$ in case of the logistic link function)
$$KL(\theta_0,\theta) \lesssim d_n^2(\theta,\theta_0),\quad V(\theta_0, \theta)\lesssim d_n^2(\theta,\theta_0).$$
Furthermore, by the mean value theorem, for all $\theta, \theta' \in \bar{\Theta}$ 
\begin{equation}
h_n(\theta, \theta')^2  = n^{-1}\sum_{i=1}^{n}(f_{\theta}(x_i)-f_{\theta'}(x_i))^2\frac{\mu'(\bar{f}(x_i))^2}{4\mu(\bar{f}(x_i))\big(1-\mu(\bar{f}(x_i))\big)},\label{eq: hellinger_L2}
\end{equation}
for some $\bar{f}(x_i)\in [f_{\theta}(x_i), f_{\theta'}(x_i)]$. In view of Lemma \ref{lem: class_Linf} we have over $\theta\in B_k(\theta_0,C\eps_n,h_n)$, $k\leq \bar{K}_n=o(n^{1/2})$ that $\|f_\theta\|_\infty=O(1)$ and as a consequence (in view of \eqref{eq: hellinger_L2})
\begin{align}
d_n(\theta,\theta_0)\asymp h_n(\theta,\theta_0).\label{eq: equiv:metric}
\end{align}
This holds in particular for $k=k_n$ as well. Therefore there exist large enough constants $c_{3},c_{4}>0$ such that 
\begin{align*}
\{\theta:\, KL(\theta_0,\theta) \leq c_{3}\eps_n^2,V(\theta_0,\theta) \leq c_{4}\eps_n^2\}\supset B_{k_n}(\theta_0,\eps_n,h_n).
\end{align*}
Taking any $\tilde\theta\in\Theta(k)$ satisfying $h_n^2(\tilde\theta,\theta_0)\leq b(k)+ k(\log n)/(4n)$ and $|\tilde\theta_j|=O(1)$ we get that
\begin{align*}
B_{k_n}(\theta_0&,\eps_n,h_n)\supset B_{k_n}(\tilde\theta, 0.5\sqrt{k_n(\log n)/n},h_n)\\
&\supset  B_{k_n}(\tilde\theta,c\sqrt{k_n(\log n)/n},d_n)\supset B_{k_n}(\tilde\theta,C\sqrt{k_n(\log n)/n},\|.\|_2),
\end{align*}
for some constants $c,C>0$, where the last line follows from condition (20) in \cite{rousseau:szabo:2015:main}, i.e. 
for all $k \leq K_n$ and $\theta,\bar\theta\in\Theta(k)$,
\begin{align}
d_n^2(\theta, \bar\theta) = n^{-1}( \theta- \bar\theta)^T \Phi_k^T \Phi_k ( \theta - \bar\theta) \asymp \|\theta - \bar\theta\|_2^2. \label{eq: help1_classification}
\end{align}
The proof of condition \textbf{A1} concludes by applying Lemma \ref{lem: prior small}.

Next let us consider the sieves 
$$\Theta_n(k) = \{ \theta \in \RR^k ; \|\theta\|_2 \leq R_n\}, \quad R_n = C_1\sqrt{k}(n\epsilon_n^2)^{1/q},$$
for some large enough constant $C_1$ and $q$ given in \eqref{eq: sieve_prior}. Condition \textbf{A2} (i) follows from Lemma \ref{lem: prior small}. Since the function $x\mapsto\mu'(x)^2/[\mu(x)(1-\mu(x)) ]$ is uniformly bounded from above, in view of \eqref{eq: hellinger_L2} and \eqref{eq: help1_classification} the Hellinger metric is bounded by (a multiple of) the $\ell_2$ distance. Let $u >J_1\sqrt{k (\log n)/n} $, the covering number of $\Theta_n(k)$ by $\ell_2$-balls of radius $c_6 u $ is bounded from above by a term of order $\exp \left( C k (\log n -  \log u)\right)$ so the local entropy is bounded by $\tilde{C} k \log n$ and since $nu^2 >  J_1^2 k \log n$, choosing $J_1$ large enough ($J_1^2\geq 2\tilde{C}/c_5$ is large enough), for all $k$, the entropy condition \textbf{A2} (iii) also holds. The testing condition \textbf{A2} (ii) follows for instance from Corollary 4 on page 149 in \cite{birge:1983} or Lemma 2 in \cite{lecam:1986}.

Next we deal with assumption \textbf{A3}. In view of \eqref{eq: equiv:metric} and \eqref{eq: help1_classification}
\begin{align*}
\pi_{|k}\big( B_k(\theta_0,J_1\sqrt{k(\log n)/n},h_n) \big)&\leq \pi_{|k}\big( B_k(\theta_0,C\sqrt{k(\log n)/n},\|.\|_2) \big)\\
&\leq c_{\max}^k \mbox{Vol}\big( B_k(\theta_0,C\sqrt{k(\log n)/n},\|.\|_2) \big)\\
&\leq e^{-c'k\log n}\leq e^{-(c'/2)M_0n\eps_n^2}
\end{align*}
and the condition follows by large enough choice of $M_0$.

Condition \textbf{A4} (ii) is verified in Lemma \ref{Lem: class:likeratio}.  
To prove condition \textbf{A4} (iii), we use Lemma \ref{lem: class_Linf} and Lemma \ref{lem:help_class}, so that $\sup_{h_n(\theta,\theta_0)\leq (M_\eps+1) \eps_n}\|f_{\theta}\|_{\infty}\leq C$ and $\|f_{\theta_{[k]}^o}\|_{\infty}\leq C$, respectively. Therefore by combining \eqref{eq: hellinger_L2} and \eqref{eq: help1_classification}, there exist constants $c,C>0$ such that for all $\theta\in\Theta(k)$ satisfying $h_n(\theta,\theta_0)\leq (M_\eps+1) \eps_n$
\begin{align*}
\frac{\pi_{|k}\big( B_k(\theta,\delta_{n,k} \sqrt{k/n},h_n) \big)}{ \pi_{|k}\big( B_k(\theta_{[k]}^o,\sqrt{k/n},h_n)\big)}\leq \frac{\pi_{|k}\big( B_k(\theta, C\delta_{n,k}\sqrt{k/n},\|.\|_2) \big)}{ \pi_{|k}\big( B_k(\theta_{[k]}^o,c \sqrt{k/n},\|.\|_2)\big)}.
\end{align*}
We conclude that condition \textbf{A4} (iii) holds, following the same lines of reasoning as in the proof of Proposition \ref{prop: reg} (below assertion \eqref{eq: regr_A5}).

Finally we deal with condition \textbf{A4} (i). Note that for all $k\in\mathcal{K}_n(M)$
\begin{equation}\label{eq: KL_bound_projection}
\begin{split}
p_{\theta_0}^{(n)} \log\frac{ p_{\theta_{[k]}^o}^{(n)}}{  p_{\theta}^{(n)}} &=\sum_{i=1}^n  q_{0,i}\Big[ (\theta_{[k]}^o- \theta)^T \Phi_k(x_i) + \log ( 1 + e^{\theta^T\Phi_k(x_i)} ) - \log ( 1 + e^{(\theta_{[k]}^o)^T\Phi_k(x_i)} )\Big]  \\
&\qquad + \sum_{i=1}^n(1-q_{0,i})\Big[\log ( 1 + e^{\theta^T\Phi_k(x_i)} ) - \log ( 1 + e^{(\theta_{[k]}^o)^T\Phi_k(x_i)} )\Big]\\
&= \sum_{i=1}^n  \big(q_{0,i} - q^o_i\big)(\theta_{[k]}^o- \theta)^T \Phi_k(x_i) + O\left(\sum_i \left( (\theta_{[k]}^o - \theta)^T \Phi_k(x_i) \right)^2 \right)\\
&= O\left(\sum_{i=1}^n \left( ( \theta_{[k]}^o - \theta)^T \Phi_k(x_i) \right)^2 \right) = O(n\| \theta_{[k]}^o - \theta\|_2^2),
\end{split}
\end{equation}
where in the second line we used the Taylor expansions of $f(\theta)=\log (1+e^{\theta^T \Phi_k(x_i)})$ around $\theta_{[k]}^o$, while the third line follows from Lemma \ref{lem:help_class} and assertion \eqref{eq: help1_classification}.  Also note that for $\theta\in\Theta(k)$, $k\in\mathcal{K}_n(M)$ satisfying $h_n(\theta_{[k]}^o,\theta)\lesssim \eps_n$ we get by triangle inequality and assertion \eqref{eq: UB:hellinger:class} that 
\begin{align*}
h_n(\theta,\theta_0)\leq h_n(\theta_{[k]}^o,\theta)+h_n(\theta_{[k]}^o,\theta_0) \lesssim\eps_n+ n^{-\beta_0/(1+2\beta_0)}
\end{align*}
 and therefore in view of Lemma \ref{lem: class_Linf},  $\|f_{\theta}\|_{\infty}=O(1)$ so the right hand side of display \eqref{eq: KL_bound_projection}  is of order $O(nh_n^2(\theta_{[k]}^o,\theta))$. Similarly we obtain
\begin{equation*}
\begin{split}
p_{\theta_0}^{(n)} \log^2\left(\frac{ p_{\theta_{[k]}^o}^{(n)}}{  p_{\theta}^{(n)}} \right) &\lesssim n \| \theta_{[k]}^o - \theta\|_2^2 = O(nh_n^2(\theta_{[k]}^o,\theta))
\end{split}
\end{equation*}
providing us \textbf{A4} (i).

Next we show that the posterior mean $\hat{q}_n=E_{\pi(.|\textbf{Y})}(\mu(f_{\theta}))$ satisfies condition \textbf{A0}. By convexity and boundedness of  $h_n^2$ 
 $$h_n^2(\hat{q}_n, q_0) \leq E_{\pi(.|\textbf{Y})}\left( h_n^2(\theta, \theta_0) \right) \leq C\epsilon_n^2 + 2\pi\left( h_n^2(\theta,\theta_0)>C\epsilon_n^2 | \mathbf Y\right)  \lesssim \epsilon_n^2,$$
 where $E_{\pi(.|\textbf{Y})}$ denotes the expectation with respect to the posterior,  as soon as $\epsilon_n^2 \gtrsim (n\epsilon_n^2)^{-1}$, see \eqref{eq: post:contraction:asymptotic}.

It remains to show the second statement of the proposition. Again as a consequence of Corollary \ref{cor:radius} it is sufficient to verify that $\eps_n\lesssim (n/\log n)^{-\beta/(1+2\beta)}$, which follows automatically from the computations in Remark \ref{rem: UBforEps_n}, where the bound $k_n\lesssim (n/\log n)^{1/(1+2\beta)}$ was derived.

\begin{lemma}\label{lem: class_Linf}
 Let $\theta_0\in S^{\beta_0}(L_0)$ for some $\beta_0>1/2$, $L_0>0$,  $\theta\in\Theta(k)$ for $k=o\big((n/\log n)^{\frac{2\beta_0}{1+2\beta_0}}\big)$ satisfying $h_n(\theta_0, \theta) \leq C (n/\log n)^{- \frac{\beta_0}{1+2\beta_0}}$, then $\|f_{\theta } \|_{\infty} =O(1) $. 
\end{lemma}

\begin{proof}
Assume that $\max_i |f_{\theta } (x_i ) |  > L $ and split $ \{ 1, ..., n\}$ into $I_1 = \{i: \, |f_{\theta } (x_i ) |  \leq L\}$, $I_2 = \{i:\, f_{\theta } (x_i ) > L\}$ and $I_3 =  \{ i:\, f_{\theta } (x_i ) < -L\}$. Then we have for all $i \in I_2$ and $l\in\mathbb{N}$ that $\mu(f_{\theta}(x_i)) \geq (1  +\delta) \mu(f_{\theta_{0,[l]}}(x_i))$ for some $\delta >0$ fixed, by choosing $L$ large enough. Similarly for all $i \in  I_3$, $1 - \mu(f_{\theta}(x_i)) \geq (1  +\delta) (1 - \mu(f_{\theta_{0,[l]}}(x_i)))$.  Therefore we can conclude that for all $i \in I_2 \cup I_3$, 
\begin{align*}
h_b^2\big(\mu(f_\theta(x_i)),\mu(f_{\theta_{0,[l]}}(x_i))\big)\geq \delta^2 \big[\big( 1-\mu(f_{\theta_{0,[l]}}(x_i)) \big)\vee\mu(f_{\theta_{0,[l]}}(x_i))\big],
\end{align*}
following from the inequality $(\sqrt{1+\delta}-1)^2\geq\delta^2/4$, for all sufficiently small $\delta>0$.

Furthermore, note that for every $\theta\in\Theta(k)$
\begin{align}
\| f_{\theta}\|_\infty \leq \max_{j=1,...,k}\|\phi_j\|_\infty \|\theta\|_1.\label{eq: UB_Linfty}
\end{align}
Then \eqref{eq: hellinger_L2} combined with the preceding display implies in particular that for all $L>0$ there exist $c,C>0$ such that if $\|\theta\|_1\leq L$, $\|\theta'\|_1\leq L$, then 
\begin{equation}\label{eq:distances}
c d_n(\theta, \theta') \leq h_n(\theta , \theta')\leq C d_n(\theta, \theta').
\end{equation}
Hence, in view of assumption (20) of \cite{rousseau:szabo:2015:main} if $\theta, \theta' \in \Theta(k)$ with $k\leq K_n$, then \eqref{eq:distances} remains valid with $\|\theta - \theta'\|_2$ replacing $d_n(\theta, \theta').$

Let $k_n^* = m_n(n/\log n)^{\frac{2\beta_0}{1+2\beta_0}}$, for some $m_n=o(1)$, and note that  \eqref{eq: UB_Linfty} with $\theta=\theta_{0,[k_n^*]}$ implies that $\|f_{\theta_{0,[k_n^*]}}\|_{\infty}\leq C$, which in case of $|f_{\theta}(x_i )| \leq L$ results in
$h_b^2(q_\theta(x_i), q_{\theta_{0,[k_n^*]}}(x_i)) \geq c (f_{\theta}(x_i)- f_{\theta_{0,[k_n^*]}}(x_i))^2$, for some sufficiently small $c>0$. By slightly abusing our notation we write $\theta\in\Theta(k)$ in the form $\theta=(\theta_1,...,\theta_k,0,...,0)\in\mathbb{R}^{k_n^* }$. Then the preceding argument implies in particular  that 
\begin{equation}\label{eq: help_hellingerLB}
\begin{split}
n h_n^2 (\theta,\theta_{0,[k_n^*]} ) &\geq c\sum_{i \in I_1} (  f_{\theta}(x_i)-f_{\theta_{0,[k_n^*]}} (x_i))^2 + c\delta^2 \sum_{i\in I_2\cup I_3} q_{0[k_n^*],i}\vee (1 - q_{0[k_n^*],i}) \\
&\gtrsim ( \theta-  \theta_{0,[k_n^*]})^T \Phi_{I_1}^T \Phi_{I_1} ( \theta  - \theta_{0,[k_n^*]}) +  | I_2 \cup I_3|  
\end{split}
\end{equation}
where in the second inequality the matrix $\Phi_{I_1} \in \mathbb{R}^{|I_1|\times k_n^*}$ is defined by its rows $(\phi_1(x_i), ..., \phi_{k_n^*}(x_i))$, $i \in I_1$.

We show below that 
\begin{align}
\frac{\Phi_{I_1}^T \Phi_{I_1}}{ n }  \asymp I_d. \label{eq: small index_mtx}
\end{align}
Then in view of \eqref{eq: UBforBias_reg}
\begin{align*}
h_n(\theta_0,\theta_{0,[k_n^*]})\lesssim d_n(\theta_0,\theta_{0,[k_n^*]}) \lesssim (k_n^*)^{-\beta_0 } \label{eq: help_UB_hellinger_dist}
\end{align*}
which together with \eqref{eq: small index_mtx} and \eqref{eq: help_hellingerLB} results in
\begin{align*}
\|\theta  -\theta_{0,[k_n^*]}\|_2^2 &\lesssim  \frac{1}{n}( \theta-  \theta_{0,[k_n^*]})^T \Phi_{I_1}^T \Phi_{I_1} ( \theta  - \theta_{0,[k_n^*]})\\
&\lesssim h_n^2 (\theta,\theta_{0,[k_n^*]} ) \lesssim h_n^2(\theta,\theta_0)+h_n^2(\theta_0,\theta_{0,[k_n^*]})\\
&\lesssim  (n/\log n)^{-\frac{2\beta_0}{1+2\beta_0}}+ (k_n^*)^{-2\beta_0}.
\end{align*}
This implies in particular that $\|\theta - \theta_{0,[k_n^*]}\|_1 \lesssim \sqrt{k_n^*} (n/\log n)^{-\beta_0/(2\beta_0+1)}+(k_n^*)^{-\beta_0+1/2} = o(1)$. Hence in view of \eqref{eq: UB_Linfty}, $\|f_{\theta }\|_\infty\leq  \|f_{\theta_0}\|_\infty+o(1)$.

 It remains to prove \eqref{eq: small index_mtx}.
 First note that
\begin{equation*}
\Phi_{I_1}^T \Phi_{I_1} = \Phi_{k_n^*}^T \Phi_{k_n^*} - \Phi_{I_1^c}^T \Phi_{I_1^c}.
\end{equation*}
For all $j_1, j_2 \in \{1, ..., k_n^*\}$, take $\tilde\theta\in\Theta(k)$, $k\leq k_n^*$, such that $h^{2}_n(\tilde\theta,\theta_0)\leq \inf_{\theta\in\Theta(k)} h^{2}_n(\theta,\theta_0)+(n/\log n)^{-2\beta_0/(2\beta_0+1)}$, then in view of \eqref{eq: help_hellingerLB}
\begin{align*}
|(\Phi_{I_1^c}^T \Phi_{I_1^c} )(j_1, j_2) | &\leq \max_j \|\phi_j\|_\infty^2 |I_2\cup I_3|\lesssim  n h_n^2 (\tilde{\theta},\theta_{0,[k_n^*]} ) \\
& \lesssim nb(k) +(n/\log n)^{-2\beta_0/(2\beta_0+1)}n+ nh_n^2(\theta_0, \theta_{0,[k_n^*]})\\
& \lesssim nk^{-2\beta_0} + (n/\log n)^{-2\beta_0/(2\beta_0+1)}n=o(n/k),
\end{align*}
where in the last inequality we used that \eqref{eq: UBforBias_reg} implies $b(k)\lesssim k^{-2\beta}$. Dividing both sides with $n$ concludes the proof of the lemma.
\end{proof}

 \begin{lemma}\label{Lem: class:likeratio}
In the classification model \eqref{def: class} for all $M>0$ there exists a $B>0$ such that
\begin{align*}
P_{\theta_0}^{(n)}\big(\sup_{k\in\mathcal{K}_n(M)}\sup_{\theta\in \Theta_n(k)\cap B_k(\theta_0, M \eps_n,h_n)} \ell_n(\theta)-\ell_n(\theta_{[k]}^o)-Bk >0\big)=o(1).
\end{align*}
 \end{lemma}
 
 \begin{proof}
First of all note that in view of Lemma  \ref{lem: class_Linf}, for $\theta\in B_k(\theta_{0}, M \eps_n,h_n)$, $k\in\mathcal{K}_n(M)$ we have $\sup_i|f_{\theta}(x_i)|=O(1)$. Note that as a direct consequence there exists $0<c<C<1$ such that $c<\inf_{i}q_i\leq \sup_{i}q_i<C$. Furthermore following from the mean value theorem there exists $\bar{f}(x_i)\in [f_{\theta}(x_i),f_{\theta_{[k]}^o}(x_i)],\, i=1,2,...,n$ such that $|\bar{f}(x_i)|\leq C$ and
\begin{align}
\Big(\log q_i-\log q_i^o\Big)=\frac{\mu'(\bar{f}(x_i))}{\mu(\bar{f}(x_i))}  \big( f_\theta(x_i)-f_{\theta_{[k]}^o}(x_i)\big).\label{eq: help11}
\end{align}
The difference of the log-likelihood functions $\ell_n(\theta)-\ell_n(\theta_{[k]}^o)$ can be written as
\begin{align}
\sum_{i=1}^n& \big(\log q_i-\log q_i^o\big)y_i+\sum_{i=1}^n \big(\log (1-q_i)-\log (1-q_i^o)\big)(1-  y_i)\nonumber\\
&=\sum_{i=1}^n \big(\log q_i-\log q_i^o\big)(y_i- q_{0,i})+\sum_{i=1}^n \big(\log q_i-\log q_i^o\big)(q_{0,i}-q_i^o)\nonumber\\
&\quad+\sum_{i=1}^n q_i^o\big(\log q_i-\log q_i^o\big)+\sum_{i=1}^n \big(\log (1-q_i)-\log (1-q_i^o)\big)( q_i^o-q_{0,i})\nonumber\\
&\quad+\sum_{i=1}^n \big(\log (1-q_i)-\log (1-q_i^o)\big)(q_{0,i}-  y_i)\nonumber\\
 &\quad+\sum_{i=1}^n \big(\log (1-q_i)-\log (1-q_i^o)\big) (1-q_i^o)\nonumber\\
&=\mathcal{A}+\mathcal{B}+\mathcal{C}+\mathcal{D}+\mathcal{E}+\mathcal{F}. \label{eq: likeratio_help}
\end{align}
We deal with the six terms on the right hand side separately. 

First of all note that in view of Lemma \ref{lem:help_class}
\begin{align*}
\mathcal{B}+\mathcal{D}&=\sum_{i=1}^n (q_i^o-q_{0,i})\big(  \log\frac{q_i^o}{1-q_i^o}- \log\frac{q_i}{1-q_i}\big)\\
&= \sum_{i=1}^n (q_i^o-q_{0,i})(\theta_{[k]}^o-\theta)^T\Phi_k(x_i) =0.
\end{align*}
Next note that by Taylor expansion 
$$\log q_i=\log q_i^o+\frac{q_i-q_{i}^o}{q_{i}^o}-\frac{(q_i-q_{i}^o)^2}{2\bar{q}_{i}},\quad \text{for some $\bar{q}_{i}\in[q_i,q_i^o]\cup [q_i^o,q_i]$},$$
and as a consequence 
$$\mathcal{C}=\sum_{i=1}^n(q_i-q_i^o) -\sum_{i=1}^n (q_i-q_i^o)^2 \frac{q_i^o}{2\bar{q}_i}\leq \sum_{i=1}^n(q_i-q_i^o)-\frac{q^o_i}{2}\sum_{i=1}^n (q_i-q_i^o)^2.$$
By similar arguments one can also derive that $\mathcal{F}\leq - \sum_{i=1}^n(q_i-q_i^o)  -\frac{1-q^o_i}{2}\sum_{i=1}^n (q_i-q_i^o)^2$ and therefore $\mathcal{C}+\mathcal{F}\leq- \sum_{i=1}^n (q_i-q_i^o)^2/2$. Since both $\|f_{\theta}\|_{\infty}=O(1)$ and $\|f_{\theta_{[k]}^o}\|_{\infty}=O(1)$ by the mean value theorem there exist a constant $c>0$ such that $\sum_{i=1}^n (q_i-q_i^o)^2\geq cnd_n^2(\theta,\theta_{[k]}^o)$ which is further bounded from below by a multiple of $n\|\theta-\theta_{[k]}^o \|_2^2$ in view of \eqref{eq: norms}. Therefore we can conclude that $\mathcal{C}+\mathcal{F}\lesssim -n\|\theta-\theta_{[k]}^o\|_2^2$.

Following from \eqref{eq: help11} and Cauchy-Schwarz the term $\mathcal{A}$ is bounded by
\begin{align}
&\sup_{k\in\mathcal{K}_n(M)}
\Big|\sum_{i=1}^{n}\frac{\mu'(\bar{f}(x_i))}{\mu(\bar{f}(x_i))}  \big( f_\theta(x_i)-f_{\theta_{[k]}^o(x_i)}\big) \Big(y_i- q_{0,i}\Big)\Big|\nonumber\\
&\quad =\sup_{k\in\mathcal{K}_n(M)} \Big| \sum_{j=1}^k \sum_{i=1}^{n}\frac{\mu'(\bar{f}(x_i))}{\mu(\bar{f}(x_i))}  \big(\theta_j-\theta_{[k],j}^o\big)\phi_j(x_i) \Big(y_i- q_{0,i} \Big) \Big|\nonumber\\
&\quad \leq\sup_{k\in\mathcal{K}_n(M)} \Big(\sum_{j=1}^k(\theta_j-\theta_{[k],j}^o)^2 \Big)^{1/2}\nonumber\\
&\qquad \quad\times \sup_{k\in\mathcal{K}_n(M)} \Big(\sum_{j=1}^k \big[\sum_{i=1}^{n}\frac{\mu'(\bar{f}(x_i))}{\mu(\bar{f}(x_i))}\phi_j(x_i) \big(y_i- q_{0,i} \big)\big]^2\Big)^{1/2}.\label{eq: help13}
\end{align}
Note that the second term on the right hand side of \eqref{eq: help13} is increasing in $k$, hence over $\mathcal{K}_n(M)$ it takes its maximum at $2M^2k_n$, see Lemma \ref{rem:Kn} in \cite{rousseau:szabo:16:main}. Also note that the centered independent random variables $\big(y_i- q_{0,i} \big)$ have bounded second moments, so
\begin{align*}
E_{\theta_0}\big[\sum_{i=1}^{n}\frac{\mu'(\bar{f}(x_i))}{\mu(\bar{f}(x_i))}\phi_j(x_i) \big(y_i- q_{0,i} \big)\big]^2\leq\sum_{i=1}^{n} \frac{\mu'(\bar{f}(x_i))^2}{\mu(\bar{f}(x_i))^2}\phi_j^2(x_i) E_{\theta_0}\big(y_i- q_{0,i} \big)^2\lesssim n.
\end{align*}
Then by Markov's inequality we get that for every $\eps>0$ there exists a large enough constant $C_{\eps}$ such that the second term on the right hand side of $\eqref{eq: help13}$  is with $P_{\theta_0}$-probability larger than $1-\eps$ bounded from above by $C_\eps \sqrt{n k_n}$. We conclude that the term $\mathcal{A}$ is bounded from above by a large enough constant times $\sqrt{n k_n} \| \theta-\theta^o_{[k]}\|_2$ with $P_{\theta_0}$-probability larger than $1-\eps$. Similar arguments lead to the same upper bound for the term $\mathcal{E}$. 

Therefore by putting together all the preceding upper bounds the right hand side of \eqref{eq: likeratio_help} is bounded from above with $P_{\theta_0}$-probability at least $1-2\eps$ by a multiple of
\begin{align*}
\sqrt{nk_n}\|\theta-\theta_{[k]}^o\|_2-cn\|\theta-\theta_{[k]}^o\|_2^2\lesssim \sqrt{n}\|\theta-\theta_{[k]}^o\|_2( \sqrt{k_n}-c\sqrt{n}\|\theta-\theta_{[k]}^o\|_2 )\lesssim k_n,
\end{align*}
finishing the proof of the lemma.
 \end{proof}
 
 \begin{lemma}\label{lem:help_class}
   If $\theta_0 \in \mathcal S^{\beta_0}(M_0)$ for some $M_0>0$ and $\beta_0>1/2$,  then $\theta_{[k]}^o$ (and the corresponding $q^o_i=q_{\theta_{[k]}^o}(x_i)$), $k\in\mathcal{K}_n(M)$ satisfies
\begin{equation}\label{KLproj:classif}
 \quad \sum_{i=1}^n (q_{0,i} - q^o_i)\phi_j(x_i) = 0, \quad \text{for all $j=1,...,k$}.
\end{equation}
Furthermore, if $K_n\gg n^{\frac{1}{2(\beta_0-1/2)}}$ then for all $k\leq K_n$ we have $\|f_{\theta_{[k]}^o}\|_{\infty}=O(1)$.

\end{lemma}


\begin{proof}[Proof of Lemma  \ref{lem:help_class}]

Equation \eqref{KLproj:classif} is a direct consequence of the definition of $\theta_{[k]}^o$.

Note that 
\begin{align}
h_n^2(\theta_{[k]}^o , \theta_0) \leq KL(  \theta_0, \theta_{[k]}^o) \leq KL( \theta_0,\theta_{0,[k]}) \lesssim d_n^2( \theta_0,\theta_{0,[k]}).\label{eq: UB:hellinger:class}
\end{align}
We have also that
\begin{equation*}
\begin{split}
d_n( \theta_0,\theta_{0,[k]}) &\leq  d_n( \theta_{0,[K_n]},\theta_{0,[k]})+ d_n(\theta_{0,[K_n]}, \theta_0) \lesssim \|\theta_{0,[K_n]}-\theta_{0,[k]}\|_2 + K_n^{-(\beta-1/2)} \\
 & \lesssim k^{-\beta} + K_n^{-(\beta-1/2)} = O((n/\log n)^{-\beta_0/(2\beta_0+1)})
\end{split}
\end{equation*}
 since $\|\theta_0\|_1<+\infty $ and $k \in  \mathcal K_n(M)$. Thus $h_n(\theta_{[k]}^o , \theta_0) \leq(n/\log n)^{-\beta_0/(2\beta_0+1)}$ and applying Lemma \ref{lem: class_Linf} with $k \lesssim k_n \lesssim (n/\log n)^{1/(2\beta_0+1)} = o((n/\log n)^{2\beta_0/(2\beta_0+1)})$ we obtain that 
 $$\|f_{\theta_{[k]}^o}\|_\infty = O(1).$$

\end{proof} 

\section{Proofs of the remaining results}

\subsection{Proof of Lemma \ref{lem:post:Kn}} \label{sec:proof:lemKn}
Let $m_n(k)  = \int_{\Theta(k)} e^{\ell_n(\theta) - \ell_n(\theta_0) }\pi_{|k}(d\theta)$ and $\eps_n=\eps_n(k_n)$, then we have

\begin{equation*}
\begin{split}
\pi_k(k|\mathbf Y )&= \frac{ \pi_k(k) m_n(k)}{  \sum_{k'}\pi_k(k' )m_n(k')}. 
\end{split}
\end{equation*}
Next we give bounds for the marginal likelihood function, starting with a lower bound for $m_n(k_n)$.

Recall from the proof of Lemma \ref{sec:pr:thprior} that 
\begin{align}
\Omega_{n,0}= \left\{m_n(k_n)>  e^{- (c_{3}+c_4+1) n\eps_n^2}    \right\}\label{eq: LB for denominator theta0}
\end{align}
satisfies $P_{\theta_0}^{(n)}(\Omega_{n,0}^c)\leq (k_n\log n)^{-r/2} =o(1)$. Let $D= c_2+c_{3}+ c_4+3/2$, then  we show below that in view of condition \textbf{A3} (with $\gamma=5/2$) and 
\textbf{A2} () with $\bar K_n = A k_n \log n $ and $A = (c_2+2c_3+2c_4+3)/c_1$,  
\begin{equation}
 P_{\theta_0}^{(n)} \left( \int_{\Theta(k)} e^{\ell_n(\theta) - \ell_n(\theta_0) }d\pi_{|k}(\theta) > e^{-  Dn\epsilon_n^2  } \right) 
 \lesssim e^{-c n\epsilon_n^2  }\label{UB_ like_help}
 \end{equation}
for all $k\leq A k_n \log n$, $k \notin \mathcal K_n(M)$ (with some large enough choice of $M$, for instance $M^2\geq 2M_0\vee 2c_5^{-1}(D+1/2)/(J_1\wedge 1)$ is sufficiently large),  and some $c>0$.  Furthermore, let us introduce the notations $\Omega_n(k) = \{  m_n(k) \leq e^{- D n\epsilon_n^2} \}$ and 
 $$\Omega_n = \Omega_{n,0} \cap_{k <A k_n\log n; k \notin \mathcal K_n(M) } \Omega_n(k) .$$
 Then $P_{\theta_0}^{(n)}\left( \Omega_n^c \right)= o(1) $ since $k_n\log n = o(e^{cn\epsilon_n^2})$ for any $c>0$ and
on $\Omega_n$
\begin{equation*}
\begin{split}
\pi_k(\{k < Ak_n \log n\}\cap \mathcal K_n(M)^c |\mathbf Y ) &\leq  \sum_{k < A k_n\log n}\1_{k \notin \mathcal K_n(M)} \frac{ e^{- D n\epsilon_n^2  }  \pi_k(k)}{\pi_k(k_n) m_n(k_n)}\\
&\lesssim  e^{- D n\epsilon_n^2 +(c_{3}+c_4+1) n\eps_n^2+c_2k_n\log n }\\
& \leq e^{- (D-c_2-c_{3}-c_4-1) k_n \log n}\\
& \leq e^{- (k_n/2) \log n} = o(1) .
\end{split}
\end{equation*}
For $k\geq A k_n \log n$ we also obtain that
 \begin{equation*}
 \begin{split}
 E_{\theta_0}^{(n)}& \pi_k(k \geq A k_n \log n|\mathbf Y )  \\
  &\leq e^{ (c_2+2c_{3}+2c_4+2 ) k_n \log n  } E_{\theta_0}^{(n)}\left(  \sum_{k\geq Ak_n \log n   } \pi_k(k) m_n(k) \right)+ P_{\theta_0}^{(n)}\left( \Omega_{n,0}^c \right) \\
 &\lesssim \pi_k(k \geq Ak_n \log n )  e^{ (c_2+2c_{3}+2c_4+2 ) k_n \log n  }+o(1)\\
 &\lesssim e^{-(c_1 A- c_2-2c_{3}-2c_4-2 ) k_n \log n} + o(1)=o(1).
 \end{split}
 \end{equation*}

It remains to verify $\eqref{UB_ like_help}$ for all $k \notin \mathcal K_n(M)$, $k\leq A k_n \log n$. 
If $b(k) > k(\log n)/n$ we have for all $\theta \in \Theta(k)$ that  $d^2(\theta_0, \theta) \geq b(k)>k (\log n)/n$, hence $d^2(\theta_0, \theta)> \epsilon_n^2(k)/2 \geq M^2 \epsilon_n^2/2$.
If  $b(k) \leq k (\log n)/n$ then $k \geq \frac{n\eps_n^2(k)}{2\log n} \geq \frac{M^2 n \epsilon_n^2}{2\log n}\geq \frac{M^2k_n}{2}$. 
Hence, for all $k \notin \mathcal K_n(M)$ such that $b(k) > k(\log n)/n$ we have 
 \begin{equation}\label{empty}
  \Theta_n(k) \cap B_k(\theta_0,  M^2 \epsilon_n^2/2,d) = \emptyset
  \end{equation}
  and for all $k \notin  \mathcal K_n(M)$ such that $b(k) \leq k(\log n)/n$ we use assumption \textbf{A3}, choosing $\gamma=2$ and $M^2 > 2M_0$.

Then by slightly abusing our notation, consider slices $\Theta_j(k) = \{ j \epsilon_n \leq d(\theta_0, \theta) \leq  (j+1) \epsilon_n\}$, $j\geq J_0(k)$ of $\Theta(k)$, where $J_0(k) =M /\sqrt{2} $ if $b(k) > k(\log n)/n$ and $J_0(k) =  J_1 \sqrt{k(\log n)/n} \epsilon_n^{-1}$ if $b(k) \leq k(\log n)/n$ (note that $J_0(k)\geq J_0$ for lage enough choice of $M$). Let us consider a minimum cover of the slice $\Theta_j(k)$ with $c_6j\eps_n$-radius $d$-balls and  denote by $\{\theta_{ji}$, $i\leq N_{n,j}(k) \}$ a collection of the centers of such balls. Next for each $\theta_{ji}$ consider the individual test $\phi_n(j,i)$ defined in assumption  \textbf{A2} (ii) satisfying
\begin{align*}
E_{\theta_0}^{(n)}\left( \phi_{n}(j,i)  \right) \leq e^{- c_5nj^2 \epsilon_n^2 } ,\\
 \quad \sup_{d(\theta_{ji}, \theta)\leq c_6 j \epsilon_n } E_{\theta}^{(n)}\left(1- \phi_{n}(j,i)  \right) \leq e^{- c_5nj^2 \epsilon_n^2 },
\end{align*}
$j\geq J_0(k)$ and construct $\phi_n(k) = \max_{j\geq J_0(k)} \max_{i\in\{1,...,N_{n,j}(k)\}} \phi_{n}(j,i) $.  
Assumption \textbf{A2} (iii) implies that $\log N_{n,j}(k) \leq c_5 j^2 n\epsilon_n^2/2$ for $k \notin \mathcal K_n(M)$, $k\leq Ak_n\log n$ and for all $j \geq J_0(k)$. Therefore
\begin{align*}
E_{\theta_0}^{(n)}\left( \phi_n(k) \right) \leq \sum_{j\geq J_{0}(k)}  e^{- c_5 nj^2\epsilon_n^2/2 } \leq  2e^{- c_5J_{0}(k)^2n\epsilon_n^2/2},\nonumber\\
 \sup_{\theta \in \Theta_n(k)\cap B_k^c(\theta_0,J_{0}(k)\eps_n,d)} E_{\theta}^{(n)}\left(1- \phi_{n}(k)  \right) \leq e^{- c_5J_{0}(k)^2n\epsilon_n^2}.
\end{align*}
We have for all $b(k)>k(\log n)/n$ with $k\notin \mathcal K_n(M)$, $J_{0}(k)^2n\epsilon_n^2 = M^2 n \epsilon_n^2/2$ and if $b(k)\leq  k(\log n)/n$ 
$$J_{0}(k)^2n\epsilon_n^2 = J_1^2 k \log n \geq \frac{J_1^2 n\eps_n^2(k) }{ 2} \geq \frac{J_1^2 M^2 n\epsilon_n^2 }{ 2} $$ 
so that for all $k \notin \mathcal K_n(M)$ by choosing $M^2\geq 2c_5^{-1}(D+1/2)/(J_1^2\wedge1)$ we have $c_5 J_0(k)^2\geq D+1/2$. Hence by applying Markov's inequality and using assumption \textbf{A2} (i)
\begin{equation} \label{eq: large_k}
\begin{split}
& P_{\theta_0}^{(n)} \left( \int_{\Theta(k)} e^{\ell_n(\theta) - \ell_n(\theta_0) }\pi_{|k}(d\theta) > e^{- D n\epsilon_n^2  } \right)\\
&\leq E_{\theta_0}^{(n)}(\phi_n(k))+ e^{ Dn\epsilon_n^2 }\pi_{|k}\big(\Theta_n(k)\cap B_k(\theta_0, J_0(k)\epsilon_n,d) \big) + \\
&\quad+ e^{ Dn\epsilon_n^2 }\pi_{|k}\left( \Theta_n(k)^c \right) +e^{ Dn\epsilon_n^2 }\int_{\Theta_n(k)\cap B_k^c(\theta_0,J_0(k)\epsilon_n,d)} E_\theta^{(n)}(1-\phi_n(k))\pi_{|k}(d\theta)\\
& \lesssim e^{-n\epsilon_n^2/2} .
\end{split}
\end{equation}
This terminates the proof of Lemma \ref{lem:post:Kn}. 

We note that the above computations actually imply that 
 \begin{equation}
 \sup_{\theta_0\in \bar\Theta} E_{\theta_0}^{(n)}\big( \pi_k( k \notin \mathcal K_n(M) |\mathbf Y )\big) \lesssim1/(n\eps_n^2). \label{eq: post:modelselect:asymptotic}
 \end{equation}

\subsection{Proof of Lemma \ref{lem:rate}}\label{sec:pr:lemrate}
In view of Lemma \ref{lem:post:Kn}
\begin{align*}
E_{\theta_0}^{(n)}&\pi\big( d(\theta, \theta_0)>\bar{M} \epsilon_n|\mathbf Y \big)\\
&  \leq E_{\theta_0}^{(n)}\pi\big( \{d(\theta, \theta_0)>\bar{M} \epsilon_n\} \cap \{k \in \mathcal K_n(M)\}  |\mathbf Y \big)  +\eps/2,
\end{align*}
for sufficiently large $M>0$ uniformly over $\bar\Theta$, where we use again the abbreviation $\eps_n=\eps_n(k_n)$.


As in the proof of Lemma \ref{lem:post:Kn} define the tests $\phi_n(j,i)$ with $j \geq\bar{M}$,
$$\phi_n = \max_{k \in \mathcal K_n(M)}\max_{j \geq \bar{M}} \max_{i\leq N_{n,j}(k)} \phi_n(j,i), \quad N_{n,j}(k) \leq \exp( c_5 n j^2 \epsilon_n^2/2).$$
We have, in view of Lemma \ref{rem:Kn}, that $k\leq 2M^2k_n$ for every $k\in\mathcal{K}_n(M)$ hence in view of condition \textbf{A2} (ii)
$$E_{\theta_0}^{(n)}(\phi_n)\leq \sum_{k\in \mathcal K_n(M)} \sum_{j\geq \bar{M}} e^{-c_5 n j^2 \epsilon_n^2/2} \leq 4 M^2 k_n \exp\Big\{\frac{-c_5 n \bar{M}^2 \epsilon_n^2}{2}\Big\}$$
and for $\theta\notin B_k(\theta_0, \bar{M} \epsilon_n ,d)$ with $k\in\mathcal{K}_n(M)$ 
$$E_{\theta}^{(n)}(1-\phi_n)\leq \exp\Big\{-c_5 n \bar{M}^2 \epsilon_n^2\Big\}.$$
Then, in view of the above assertions together with \eqref{eq: LB for denominator theta0}
\begin{equation*}
\begin{split}
E_{\theta_0}^{(n)}\big( \pi( d(\theta, \theta_0)>\bar{M} \epsilon_n& |\mathbf Y ) \big) \leq P_{\theta_0}^{(n)}(\Omega_{n,0}^c ) +E_{\theta_0}^{(n)}(\phi_n)+\eps/2 \\
& + 
 \exp\big\{  -\big( c_5  \bar{M}^2 - c_2-c_{3}-c_4-1\big)n\epsilon_n^2  \big\}
\end{split}
\end{equation*}
which is bounded from above by $\eps$ for a sufficiently large choice of $\bar{M}$. 
 
Here we also note that the above computations in view of \eqref{eq: post:modelselect:asymptotic} actually imply
\begin{align}
\sup_{\theta_0\in \bar\Theta} E_{\theta_0}^{(n)}\big(\pi\left( d(\theta, \theta_0)\geq C_{\eps}\epsilon_n|\mathbf Y \right)\big)\leq C/(n\epsilon_n^2),\label{eq: post:contraction:asymptotic}
\end{align}
for some sufficiently large constant $C>0$.

\subsection{Proof of Theorem \ref{th:EmpBayesCoverage}}\label{sec:proof:covEB}
First we show that for large enough $M>0$
\begin{align}
P_{\theta_0}^{(n)}(\hat{k}_n\notin \mathcal{K}_n(M) )\leq \eps.\label{eq: MMLE}
\end{align}
In view of \eqref{eq: large_k}, with $D$ replaced by $c_2+c_{3}+ c_4+3/2+H$ using assumption \textbf{A3} with $\gamma = H+5/2$
\begin{equation*}
\begin{split}
 P_{\theta_0}^{(n)} \left(\sup_{k\notin\mathcal{K}_n(M), k\leq \bar K_n}m_n(k) > e^{- (c_2+c_{3}+ c_4+3/2+H) n\epsilon_n^2  } \right)
& \lesssim  \bar K_n e^{-(H+1/2)n\epsilon_n^2}\\
& \leq n^H e^{-(H+1/2)k_n\log n}
\end{split}
 \end{equation*}
 where the right hand side tends to zero for $k_n\geq 1$. Furthermore, in view of \eqref{eq: LB for denominator theta0} we get that 
$$P_{\theta_0}^{(n)}\Big(\sup_{k\notin\mathcal{K}_n(M)}m_n(k)> m_n(k_n)\Big)=o(1)$$
leading to \eqref{eq: MMLE}.

Next by using the notation \eqref{def:hyper_post_prob} and following from \eqref{eq: UB_hyper_post} and \eqref{eq: MMLE} we have with $P_{\theta_0}^{(n)}$-probability at least $1-\eps$
\begin{align*}
  \pi_{|\hat k_n}\left( \{d(\theta, \hat \theta_n)\leq \rho_n \epsilon_n \}\cap \Theta_n |\mathbf Y \right) 
\leq  \sum_{k\in\mathcal{K}_n(M)} \pi_{n,k}	\leq C k_n e^{-c'k_n\log \delta_\eps^{-1}} \leq\eps,
 \end{align*}
for sufficiently small choice of $\delta_\eps>0$.
Then the proof of the first statement automatically follows from the proof of Theorem \ref{th:coverage}. The proof of the second statement follows by similar lines of reasoning as above combined with the proof of Lemma \ref{lem:rate}.

\subsection{Proof of Lemma \ref{rem:Kn}} \label{sec:rem:Kn}

Note that if $k \in \mathcal K_n(M)$, then 
 $$b(k) + \frac{k\log n }{ n } \leq M^2 \left( b(k_n) +\frac{k_n\log n }{ n }\right)\leq 2M^2 \frac{k_n\log n }{ n },$$
 so that $k \leq 2M^2 k_n$  verifying the first statement of the lemma.

Now assume that $\theta_0\in  \Theta_{0,n}(R_0, k_0,\tau)$ and that $k_n > 2R_0^{m+1}\vee R_0^m k_0 A_0$. Then first note that for all $\delta >0$ and all $\delta k_n \leq k < k_n$ then 
 $$b(k) > \frac{ k \log n }{ n} \geq \frac{ \delta k_n \log n}{ n} \geq \frac{ \delta \epsilon_n^2(k_n )}{2}, $$
and for all $k_0\leq k \leq R_0^{-m} (k_n-1)$ we have $b(R_0^m k ) \leq \tau^m b(k) $.

Let us distinguish three cases $R_0^{-m-1}(k_n-1)- 1\leq k\leq R_0^{-m} (k_n-1)$, $ k_0 \leq k< R_0^{-m-1}(k_n-1)-1$ and $k<k_0$ and we show that in each case $k\notin \mathcal{K}_n(M)$.

If $k \geq  R_0^{-m-1}(k_n-1)- 1$ then 
$$R_0^m k \geq (k_n-1)/R_0-R_0^m\geq \frac{ k_n-1}{ 2 R_0} \geq \frac{ k_n }{ 4 R_0} $$  so that 
$$b(R_0^m k) \geq \frac{k_n \log n}{4 R_0 n} \geq \frac{  \epsilon_n^2(k_n )}{8R_0} , \quad \text{and} \quad b(k) \geq \tau^{-m}\frac{  \epsilon_n^2(k_n )}{8R_0} \geq M^2\epsilon_n^2(k_n ),$$
hence $k\notin \mathcal{K}_n(M)$.

If $ k_0 \leq k< R_0^{-m-1}(k_n-1)-1$, define $j^*(k) = \min\{ j:\, R_0^j k \geq R_0^{-m-1}(k_n-1)- 1\}$, so that 
 $$k_n -1  - R_0^{m+1} > R_0^{m+j^*(k)} k > \frac{ k_n-1}{R_0}-R_0^m > \frac{ k_n}{4R_0}$$
 which implies that 
 $$b( R_0^{m+j^*(k)} k ) \geq \frac{k_n \log n}{4R_0 n} \geq\frac{  \epsilon_n^2(k_n )}{8R_0} \quad \text{and} \quad b(k) \geq \tau^{ -j^*(k)} M^2\epsilon_n^2(k_n ),$$
and therefore $k\notin \mathcal{K}_n(M)$.

If $k\leq k_0$ then by assumption \eqref{eq: cond:bias} the inequality $b(k) \geq b(k')$ holds for some $k_0 \leq k' < k_0 A_0$. Since
 $k_0 A_0\leq R_0^{-m} (k_n-1)$ we get $b(k') \geq  M^2\epsilon_n^2(k_n )$, which again concludes that $k\notin \mathcal{K}_n(M)$.

Finally, note that if $k_n\leq 2R_0^{m+1}\vee R_0^m k_0 A_0$ then $1\geq k_n/(2R_0^{m+1}\vee R_0^m k_0 A_0)$.

\subsection{Proof of Remark \ref{rem: polished:tail:regression}}\label{sec: polished:tail:regression}

First of all note that
$$b(k)\leq d_n^2(\theta_0, \theta_{0,[k]}) \leq \frac{ 2 }{ n } ( \theta_{0,[K_n]} - \theta_{0,[k]})^T \Phi_{K_n}^T\Phi_{K_n} ( \theta_{0,[K_n]} - \theta_k) + 2d_n^2(\theta_{0,[K_n]} , \theta_0)$$
where $\theta_{0,[k]} $ in $\mathbb R^{K_n}$ is to be understood as the completion of $\theta_{0,[k]}\in\mathbb{R}^k$ by zeros. 
We then have, 
\begin{align*}
b(R_0 k ) \leq 2 C_0 \| \theta_{0,[K_n]} - \theta_{0,[R_0k]}\|_2^2 + 2d_n^2(\theta_{0,[K_n]} , \theta_0)
\end{align*}

 and as soon as $K_n \geq R_0^{a_n} k$ for some sequence $a_n$ tending to infinity arbitrarily slowly   
\begin{equation*}
\begin{split}
b(R_0 k ) &\leq 2 C_0\tau_1 \| \theta_{0,[K_n]} - \theta_{0,[k]}\|_2^2 + 2C_0(1-\tau_1)\| \theta_{0,[K_n]} - \theta_{0}\|_2^2+ 2d_n^2(\theta_{0,[K_n]} , \theta_0)\\
&\leq 2 C_0\tau_1(1+\tau_1^{a_n-1}) \| \theta_{0,[K_n]} - \theta_{0,[k]}\|_2^2 +2d_n^2(\theta_{0,[K_n]} , \theta_0).
\end{split}
\end{equation*}

We show below that in the fixed design regression case for $K_n\geq n^{\frac{1}{2(\beta_0-1/2)}}$ and in the random regression case (for every $K_n$ tending to infity)  with $\nu$-probability at least $1-\eps$ (for arbitrary $\eps>0$)
\begin{align}
d_n^2(\theta_{0,[K_n]},\theta_0)=o(\tau_1 \|\theta_{0,[K_n]}-\theta_{0,[k]}\|_2^2)\label{eq: help:PT}
\end{align}
holds.
Then by noting that
\begin{equation*}
\begin{split}
\| \theta_{0,[K_n]} - \theta_{0,[k]}\|_2^2 &\leq \| \theta_{0,[K_n]} - \theta_{[k]}^o\|_2^2 \leq C_0 d_n^2(\theta_{0,[K_n]}, \theta_{[k]}^o)\\
& \leq 2C_0[b(k)  + d_n^2(\theta_0, \theta_{0,[K_n]})]\\
& \leq 2C_0[b(k)  + o(\tau_1\| \theta_{0,[K_n]} - \theta_{0,[k]}\|_2^2)]
\end{split}
\end{equation*}
we have $b(R_0k) \leq 5C_0^2 \tau_1 b(k)$ for all $ k_0\leq k \leq R_0^{-a_n} K_n$ with $a_n $ going to infinity arbitrarily slowly. Therefore $f_{\theta_0}$  satisfies the polished tail condition associated to $d_n(.,.)$ with $\tau = 5C_0^2 \tau_1$ as soon as $\tau_1 < 1/(5C_0^2)$ (in the random design regression case this holds with $\nu$-probability arbitrarily close to one).

It remained to prove assertion \eqref{eq: help:PT}. 
In the random design regression case let $C>0$ be an upper bound on the density of $\nu$, then 
 $$ \mathbb \nu[ d_n(\theta_{0,[K_n]} , \theta_0) \geq t ] \leq  C\frac{ \| \theta_{0,[K_n]} - \theta_{0}\|_2^2 }{ t^2 } $$
so that with probability greater than $1-\epsilon$ , $d_n^2(\theta_{0,[K_n]} , \theta_0)\leq C\epsilon^{-1} \| \theta_{0,[K_n]} - \theta_{0}\|_2^2$ and
\begin{equation*}
\begin{split}
b(R_0 k ) &\leq 2 C_0\tau_1(1+(1+C\eps^{-1})\tau_1^{a_n-1}) \| \theta_{0,[K_n]} - \theta_{0,[k]}\|_2^2 .
\end{split}
\end{equation*}
While in the fixed design regression case  we have
\begin{align*}
\|\theta_{0,[k]}-\theta_{0,[K_n]}\|_2^2&\asymp d_n(\theta_{0,[k]},\theta_{0,[K_n]})\geq b(k)/2-b(K_n)\\
&\geq k(\log n)/n-K_n^{-2(\beta_0-1/2)}\gg K_n^{-2(\beta_0-1/2)}\geq b(K_n),
\end{align*}
finishing the proof of \eqref{eq: help:PT}, where the last inequality follows from Remark \ref{rem:Kn2}  in the main text \cite{rousseau:szabo:16:main}. 

  In particular it implies that if $k_n \leq R_0^{-a_n} K_n$ for some sequence $a_n $ going to infinity, then $b(K_n) \leq  \epsilon^{-1} \| \theta_{0,[K_n]} - \theta_{0}\|_2^2 \leq \epsilon^{-1} \tau_1^{a_n} \| \theta_{0,[k_n]} - \theta_{0}\|_2^2 \lesssim \tau_1^{a_n}\epsilon^{-1}b(k_n) \leq \delta K_n (\log n)/n$. In the random design regression case this holds with $\nu$-probability greater than $1-\epsilon$.

\subsection{Lepski's method for choosing the centering point}
Let us consider the Gaussian sequence model
\begin{align*}
X_i=\theta_{0,i}+n^{-1/2}Z_i,\quad Z_i\stackrel{iid}{\sim}N(0,1),\,i=1,2,....
\end{align*}
We endow the parameter $\theta\in\ell_2$ with a prior $\prod_{i=1}^{k}g\otimes \delta_0\times\delta_0\times...$, where $g$ satisfies \eqref{eq: sieve_prior}, and choose $k$ via the maximum marginal likelihood estimator. 
Next we describe (a version of) Lepski's method for constructing the center of the credible set. Let us introduce the notation $\hat\theta_{[k]}=(X_1,...,X_{k},0,...)$ for the truncation estimator at level $k$ of the data $X=(X_1,X_2,...)$. Then Lepski's estimator for the optimal truncation level is
\begin{align*}
\tilde{k}=\min_{k\geq \log n}\{\|\hat\theta_{[k]}-\hat\theta_{[j]}\|_2^2\leq t j/n, \quad\forall  k\leq j\leq n\},
\end{align*}
for some parameter $t>0$ to be specified later. We denote by $\tilde{\theta}=\hat{\theta}_{[\tilde{k}]}$ the corresponding Lepski's estimator.
Then we consider the credible ball $\hat{C}=\{\theta\in\Theta(\hat{k}_n):\, \|\theta-\tilde\theta\|_2\leq r(\hat{k}_n) \}$, where $\hat{k}_n$ is the MMLE, the radius $r(\hat{k}_n)$ is chosen such that the set $\hat{C}$ accumulates $95\%$ of the empirical Bayes posterior mass (i.e. $\pi_{\hat{k}_n}(\hat{C}|X)= 0.95$) and we again enlarge the set with a multiplicative factor $L\geq 1$. We  show below that under the polished tail condition the so constructed set $\hat{C}(L)$ achieves good frequentist coverage for a constant multiplicative factor $L$.

\begin{lemma}\label{lem: Lep:center}
In the Gaussian sequence model the credible set $\hat{C}(L)$ constructed above satisfies for some large enough $L$ that
\begin{align*}
\inf_{\theta_0\in\Theta_0}P_{\theta_0}( \theta_0\in \hat{C}(L))\geq 0.95.
\end{align*}
\end{lemma}

\begin{proof}
First of all let us introduce the notation
$$k^*=\max_{k\geq \log n}\{ \sum_{i=k+1}^{\infty}\theta_{0,i}^2\geq k/n\}.$$
From the proof of Lepski's method (e.g. \cite{lepski:90,gine:nickl:2016}) we know that $P_{\theta_0}(\tilde{k}>k^*+1)=o(1)$ and $E_{\theta_0}\|\theta_0-\hat\theta_n\|_2^2\lesssim k^*/n$. Therefore it remained to show that the radius of the credible set is bounded from below by constant times $k^*/n$.

 First we show that under the polished tail condition we also have with high probability that $\tilde{k}\gtrsim k^{*}$.  Note that in view of the definition of $k^*$ and the polished tail condition we have
\begin{align*}
\tau^{-r}k^* /n\leq  \tau^{-r}\sum_{i=k^*+1}^{\infty}\theta_{0,i}^2\leq \sum_{i=R_0^{-r}k^*+1}^{\infty}\theta_{0,i}^2,
\end{align*}
for arbitrary $r\in\mathbb{N}$ satisfying that $k^*R_0^{-r}\geq k_0$.
Therefore by using the inequality $(a+b)^2\geq  a^2/2 -b^2$ we get that
\begin{align*}
P_{\theta_0}( \tilde{k}\leq k^* R_0^{-r})&\leq P_{\theta_0}\big( \sum_{i=R_0^{-r}k^*+1}^{k^*}X_i^2\leq tk^*/n\big)\\
&\leq P_{\theta_0}\big(n^{-1}\sum_{i=1}^{k^*}Z_i^2\geq 2^{-1}\sum_{i=R_0^{-r}k^*+1}^{k^*}\theta_{0,i}^2-tk^*/n\big)\\
&\leq P_{\theta_0}\big(\sum_{i=1}^{k^*}Z_i^2\geq ((\tau^{-r}-1)/2-1-t)k^*\big)
=o(1),
\end{align*}
by taking $r$ large enough ($\tau^{-r}> 2(2+t)+1$ is sufficiently large), following from the standard concentration inequality on chi-square distributions,  see for instance Theorem 4.1.9 of \cite{gine:nickl:2016}. For convenience let us introduce the notation $c^*=(\tau^{-r}-1)/2-1-t>0$. Furthermore, recall that under the polished tail condition the MMLE estimator $\hat{k}_n$ satisfies that $P_{\theta_0}(c_1k_n\leq\hat{k}_n\leq c_2 k_n)=1-o(1)$, for some positive constants $c_1$ and $c_2$.

Next let us consider two cases depending on the relationship of $k_n$ and $k^*$. If $k_n\geq c^*k^*/(2c_2)$, then since the squared radius of the credible ball is bounded from below by a multiple of $k_n/n\gtrsim k^*/n\gtrsim E_{\theta_0}\|\theta_0-\hat\theta_n\|_2^2$, the true parameter $\theta_0$ will be inside of the inflated credible ball, for large enough inflation factor $L>0$, with high probability. For $k_n\leq c^*k^*/(2c_2)$ we have
\begin{align*}
\inf_{\theta\in \mathbb{R}^{\hat{k}_n}}\|\hat\theta-\theta\|_2^2\geq \inf_{\theta\in \mathbb{R}^{c_2 k_n}}\|\hat\theta-\theta\|_2^2\geq  \sum_{i=c_2k_n+1}^{ c^*k^*}X_i^2
\geq \sum_{i=(c^*/2)k^*+1}^{ c^*k^*}X_i^2,
\end{align*}
with probability tending to one. Furthermore, in view of Anderson's lemma we have that
\begin{align*}
P_{\theta_0}\Big( \sum_{i=(c^*/2)k^*}^{ c^*k^*}(Z_i+\theta_{0,i})^2\leq  (c^*/4)k^* \Big)\leq P_{\theta_0}( \sum_{i=(c^*/2)k^*}^{ c^*k^*}Z_i^2\leq  (c^*/4)k^* ),
\end{align*}
where the right hand side tends to zero, again by using the standard concentration inequality on the chi square distribution. Hence the squared radius of the credible ball is at least $(c^*/4)k^*/n$ as the posterior puts all of its mass to the space $\mathbb{R}^{\hat{k}_n}$ (more precisely to the space $\mathbb{R}^{\hat{k}_n}\times\delta_0\times\delta_0\times...$), finishing the proof of our statement.

\end{proof}

\subsection{ Proof of Equation \ref{LB:para} of Section \ref{parametric} } \label{pr:LB:para}

We consider the following set of assumptions: $\mathbf Y = (Y_1,\cdots, Y_n) $ with $Y_i \stackrel{iid}{\sim} f_{\theta_0}$.
Assume that $\theta_0 \in \Theta(k_0) $ for some $k_0 >1$ and that for all $k <k_0 $ there exists a unique $\theta_{[k]}^o $  minimizing the Kullback-Leibler divergence between $f_{\theta_0}$ and $f_\theta$, $\theta \in \Theta(k)$, 
\begin{itemize}
\item (i) For all $\epsilon > 0$ 
 $$ \Pi( \| \theta - \theta_0\| <\epsilon |\mathbf Y) = 1 + o_{P_{\theta_0}}(1).$$
 \item (ii) For all $k$, $\theta \rightarrow \log f_\theta(y) $ is twice continuously differentiable on $\Theta(k) $ and there exists $\delta_0>0$ such that for all $k$, on $\Theta(k)$, 
 \begin{equation*}
 E_{\theta_0} \left( \sup_{\| \theta' - \theta_{[k]^o} \| \leq \delta} \|  \frac{ \partial^2\log f_{\theta'}(Y_1) }{\partial^2 \theta'}\| \right) <\infty. 
 \end{equation*} 
 \item (iii) For all $k$
  $$ \ell_n(\hat \theta_k) - \ell_n( \theta_{[k]}^o ) = O_{P_{\theta_0}}(1).$$ 
  \item (iv) For all $k$, $\pi_{|k}$ has positive and continuous density with respect to Lebesgue measure at $\theta_{[k]}^o$. 
\end{itemize} 
Recall that 
\begin{align}\label{assumpsions:parametric}
 \inf_{\theta \in \Theta(k_1) } \| \theta_0 - \theta\| =  \delta \sqrt{ (\log n)/n}  , \quad 
\min_{k <k_1}  \inf_{\theta \in \Theta(k) } \| \theta_0 - \theta\| \geq  C \sqrt{ (\log n)/n}.
\end{align}

Under the above set of conditions for all $k$, if $\| \theta - \theta_{[k]}^o\| =o(1) $, 
 $$ \ell_n( \theta ) - \ell_n(\theta_{[k]}^o) = (\theta - \theta_{[k]}^o)^T  \nabla \ell_n(\theta_{[k]}^o) -  \frac{ n  (\theta - \theta_{[k]}^o)^T I_k(\theta_0) (\theta - \theta_{[k]}^o) (1 +o_{P_{\theta_0}}(1)) }{ 2  } $$ 
and  the BIC formula holds together with the local Bernstein-von Mises theorems:
 \begin{equation}\label{BIC}
 \begin{split}
 \log m_n(k) &= \ell_n(\hat \theta_k) - \ell_n(\theta_0) - \frac{ d_k \log n}{2 } + O_{P_{\theta_0}}(1) ,\\
 \pi_{|k}(\sqrt{n}(\theta - \hat \theta_k) \in A | \mathbf Y)& = P(Z_k \in A) + o_{P_{\theta_0}}(1),  
 \end{split}
 \end{equation} 
 where 
 $$ m_n(k) = \int_{\Theta(k) } e^{\ell_n(\theta) - \ell_n(\theta_0)} d\pi_{|k}(\theta) , \quad Z_k \sim \mathcal N(0, I_k(\theta_0)), $$
and $I_k(\theta_0) $ is the local (misspecified) Fisher information, i.e. the limit of $-\partial^2\ell_n(\theta_{[k]}^o)/n$ under $P_{\theta_0}$. Let  $k > k_1$ 
\begin{align*}
\frac{ \pi_k(k | \mathbf Y)  }{ \pi_k(k_1 | \mathbf Y) }& = \frac{ \pi_k(k) m_n(k) }{ \pi_k(k_1) m_n(k_1) } = \exp \left( \ell_n(\hat \theta_{k} )-  \ell_n(\hat \theta_{k_1}) - \frac{ (d_k - d_{k_1})\log n }{ 2 } + O_{P_{\theta_0}}(1) \right) \\ 
&  =  \exp \left( \ell_n(\theta_{[k]}^o )-  \ell_n( \theta_{[k_1]}^o) - \frac{ (d_k - d_{k_1})\log n }{ 2 } + O_{P_{\theta_0}}(1) \right) \\
& = \exp\Big(  - \frac{ (d_k - d_{k_1})\log n }{ 2 }   -(\theta_{[k_1]}^o - \theta_{[k]}^o)^T  \nabla \ell_n(\theta_{[k]}^o) \\
&\qquad  + \frac{ n  (\theta_{[k_1]}^o - \theta_{[k]}^o)^T I_k(\theta_0) (\theta_{[k_1]}^o - \theta_{[k]}^o) (1 +o_{P_{\theta_0}}(1)) }{ 2  }  + O_{P_{\theta_0}}(1) \Big) \\
& \leq \exp\Big( C_1 \delta^2 \log n - \frac{ (d_k - d_{k_1})\log n }{ 2 } + O_{P_{\theta_0}}(1) \Big),
\end{align*}
for some $C_1>0$ so that if $\delta$ is small enough 
 $$
 \frac{ \pi_k(k | \mathbf Y)  }{ \pi_k(k_1 | \mathbf Y) } \lesssim  n^{ - (d_k - d_{k_1}- 2C_1\delta^2 +O_{P_{\theta_0}}(1))/2 }.
 $$
 If $k < k_1$ using a similar argument we obtain that if $C$ is large enough in \eqref{assumpsions:parametric}, 
 $$
 \frac{ \pi_k(k | \mathbf Y)  }{ \pi_k(k_1 | \mathbf Y) } \lesssim  n^{ - \tau C^2 }, 
 $$
for some small enough $\tau>0$. It implies in particular that 
 $$ \pi_k(k_1|\mathbf Y) = 1 + o_{P_{\theta_0}}(1) , \quad \hat \theta = \hat \theta_{k_1} + o_{P_{\theta_0}}(1/\sqrt{n} ) ,$$
 and 
 $$ r_\alpha \asymp n^{-1/2}, $$ 
 so that 
 \begin{equation*} 
 P_{\theta_0}^{(n)}\left( \theta_0 \in \hat C(L_n,\alpha) \right) \leq P_{\theta_0}^{(n)}\left(\|\theta_0 - \hat \theta_{k_1}\|\leq 2 r_\alpha L_n \right) = o(1) 
 \end{equation*}  
 as soon as $L_n = o(\sqrt{\log n})$.

\end{document}